\documentclass[oneside,a4paper]{amsart}

\usepackage[T1]{fontenc}
\usepackage[utf8]{inputenc}
\usepackage[british]{babel}
\usepackage{amsmath, amsthm, amsfonts, amssymb}

\usepackage{url}
\usepackage[all]{xy}
\usepackage{placeins} 
\usepackage[usenames, dvipsnames]{color} 
\usepackage{euscript} 
\usepackage{bbm} 
\usepackage{enumitem} 
\usepackage{faktor} 

\usepackage[dvipsnames]{xcolor} 
\definecolor{dark-red}{rgb}{0.5,0.15,0.15}
\definecolor{dark-blue}{rgb}{0.15,0.15,0.6}
\definecolor{dark-green}{rgb}{0.15,0.6,0.15}
\usepackage{hyperref}
\hypersetup{
    colorlinks, linkcolor=NavyBlue,
    citecolor=dark-blue, urlcolor=dark-green
}
\usepackage[nameinlink,capitalise,noabbrev]{cleveref}

\usepackage{tikz}
\usepackage{tikz-cd}
\usetikzlibrary{arrows,backgrounds,
    decorations.pathreplacing,
    decorations.pathmorphing
}
\usetikzlibrary{matrix,arrows}

\mathchardef\mhyphen="2D 
\newcommand{\euscr}[1]{\EuScript{#1}} 
\newcommand{\ccat}{\euscr{C}} 
\newcommand{\dcat}{\euscr{D}} 

\newcommand{\presheaves}{\mathcal{P}}
\newcommand{\Fun}{\textnormal{Fun}} 
\newcommand{\Hom}{\textnormal{Hom}} 
\newcommand{\map}{\textnormal{map}} 
\newcommand{\inftycat}{$\infty$-category } 
\newcommand{\inftygroupoid}{$\infty$-groupoid} 
\newcommand{\inftyc}{\euscr{C}} 
\newcommand{\inftyd}{\euscr{D}} 
\newcommand{\spaces}{\euscr{S}} 
\newcommand{\abeliangroups}{\euscr{A}b} 
\newcommand{\largespaces}{\widehat{\euscr{S}}} 
\newcommand{\spheres}{\euscr{S}ph} 
\newcommand{\Aut}{\mathrm{Aut}} 
\newcommand{\Alg}{\mathrm{Alg}} 
\newcommand{\inftycatcat}{\euscr{C}at_{\infty}} 
\newcommand{\ltopoi}{\euscr{LT}op} 
\newcommand{\fun}{\mathrm{Fun}} 
\newcommand{\pspaces}{\euscr{S}_{*}} 
\newcommand{\sets}{\euscr{S}et} 
\newcommand{\B}{\mathrm{B}} 
\newcommand{\emspaces}{\euscr{EM}} 

\newcommand{\Mod}{\euscr{M}od} 

\theoremstyle{plain}

\newtheorem{theorem}{Theorem}[section]
\newtheorem{lemma}[theorem]{Lemma}
\newtheorem{prop}[theorem]{Proposition}
\newtheorem{cor}[theorem]{Corollary}

\newtheorem*{theorem*}{Theorem}

\theoremstyle{definition}
\newtheorem{example}[theorem]{Example}
\newtheorem{warning}[theorem]{Warning}
\newtheorem{notation}[theorem]{Notation}

\newtheorem{defin}[theorem]{Definition}
\newtheorem{rem}[theorem]{Remark}
\newtheorem{observation}[theorem]{Observation}
\newtheorem{construction}[theorem]{Construction}
\newtheorem{assumption}[theorem]{Assumption}
\newtheorem*{rem*}{Remark}
\newtheorem*{interpretation*}{Interpretation}
\newtheorem*{defin*}{Definition}
\newtheorem*{conjecture*}{Conjecture}
\newtheorem*{notation*}{Notation}
\newtheorem*{convention*}{Convention}
\newtheorem*{theorem_italics*}{Theorem}

\theoremstyle{remark}

\makeatletter
  \def\subsection{\@startsection{subsection}{1}%
  \z@{.7\linespacing\@plus\linespacing}{.5\linespacing}%
  {\normalfont\bfseries\centering}}
\makeatother

\calclayout

\begin{document}

\title[Moduli of spaces with prescribed homotopy groups]{Moduli of spaces with prescribed homotopy groups}
\author[Piotr Pstr\k{a}gowski]{Piotr Pstr\k{a}gowski}
\address{Department of Mathematics, Harvard University, 1 Oxford St, Cambridge, MA 02138}
\email{pstragowski.piotr@gmail.com}

\begin{abstract}
We describe a homotopy-theoretic approach to the theory of moduli of realizations of Blanc-Dwyer-Goerss, reproducing their obstructions to realizing a given $\Pi$-algebra as homotopy groups of a pointed space. Our techniques are based on the $\infty$-category $\presheaves_{\Sigma}(\spheres)$ of product-preserving presheaves on finite wedges of positive-dimensional spheres, leading to more conceptual and streamlined arguments. 
\end{abstract}

\maketitle 

\tableofcontents


\section{Introduction}
The main subject of the current work is a homotopy-theoretic approach to the theory of moduli of realizations of Blanc-Dwyer-Goerss \cite{realization_space_of_a_pi_algebra}, which yields a sequence of obstructions to realizing a $\Pi$-algebra $A$ as homotopy groups of a pointed space. 

By a homotopy-theoretic approach we mean that we work directly with the relevant $\infty$-categories and that everything we write down is manifestly homotopy-invariant. This greatly simplifies the exposition, as we use general arguments instead of calculations in specific model categories. The main technical tool is a theory of Postnikov invariants, which we develop in the generality of an arbitrary presheaf $\infty$-category using the Grothendieck construction, an approach which we believe is novel even in the classical case of spaces. 

The ideas first appearing in the work of Blanc-Dwyer-Goerss have been instrumental in the development of more sophisticated obstruction theories, in particular that of Goerss-Hopkins which studies ring spectra with prescribed homology \cite{moduli_spaces_of_commutative_ring_spectra, moduli_problems_for_structured_ring_spectra}. 

Similarly, the approach based on product-preserving presheaves introduced in this paper was followed by subsequent work of Paul VanKoughnett and the author which develops Goerss-Hopkins theory from the point of view of synthetic spectra \cite{pstrkagowski2018synthetic, pstrkagowski2022abstract}, as well as the work of Balderrama \cite{balderrama2021deformations}. As the case of $\Pi$-algebras and pointed spaces is quite concrete and elementary, we hope that the current work will be pedagogically useful also to readers ultimately interested in these more refined variants of obstruction theory. 

\subsection{Main results}

We now discuss the results of the paper in more detail.  Throughout the current work, by a space we mean an $\infty$-groupoid; these are also often called anima or homotopy types. 

Let $\pspaces$ denote the $\infty$-category of pointed spaces and $\spheres \hookrightarrow \pspaces$ the full subcategory spanned by finite wedges $\vee_{i \in I} S^{n_{i}}$ of positive-dimensional spheres. We have the following classical definition due to Stover: 

\begin{defin}[{\cite[\S 4.1]{stover1988van}}]
\label{defin:pi_algebra_definition_in_introduction}
A \emph{$\Pi$-algebra} is a functor $A \colon h \spheres ^{op} \rightarrow \sets$ that takes coproducts in the homotopy category of spheres into products of sets. 
\end{defin}
As any object in $h \spheres$ is a coproduct of $S^{i}$, a $\Pi$-algebra $A$ is determined by the sequence of sets
\[
A_{i} := A(S^{i})
\]
together with a specified family of maps between them. For example, the pinch map $S^{i} \rightarrow S^{i} \vee S^{i}$ gives each of $A_{i}$ a structure of a group, commutative if $i \geq 2$. One also has maps between $A_{i}$ given by precomposition with elements of the homotopy groups of spheres, as well as Whitehead products. As a consequence of the Hilton-Milnor theorem \cite{hilton1955homotopy}, a $\Pi$-algebra is uniquely defined by a sequence of groups together with these two classes of maps, satisfying appropriate identities. 

A $\Pi$-algebra is an algebraic construction which captures homotopy groups of a pointed space together with all of the primary operations between them. Indeed, the main example of such a structure is the $\Pi$-algebra of homotopy groups associated to a pointed space $X$, defined by
\[
(\pi_{*}X)(U) := \pi_{0} \map_{\spaces_{*}}(U, X).
\]
One calls $\Pi$-algebras of this form  \emph{realizable}. In general, it is a difficult question to determine whether a given $\Pi$-algebra is realizable or not \cite{blanc1995higher}.

\begin{defin}
If $A$ is a $\Pi$-algebra, then a pointed, connected space $X$ is a  \emph{realization} of $A$ if there exists an isomorphism $\pi_{*}X \simeq A$. We denote the full subgroupoid of $\spaces^{\cong}$ spanned by realizations by $\mathcal{M}(A)$ and call it the \emph{moduli of realizations}. 
\end{defin}
Our main result, which as explained above goes back to the work of Blanc-Dwyer-Goerss \cite{realization_space_of_a_pi_algebra}, gives a decomposition of the moduli $\infty$-groupoid $\mathcal{M}(A)$ in terms of cohomology of $\Pi$-algebras. 

To state our main result, we recall that an \emph{animated $\Pi$-algebra} is a product-preserving functor $h \spheres^{op} \rightarrow \spaces$. Together with natural transormations these assemble into an $\infty$-category $\Pi\mhyphen\Alg^{an}$, the animation of the category of $\Pi$-algebras\footnote{Following Cesnavicius and Scholze \cite[\S 5.1.4]{cesnavicius2019purity}, given a category $\inftyc$ generated under colimits by its subcategory $\inftyc^{cp}$ of compact projective objects, we call the free sifted cocompletion of $\inftyc^{cp}$ as an $\infty$-category the \emph{animation of $\inftyc$}, and denote it by $\inftyc^{an}$. If $\inftyc^{cp}$ is small, such as in the case of the category of $\Pi$-algebras, $\inftyc^{an} \simeq \presheaves_{\Sigma}(\inftyc^{cp})$, the $\infty$-category of product-preserving presheaves of spaces.}. This is the $\infty$-category whose mapping spaces are used to define Andr\'{e}-Quillen cohomology of $\Pi$-algebras. We write $A_{S^{n}}$ for the $A$-module defined by $A(U) := A(S^{n} \wedge U)$.

\begin{theorem}[\ref{example:moduli_of_0_stages}, \ref{prop:moduli_of_infinity_stages_as_a_limit}, \ref{thm:cartesian_square_of_moduli} or \cite{realization_space_of_a_pi_algebra}(1.1)]
\label{thm:main_theorem}
There exist $\infty$-groupoids $\mathcal{M}_{n}(A)$ and a tower 

\begin{center}
$\mathcal{M}(A) \rightarrow \ldots \rightarrow \mathcal{M}_{2}(A) \rightarrow \mathcal{M}_{1}(A) \rightarrow \mathcal{M}_{0}(A)$
\end{center}
such that
\begin{enumerate}
\item $\mathcal{M}_{0}(A) \simeq \B\Aut_{\Pi\mhyphen\Alg}(A)$, where $\Aut_{\Pi\mhyphen\Alg}(A)$ is the group of automorphisms of $A$,
\item for each $n \geq 1$ there exist a canonical cartesian square
\begin{center}
	\begin{tikzpicture}
		\node (TL) at (0, 1.5) {$ \mathcal{M}_{n}(A) $};
		\node (BL) at (0, 0) {$ \mathcal{M}_{n-1}(A) $};
		\node (TR) at (3.6, 1.5) {$ \B \Aut(A, A_{S^{n}}) $};
		\node (BR) at (3.6, 0) {$  \euscr{E}\mathrm{xt}(A, \B^{n+1}A_{S^{n}}) $};
		
		\draw [->] (TL) to (TR);
		\draw [->] (TL) to (BL);
		\draw [->] (TR) to (BR);
		\draw [->] (BL) to (BR);
	\end{tikzpicture},
\end{center}
where the bottom right vertex is the $\infty$-groupoid of extensions of $A$ by the module $\B^{n+1}A_{S^{n}}$ in animated $\Pi$-algebras, the top right vertex is the classifiying space of the group of automorphisms of the pair $(A, A_{S^{n}})$, and the right vertical map classifies the action on the split extension, 
\item $\mathcal{M}(A) \simeq \varprojlim \mathcal{M}_{n}(A)$. 
\end{enumerate}
\end{theorem}
As we explain in detail in \cref{rem:obstructions_to_realization}, \cref{thm:main_theorem} implies that there exist inductively defined cohomological obstructions 
\[
\theta_{n} \in \mathrm{H}^{n+2}_{A}(A, A_{S^{n}})
\]
to realizing $A$ as homotopy groups of a pointed space, where $n \geq 1$. In \cite{frankland_2_stage_postnikov_systems}, Frankland uses these methods to describe the moduli of realizations of $2$-stage $\Pi$-algebras.

\subsection{Overview of the approach}
\label{subsection:overview_of_the_approach}

Before describing our approach, let us briefly recall the key ideas appearing the work of Blanc, Dwyer and Goerss. The main technical tool in \cite{realization_space_of_a_pi_algebra} is the category of simplicial pointed topological spaces equipped with its $E_{2}$-model structure. 

In the latter model structure, first introduced in \cite{dwyer_kan_stover}, a map $X_{\bullet} \rightarrow Y_{\bullet}$ of simplicial pointed topological spaces is a weak equivalence if and only if the map $(\pi_{i} X)_{\bullet} \rightarrow (\pi_{i} Y)_{\bullet}$ is a weak equivalence of simplicial sets for all $i \geq 1$. Every levelwise weak equivalence of simplicial topological spaces has this property, but the class of $E_{2}$-weak equivalences is much more broad.

Cofibrant objects of the category of simplicial pointed topological spaces equipped with the $E_{2}$ structure can be thought of as ``nonabelian projective resolutions'' of a pointed topological space, namely the geometric realization. The main idea of the work of Blanc-Dwyer-Goerss is that instead of trying to construct a space realizing a given $\Pi$-algebra directly, one tries to construct its resolution. This can be then done inductively due to the additional simplicial direction.

There are two kinds of bigraded homotopy groups that one can associate to a simplicial pointed topological space, see \cite{bigraded_homotopy_groups},  one of them being the homotopy groups $\pi_{j} | (\pi_{i}X)_{\bullet}|$ of the geometric realization of the simplicial set $(\pi_{i} X)_{\bullet}$ for each $i \geq 1$.  Both kinds of bigraded homotopy groups are invariant under $E_{2}$-equivalences. In the work of Blanc-Dwyer-Goerss, computations are performed by playing the two kinds off each other using a certain long exact sequence, called the spiral exact sequence. 

Our approach begins with the observation, which the author first learned from Aaron Mazel-Gee, that the underlying $\infty$-category of the $E_{2}$-model structure can be described explicitly as the $\infty$-category $\presheaves_{\Sigma}(\spheres)$; that is, we have an equivalence of $\infty$-categories 
\[
s \euscr{T}\mathrm{op}[W_{E_{2}}^{-1}] \simeq \presheaves_{\Sigma}(\spheres),
\]
where the left hand side is the $\infty$-categorical localization of the category of simplicial topological spaces at the $E_{2}$-equivalences. 
Here, 
\[
\presheaves_{\Sigma}(\spheres) \subseteq \presheaves(\spheres) = \Fun(\spheres^{op}, \spaces)
\]
denotes the full subcategory spanned by those presheaves of spaces given which take finite coproducts in $\spheres$ to products of spaces. 

This motivates the question whether one can eschew the model category-theoretic techniques and develop the obstruction theory of Blanc-Dwyer-Goerss directly using the $\infty$-category $\presheaves_{\Sigma}(\spheres)$. To show that this is not only possible, but gives one a more clear picture of the theory, is the main subject of the current work. 

As an example of what we mean by a more clear picture, note that the other important model category appearing in \cite{realization_space_of_a_pi_algebra} is the category of simplicial $\Pi$-algebras equipped with its projective model structure. By a result of Bergner \cite{bergner_rigidification_of_algebras}, the underlying $\infty$-category of this model structure is equivalent to the $\infty$-category of animated $\Pi$-algebras $\presheaves_{\Sigma}(h \spheres)$; that is, product-preserving presheaves on the homotopy category $h \spheres$. One of the advantages of working directly with the relevant $\infty$-categories is that $\presheaves_{\Sigma}(\spheres)$ and $\presheaves_{\Sigma}(h \spheres)$ can now be treated in a uniform manner.

The projection $p \colon \spheres \rightarrow h \spheres$ onto the homotopy category induces an adjunction 
\[
p^{*} \colon \presheaves_{\Sigma}(\spheres) \rightleftarrows \presheaves_{\Sigma}(h \spheres) \colon p_{*},
\]
where $p_{*}$ is precomposition and $p^{*}$ is the left Kan extension. One can show that the two kinds of homotopy groups appearing in the work of Blanc-Dwyer-Goerss correspond to the homotopy groups of either $X$ or $p_{*} p^{*} X$, computed levelwise. This explains the two kinds of bigraded homotopy groups as an instance of the same phenomena. We also show that there is a natural fibre sequence involving $X$ and $p_{*} p^{*} X$, and that the long exact sequence of homotopy associated to that is the spiral sequence discussed above. 

To relate $\presheaves_{\Sigma}(\spheres)$ to spaces, we note that the restricted Yoneda embedding presents the $\infty$-category of pointed, connected spaces as a full subcategory of $\presheaves_{\Sigma}(\spheres)$. We show that the essential image of this embedding is the subcategory of those presheaves $X$ for which a certain naturally defined loop comparison map $X_{S^{1}} \rightarrow \Omega X$ is an equivalence. 

The relation of $\presheaves_{\Sigma}(\spheres)$ to $\Pi$-algebras is even more simple, as it is immediate from the definition that the category of the latter can be identified with the full subcategory spanned by discrete objects; that is, by presheaves valued in sets. Moreover, if $X \in \presheaves_{\Sigma}(\spheres)$ is represented by a pointed, connected space $Y$, then the homotopy $\Pi$-algebra $\pi_{*}Y$ can be described as the $0$-truncation $X_{\leq 0}$ in the $\infty$-category of presheaves. 

This suggests that if we view both as objects of $\presheaves_{\Sigma}(\spheres)$, then there is a whole sequence of objects interpolating between a pointed, connected space and its $\Pi$-algebra, namely the Postnikov tower in presheaves. We axiomatize what it means for a product-preserving presheaf to, informally, ``behave like an $n$-truncation of a Yoneda embedding of a pointed, connected space''. The conditions are simple: they merely state that $X$ has to take values in $n$-truncated spaces and that its loop comparison map is $(n-1)$-connective.  We call such presheaves \emph{potential $n$-stages}. 

The $\infty$-groupoids $\mathcal{M}_{n}(A)$ appearing in the statement of  \cref{thm:main_theorem} are given by full subgroupoids of $\presheaves_{\Sigma}(\spheres)$ spanned by potential $n$-stages and the connecting functors 
\[
\mathcal{M}_{n}(A) \rightarrow \mathcal{M}_{n-1}(A)
\]
are induced by Postnikov truncation. Using the Yoneda embedding, $\mathcal{M}_{\infty}(A)$ can be identified with the $\infty$-groupoid of realizations of $A$, while $\mathcal{M}_{0}(A)$ is the groupoid of $\Pi$-algebras isomorphic to $A$, which proves parts (1) and (3) of the main result. 

To understand the map $\mathcal{M}_{n}(A) \rightarrow \mathcal{M}_{n-1}(A)$ and its fibres, we develop a theory of Postnikov invariants in $\infty$-categories of the form $\presheaves_{\Sigma}(\inftyc)$, where $\inftyc$ is a small $\infty$-category satisfying some conditions, see \cref{assumption:assumption_on_c_to_have_good_presheaves}. These results form the heart of the current work; the main one being the classification of presheaves of $n$-types with a prescribed $(n-1)$-truncation and with a given module structure on homotopy groups. 

The $k$-invariants we describe can be identified with the ones appearing in the work of Blanc-Dwyer-Goerss; however, we obtain them in a different way. We begin with the fundamental observation that if $X$ is an $n$-type in spaces, then the map $X \rightarrow X_{\leq (n-1)}$ has fibres given by Eilenberg-MacLane spaces of degree $n$ and conversely any map with this property that has a source which is an $n$-type with $(n-1)$-truncation precisely $X_{\leq (n-1)}$. Working $\infty$-categorically, such maps can be classified by the Grothendieck construction, which associates to such a map a functor $X_{\leq (n-1)} \rightarrow \emspaces_{n}$ valued in the $\infty$-category of Eilenberg-MacLane spaces of degree $n$. 

We show how to adapt the Grothendieck construction to the setting of product-preserving presheaves and prove the needed classification by computing the homotopy types of the relevant classifying spaces. Note that in this approach objects tend to be brought into existence by their universal properties; their properties are only established down the line.  In particular, the technical difference construction appearing in \cite{realization_space_of_a_pi_algebra} is no longer needed, nor are the relevant connectivity calculations. 

A computationally significant step is that we show that the invariants of $n$-types discussed above can be identified not with homotopy classes inside $\presheaves_{\Sigma}(\spheres)$, but in the more simple $\infty$-category $\presheaves_{\Sigma}(h \spheres)$. These can be then identified with Andr\'{e}-Quillen cohomology of $\Pi$-algebras.

\subsection{Related work}

In \cite{lurie_hopkins_brauer_group}, Hopkins and Lurie apply techniques similar to ours to the problem of determining the Brauer group of Lubin-Tate spectra. In a similar direction, an account of Goerss-Hopkins theory based on product-preserving (pre)sheaves appears in the joint work of the VanKoughnett and the author \cite{pstrkagowski2018synthetic, pstrkagowski2022abstract}. In \cite{balderrama2021deformations}, Balderrama extends the methods of this work to cover a wide class of algebraic theories. 

\subsection{Notation and terminology}

We use the language of $\infty$-categories or quasicategories, as developed by Joyal and Lurie; standard references are  \cite{lurie_higher_topos_theory} and \cite{joyal2002quasi}. We identify ordinary categories with $\infty$-categories with discrete mapping spaces through the nerve construction. Usually, if an $\infty$-category is an ordinary category, we will be explicit about it.

An \emph{\inftygroupoid} is an $\infty$-category whose morphisms are all equivalences. Following Lurie, we usually refer to $\infty$-groupoids 
 as \emph{spaces} and denote their $\infty$-category by $\spaces$. If $\inftyc$ is an $\infty$-category, then we denote its \emph{underlying \inftygroupoid} by $\inftyc^{\cong}$. This is a subcategory with the same objects but only the equivalences. We denote the \emph{homotopy category} of $\inftyc$ by $h\inftyc$. 
 
We use the term \emph{moduli} somewhat informally to refer to an object which classifies structures of specified type. A common case is that of a \emph{moduli space} by which we mean the $\infty$-groupoid spanned by these objects and their equivalences. 

If $X$ is a space and $n \geq 0$, then we say that it is \emph{$n$-truncated} or an \emph{$n$-type} if all of its homotopy groups $\pi_{i}(X, x)$ at all basepoints $x \in X$ vanish for $i > n$. The inclusion $\tau _{\leq n} \spaces \hookrightarrow \spaces$ of the full subcategory spanned by $n$-types admits a left adjoint which we call $n$-truncation and denote by $(-)_{\leq n}$. 

More generally, if $\inftyc$ is any $\infty$-category and $c \in \inftyc$ is an object, then we say $c$ is an \emph{$n$-type} if for all $c^{\prime} \in \inftyc$ the space $\map(c^{\prime}, c)$ is an $n$-type. We say an object is \emph{discrete} if it is a $0$-type. If the inclusion $\tau _{\leq n}\inftyc \hookrightarrow \inftyc$ of the full subcategory spanned by $n$-types admits a left adjoint, then we also call it $n$-truncation and denote it by $(-) _{\leq n}$. 

We say that a map $X \rightarrow Y$ is an \emph{$n$-equivalence} whenever $X_{\leq n} \rightarrow Y_{\leq n}$ is an equivalence. We denote the subspace of $n$-equivalences by $\map_{n-\mathrm{eq}}(X, Y)$; it is a union of path components of $\map(X, Y)$. 

We use the term \emph{$\infty$-group} to refer to homotopy coherent-groups; that is, to group objects in the sense of \cite[7.2.1.1]{lurie_higher_topos_theory}. An \emph{$\infty$-cogroup} in an $\infty$-group in the opposite $\infty$-category. We reserve the word \emph{group} for classical, homotopically discrete groups. 

By the recognition principle of Boardman-Vogt and May \cite{boardman2006homotopy, may2006geometry}, the $\infty$-category $\mathrm{Grp}(\spaces)$ of $\infty$-groups in spaces is equivalent to the $\infty$-category $\spaces_{*}^{cn}$ of pointed, connected spaces. A pair of inverse equivalences is given by the loop space functor $\Omega \colon \spaces_{*}^{cn} \rightarrow \mathrm{Grp}(\spaces)$ and the classifying space functor $\B \colon \mathrm{Grp}(\spaces) \rightarrow \spaces_{*}^{cn}$. We use these equivalence freely; for a modern account see \cite[5.2.6.10]{higher_algebra}.

If $\inftyc$ is an $\infty$-category which admits finite coproducts, then we say a presheaf $X \colon \inftyc^{op} \rightarrow \spaces$ is \emph{product-preserving} or \emph{spherical}\footnote{We will mainly use the word ``spherical'' rather than ``product-preserving'' for two distinct reasons: 
\begin{enumerate}
\item brevity: \emph{product-preserving} is two words, but \emph{spherical} is only one, 
\item clarity: \emph{product-preserving} means to take products to products, which is a property of covariant functors. As to turn a contravariant functor into a covariant one is a matter of convention, this term is potentially ambiguous when applied to contravariant functors. 
\end{enumerate}
Our specific choice of the word \emph{spherical} is inspired by the fact that our applications are related to presheaves indexed by the $\infty$-category of finite wedges of spheres.} if it takes finite coproducts in $\inftyc$ to products of spaces. We denote the full subcategory of $\presheaves(\inftyc) := \Fun(\inftyc^{op}, \spaces)$ spanned by spherical presheaves by $\presheaves_{\Sigma}(\inftyc)$.
 
\subsection{Acknowledgments}

I would like to thank Neil Strickland, at whose suggestion I first became interested in moduli problems in algebraic topology. I would like to thank Aaron Mazel-Gee for explaining to me the importance of the non-abelian derived category. I would like to thank the MathOverflow homotopy theory chat, where I learned a lot of $\infty$-category theory. I would like to thank Paul Goerss for helpful discussions and for reading an earlier version of the current work. I would like to thank the anonymous referee for many helpful suggestions which greatly improved the quality of this paper. 

\section{Classification of spaces}
\label{section:classification_of_spaces}

In this preparatory chapter, we review the theory of Postnikov invariants of spaces. While this subject is classical, our approach is novel and based on the Grothendieck construction for $\infty$-categories. While somewhat technical, this approach has the advantage of generalizing well to the case of presheaves which we take up later in \S\ref{section:classification_of_presheaves}. 

\subsection{Maps between Eilenberg-MacLane spaces} 
\label{subsection:maps_between_eilenberg_maclane_spaces}

To identify Postnikov invariants with appropriate cohomology classes, we need some elementary results on maps between Eilenberg-MacLane spaces, which we prove in this section. Everything here is folklore, but we couldn't find a suitable reference. Throughout, we assume that $n \geq 1$. 

Associated to an abelian group $A$ we have the itarated classifying space $\B^{n}A$, which is canonically an $\mathbf{E}_{\infty}$-space of homotopy type $K(A, n)$. This construction is functorial, so that given two abelian groups $A_{1}, A_{2}$, we have an induced morphism
\begin{equation}
\label{equation:bn_on_morphisms}
\B^{n}(-)  \colon \Hom_{\abeliangroups}(A_{1}, A_{2}) \rightarrow \map_{\spaces_{*}}(\B^{n}A_{1}, \B^{n} A_{2})
\end{equation}
into the mapping space of base-point preserving maps. This is an equivalence, as $\pi_{0}$ of the target can be identified with 
\[
\mathrm{H}^{n}(\B^{n}A_{1}, A_{2}) \simeq \Hom_{\abeliangroups}(A_{1}, A_{2})
\]
by the universal coefficient theorem, and the other homotopy groups vanish. An explicit inverse to (\ref{equation:bn_on_morphisms}) is induced by taking the $n$-th homotopy group.

\begin{notation}
Throughout the current work, by an Eilenberg-MacLane space we mean a space of homotopy type $K(A, n)$ for some \emph{abelian} group $A$. The latter condition is automatic when $n \geq 2$. Usually, the degree is fixed and we write $\emspaces_{n}$ for the corresponding full subcategory of spaces. 
\end{notation}
Note that for any $n \geq 1$, $\emspaces_{n} \subseteq \spaces$ is closed under the cartesian product, so that we can consider the $\infty$-categories $\Alg_{\mathbf{E}_{\infty}}(\emspaces_{n})$ of $\mathbf{E}_{\infty}$-algebras and $(\emspaces_{n})_{*}$ of pointed objects (that is, $\mathbf{E}_{0}$-algebras). Note that since we work in the $\infty$-categorical setting, to give a pointed map between pointed objects is additional data, rather than a condition. 

\begin{prop}
\label{prop:projectivity_of_em_spaces}
In the commutative triangle 
\[
\begin{tikzcd}
	& \abeliangroups \\
	\Alg_{\mathbf{E}_{\infty}}(\emspaces_{n}) && (\emspaces_{n})_{*}
	\arrow["{\mathrm{B}^{n}}", from=1-2, to=2-3]
	\arrow[from=2-1, to=2-3]
	\arrow["{\mathrm{B}^{n}}"', from=1-2, to=2-1]
\end{tikzcd},
\]
where the horizontal arrow is the forgetful functor, each arrow is an equivalence of $\infty$-categories. 
\end{prop}

\begin{proof}
The left vertical arrow is an equivalence of $\infty$-categories with inverse given by $\Omega^{n}$, so it is enough to verify that the right vertical arrow is also an equivalence of $\infty$-categories. To see that it is essentially surjective, observe that if $X$ is a pointed Eilenberg-MacLane space of degree $n$, then by definition there exists an equivalence $X \simeq \mathrm{B}^{n}(\pi_{n}X)$. To see that it is fully faithful, note that (\ref{equation:bn_on_morphisms}) is an equivalence as a consequence of the universal coefficient theorem.
\end{proof}

\begin{rem}
Note that \cref{prop:projectivity_of_em_spaces} implies that for any pointed Eilenberg-MacLane space $X \in (\emspaces_{n})_{*}$, the $\infty$-categorical fibre 
\[
\Alg_{\mathbf{E}_{\infty}}(\emspaces_{n}) \times_{(\emspaces_{n})_{*}} \{ X \} 
\]
is a contractible $\infty$-groupoid. Informally, this can be interpreted as saying that any pointed Eilenberg-MacLane space admits a unique $\mathbf{E}_{\infty}$-structure, functorial in maps of Eilenberg-MacLane of the same degree.
 \end{rem}

It will be useful to also describe morphisms in $\emspaces_{n}$ itself, where we do not have chosen basepoints. To do so, observe that we can compose $\B^{n}(-)$ of (\ref{equation:bn_on_morphisms}) with the forgetful functor into (non-pointed) spaces. This composite is not an equivalence, but since $\B^{n}A_{2}$ is an $\mathbf{E}_{\infty}$-space, it acts on itself by left multiplication. This action determines a homomorphism of $\infty$-groups
\begin{equation}
\label{equation:left_multiplication_action}
m  \colon \B^{n}A_{2} \rightarrow \map_{\spaces}(\B^{n}A_{2}, \B^{n}A_{2}).
\end{equation}
The two maps of (\ref{equation:bn_on_morphisms}) and (\ref{equation:left_multiplication_action}) together determine a morphism of spaces
\[
\phi  \colon \Hom_{\abeliangroups}(A, B) \times \B^{n}A_{2} \rightarrow \map_{\spaces}(\B^{n}A_{1}, \B^{n} A_{2})
\]
as the composite 
\[\begin{tikzcd}[column sep=small]
	 {\Hom(A_{1}, A_{2}) \times \B^{n}A_{1} } & {} & {\map_{\spaces}(\B^{n}A_{1}, \B^{n}A_{2}) \times \map_{\spaces}(\B^{n}A_{2}, \B^{n}A_{2})} \\
	& {map_{\spaces}(\B^{n}A_{1}, \B^{n}A_{2})}
	\arrow["{\B^{n}(-) \times m}", from=1-1, to=1-3]
	\arrow["\phi"', from=1-1, to=2-2]
	\arrow["- \circ -", from=1-3, to=2-2],
\end{tikzcd}\]
where the right arrow is the composition of morphisms in the $\infty$-category of spaces. 

\begin{lemma}
\label{lemma:maps_between_eilenberg_maclane_spaces}
For any $n \geq 1$ and abelian groups $A_{1}, A_{2}$, the morphism 
\[
\phi  \colon \Hom_{\abeliangroups}(A, B) \times \B^{n}A_{2} \rightarrow \map_{\spaces}(\B^{n}A_{1}, \B^{n} A_{2})
\]
is an equivalence of spaces. 
\end{lemma}

\begin{proof}
By construction of $\phi$, we have a commutative diagram of spaces
\[
\begin{tikzcd}
	{\Hom_{\abeliangroups}(A, B) \times \B^{n}A_{2} } && {\map_{\spaces}(\B^{n}A_{1}, \B^{n} A_{2})} \\
	& {\B^{n}A_{2}}
	\arrow["{p_{2}}"', from=1-1, to=2-2]
	\arrow["ev", from=1-3, to=2-2]
	\arrow["\phi", from=1-1, to=1-3]
\end{tikzcd},
\]
where $p_{2}$ is the projection onto the second coordinate and $ev$ is evaluation at the basepoint of $\B^{n}A_{1}$. 

Since $\B^{n}A_{2}$ is connected, to verify that $\phi$ is an equivalence it is enough to verify that the induced map of fibres of $p_{2}$ and $ev$ over the basepoint of $\B^{n}A_{2}$ is an equivalence. However, this map between the fibres can be identified with (\ref{equation:bn_on_morphisms}), which we've already seen is an equivalence.
\end{proof}

\begin{observation}
\label{observation:phi_restricts_to_equivalences}
By construction, $\phi$ restricts to an equivalence
\[
\Hom^{\simeq}_{\abeliangroups}(A, B) \times \B^{n}A_{2} \rightarrow \map^{\simeq}_{\spaces}(\B^{n}A_{1}, \B^{n} A_{2}),
\]
where the superscript $\simeq$ denotes the subspaces of equivalences. Note that both sides are empty unless $A$ and $B$ are isomorphic, as expected.
\end{observation}

\begin{observation}
\label{observation:self_equivalences_of_an_em_space_as_a_group}
In the particular case of $A_{1} = A_{2} = A$, \cref{observation:phi_restricts_to_equivalences} yields an equivalence of spaces $\Aut_{\abeliangroups}(A) \times \B^{n}A \simeq \Aut_{\spaces}(\B^{n}A)$. Retracing the definition of $\phi$, we see that it is in fact an equivalence \emph{of $\infty$-groups}
\[
\Aut_{\abeliangroups}(A) \ltimes \B^{n}A \simeq \Aut_{\spaces}(\B^{n}A),
\]
where the left hand side is the semidirect product induced by the action of $\Aut_{\abeliangroups}(A)$ on $\B^{n}A$. 
\end{observation}

\subsection{The Grothendieck construction and Postnikov invariants of spaces}
\label{subsection:postnikov_invariants_of_spaces}

In this section we review the theory of Postnikov invariants of spaces. This is classical and our main goal is to develop the necessary language which generalizes well to the case of presheaves. 

Recall that the inclusion $\tau_{\leq n} \spaces \hookrightarrow \spaces$ of the full subcategory spanned by $n$-types into the $\infty$-category of spaces admits a left adjoint, which we denote by $(-)_{\leq n}$. Concretely, if $X$ is a space, then $X_{\leq n}$ can be constructed by attaching, at each path component of $X$ separately, cells in dimensions $(n+1)$ and higher to kill all homotopy groups above dimension $n$. 

Using the unit maps of the adjunctions $\spaces \rightleftarrows \tau_{\leq n} \spaces$ for varying $n$, to each space $X$ one associates its Postnikov tower
\begin{center}
$\ldots \rightarrow X_{\leq 3} \rightarrow X_{\leq 2} \rightarrow X_{\leq 1} \rightarrow X_{\leq 0}$.
\end{center}
This tower has the properties that each $X_{\leq n}$ is an $n$-type and that $X_{\leq m} \rightarrow X_{\leq n}$ for $m \geq n$ is an $n$-equivalence, that is, it is an equivalence after applying the functor $(-)_{\leq n}$.

\begin{rem}
One can show that the obvious functor from spaces
\[
\spaces \rightarrow \mathrm{Post}(\spaces)
\]
into the $\infty$-category of towers satisfying the two conditions above, is an equivalence of $\infty$-categories \cite[7.2.1.10]{lurie_higher_topos_theory}. The inverse equivalence is given by forming the limit.
\end{rem}

The Postnikov tower presents each space as being constructed in steps, each step taken one level up the tower. To describe the necessary data to take such a step, that is, to describe $X_{\leq n}$ given $X_{\leq (n-1)}$, is the classical subject of Postnikov invariants. Our approach will be based on the Grothendieck construction, so we begin by reviewing the latter. 

If $p \colon Y \rightarrow X$ is a map of spaces, then it is a classical result that we have an associated functor from the fundamental groupoid of $X$ into the homotopy category of spaces. The functor takes a point $x \in X$ to the fibre of $p$ over $x$, and functoriality is provided by lifting of paths in $X$.

Using the langauge of $\infty$-categories, one can be more precise and show that from such a map one can construct an honest functor $\widetilde{p} \colon X \rightarrow \spaces$ of $\infty$-categories, the classical construction being recovered by passing to homotopy categories. Conversely, the colimit of any such $\widetilde{p}$ is canonically a space over $X$, since the colimit of a terminal functor is $X$ itself. A fundamental result is that these two constructions are inverse to each other:

\begin{theorem}[Lurie]
\label{thm:grothendieck_construction_for_spaces}
The functor $\varinjlim \colon \fun(X, \spaces) \rightarrow \spaces _{/X}$ is an equivalence of $\infty$-categories. 
\end{theorem}

\begin{proof}
This is \cite[2.2.1.2]{lurie_higher_topos_theory} in the special case when the given equivalence of simplicial categories is the identity and when the simplicial set in question is a Kan complex.
\end{proof}

\begin{rem}
The general form of the Grothendieck construction gives for any small $\infty$-category an equivalence $\fun(\inftyc, \inftycatcat) \simeq \mathrm{CoCart}(\inftyc)$ between functors into the \inftycat $\inftycatcat$ of small $\infty$-categories and cocartesian fibrations over $\inftyc$, see \cite[\S 3.1]{lurie_higher_topos_theory}. We will not need this construction in the current work. 
\end{rem}

\begin{observation}[Naturality of the Grothendieck construction]
\label{observation:naturality_of_the_grothendieck_construction}
By \cite[Appendix A]{gepner2015lax}, the equivalence of \cref{thm:grothendieck_construction_for_spaces} commutes with base change in the following sense: given a map of spaces $f \colon X^{\prime} \rightarrow X$ and a functor $p \colon X \rightarrow \spaces$, the corresponding diagram of spaces 
\[
\begin{tikzcd}
	{\varinjlim (p \circ f)} & {\varinjlim p} \\
	{X'} & X
	\arrow[from=2-1, to=2-2]
	\arrow[from=1-2, to=2-2]
	\arrow[from=1-1, to=2-1]
	\arrow[from=1-1, to=1-2]
\end{tikzcd}\]
is cartesian. In particular, in the special case of an inclusion of a point $\{ x \} \hookrightarrow X$, we recover the description of the functor $\widetilde{p} \colon X \rightarrow \spaces$ associated to a map $p \colon Y \rightarrow X$ as informally given by 
\[
\widetilde{p}(x) := Y \times_{X} \{ x \},
\]
the fibre over $x$. 
\end{observation}

\begin{rem}
The Grothendieck construction can be informally interpreted as saying that any map $Y \rightarrow X$ of spaces is a ``bundle'' in the sense that the fibres are functorial in the base. 
\end{rem}

\begin{observation}[Homotopy orbits of an action of an $\infty$-group]
\label{observation:homotopy_quotient_by_an_infty_group}
Associated to an $\infty$-group object $G$ in spaces we have a pointed space $\B G$, its classifying space, which we can consider as an $\infty$-groupoid. A functor $\B G \rightarrow \spaces$ can be identified with a choice of a space $X$ (the image of the basepoint) together with a homotopy coherent action of $G$. 

In this case, Grothendieck construction of \cref{thm:grothendieck_construction_for_spaces} can be interpreted as saying that the homotopy orbits construction
\[
X \mapsto X_{h G} := \varinjlim_{\B G} X
\]
provides an equivalence $\Fun(\B G, \spaces) \simeq \spaces_{/ \B G}$ between the $\infty$-category of spaces with a $G$-action and spaces over $\B G \simeq (\mathrm{pt})_{h G}$. Using \cref{observation:naturality_of_the_grothendieck_construction}, we see that associated to any $X$ with a $G$-action we have a fibre sequence of spaces of the form 
\[
X \rightarrow X_{hG} \rightarrow \B G. 
\]
In turn, any fibre sequence of the form 
\[
Y \rightarrow Z \rightarrow \B G 
\]
determines an action of $G$ on $Y$ and an equivalence $Z \simeq Y_{h G}$. 
\end{observation}

Let us move to the subject of Postnikov invariants. Let $X$ be a space, assumed to be connected, so that we have a (non-canonical) equivalence $X_{\leq 1} \simeq \B(\pi_{1}X)$, which we can interpret as a starting point of an iterative construction of $X$.

If $n > 1$, we can consider the map $X_{\leq n} \rightarrow X_{\leq (n-1)}$. It follows immediately from the long exact sequence of homotopy groups that all of the fibres are necessarily Eilenberg-MacLane spaces of type $K(\pi_{n}X, n)$. The latter can be naturally assembled into an $\infty$-groupoid.

\begin{notation}
\label{notation:moduli_of_em_spaces_of_speciifed_type}
We write
\[
\mathcal{M}(K(\pi_{n}X, n))
\]
for the full subcategory of $\spaces^{\cong}$ spanned by spaces of homotopy type $K(\pi_{n}X, n)$. We refer to this $\infty$-groupoid as the \emph{moduli of spaces of type $K(\pi_{n}, n)$}.
\end{notation}
As observed above, the fibres of $X_{\leq n} \rightarrow X_{\leq (n-1)}$ are objects of $\mathcal{M}(K(\pi_{n}X, n))$. It follows that the classifying functor $X_{\leq n-1} \rightarrow \spaces$ has a unique factorization
\[\begin{tikzcd}
	{X_{\leq (n-1)}} && \spaces \\
	& \mathcal{M}(K(\pi_{n}X, n))
	\arrow["{h_{n}}", from=1-1, to=2-2]
	\arrow[from=1-1, to=1-3]
	\arrow[hookrightarrow, from=2-2, to=1-3]
\end{tikzcd}.
\]
The functor $h_{n} \colon X_{\leq (n-1)} \rightarrow \mathcal{M}(K(\pi_{n}X, n))$ can be thought of as a variation on the Postnikov invariant, as it determines $X_{\leq n}$ as a space over $X_{\leq (n-1)}$. 

The moduli space of \cref{notation:moduli_of_em_spaces_of_speciifed_type} can be described explicitly. Indeed, notice that by definition all objects of $\mathcal{M}(K(\pi_{n}X, n))$ are equivalent to the iterated classifying space $\B^{n}(\pi_{n}X)$. Thus, the inclusion of the full subcategory spanned by $\B^{n}(\pi_{n}X)$, which we can identify with 
\[
\B\Aut_{\spaces}(\B^{n}(\pi_{n}X)) \hookrightarrow \mathcal{M}(K(\pi_{n}X, n)),
\]
is an equivalence of spaces. This, together with the Grothendieck construction, implies the following: 

\begin{prop}
\label{prop:classification_of_n_types_over_a_given_nminusone_type}
Given a map 
\[
h_{n}  \colon X_{\leq (n-1)} \rightarrow \B\Aut_{\spaces}(\B^{n}\pi_{n}X),
\]
taking the colimit of the composite 
\[
X_{\leq n-1} \rightarrow \B\Aut_{\spaces}(\B^{n}\pi_{n}X) \hookrightarrow \mathcal{M}(K(\pi_{n}X, n)) \hookrightarrow \spaces
\]
provides an equivalence
\[
\Fun(X_{\leq n-1}, \B\Aut_{\spaces}(\B^{n}\pi_{n}X)) \simeq (\spaces_{/X_{\leq n-1}}^{K(\pi_{n}X, n)})^{\cong}
\]
between the functor $\infty$-category and the $\infty$-groupoid of spaces over $X_{\leq n-1}$ with fibres of type $K(\pi_{n}X, n)$. 
\end{prop}

\begin{cor}
The homotopy type of $X_{\leq n}$ as a space over $X_{\leq n-1}$ is uniquely determined by a homotopy class of maps $h_{n}  \colon X_{\leq (n-1)} \rightarrow \B\Aut_{\spaces}(\B^{n}\pi_{n}X)$.  
\end{cor}

\begin{rem}
Note that we have determined the target of the map in  \cref{prop:classification_of_n_types_over_a_given_nminusone_type} explicitly in 
\cref{observation:self_equivalences_of_an_em_space_as_a_group}, as we have an equivalence
\[
\B\Aut_{\spaces}(\B^{n}(\pi_{n}X)) \simeq \B(\Aut_{\abeliangroups}(\pi_{n}X) \ltimes \B^{n}(\pi_{n}X)).
\]
\end{rem}
The map $h_{n}$ of  \cref{prop:classification_of_n_types_over_a_given_nminusone_type} can be thought of as a $k$-invariant which classifies $X_{\leq n}$ as a space over $X_{\leq n-1}$. 

In practice, one is more interested in classifying spaces $n$-truncated spaces on their own; that is, without the additional data of a fixed equivalence between their $(n-1)$-truncation and a fixed space. We now show how to deduce such a classification from \cref{prop:classification_of_n_types_over_a_given_nminusone_type}. To organize ideas, we make the following definition. 

\begin{defin}
\label{defin:space_of_type}
Let us fix $n \geq 2$, a discrete abelian group $A$, and a connected $(n-1)$-type $Y$. We say that an that an $n$-truncated space $X \in \spaces$ is of \emph{type} $Y + (A, n)$ if 
\begin{enumerate}
    \item there exists an equivalence $X_{\leq (n-1)} \simeq Y$ of spaces and 
    \item an isomorphism $\pi_{n}X \simeq A$ of abelian groups.
\end{enumerate} 
The \emph{moduli space} $\mathcal{M}(Y + (A, n))$ is the full subgroupoid of $\spaces^{\cong}$ spanned by spaces of type $Y + (A, n)$.
\end{defin}

By definition, the path components of $\mathcal{M}(Y + (A, n))$ correspond to homotopy types of spaces of type $Y + (A, n)$. However, even when one is only interested in path components, it is is often easier to describe the whole moduli space itself, as it is tends to be better behaved. 

The following gives an alternative description of $\mathcal{M}(Y + (A, n))$ based on the Grothendieck construction:

\begin{lemma}
\label{lem:postnikov_determination}
Let $\mathcal{M}(Y + (A, n) \looparrowright Y)$ denote be the full subgroupoid  of $\Fun(\Delta^{1}, \spaces^{\cong})$ spanned by those arrows $p  \colon X \rightarrow Y^{\prime}$ such that 
\begin{enumerate}
    \item $p$ is an $(n-1)$-equivalence,
    \item $X$ is of type $Y + (A, n)$ and 
    \item $Y^{\prime}$ equivalent to $Y$
\end{enumerate}
Then, the forgetful mapping
\[
\mathcal{M}(Y + (A, n) \looparrowright Y) \rightarrow \mathcal{M}(Y + (A, n))
\]
given by 
\[
(p  \colon X \rightarrow Y^{\prime}) \mapsto X
\]
is an equivalence of spaces. 
\end{lemma}

\begin{proof}
It is enough to show that all of the fibres are contractible, so let $X \in \mathcal{M}(Y + (A, n))$. The fibre over $X$ is the $\infty$-groupoid of $(n-1)$-types $Y^{\prime}$ equipped with an $(n-1)$-equivalence $X \rightarrow Y^{\prime}$. Using the adjunction $\spaces \rightleftarrows \tau_{\leq n-1} \spaces$ , we see that this is equivalent to the $\infty$-groupoid of $(n-1)$-types $Y^{\prime}$ together with a fixed equivalence $X _{\leq (n-1)} \simeq Y^{\prime}$, which is contractible.
\end{proof}

The above result is useful, as $\mathcal{M}(Y + (A, n) \looparrowright Y)$ participates in interesting fibre sequences. Above, we've shown that the map induced by forgetting the codomain is an equivalence. However, we can instead forget the domain, which yields a fibre sequence 
\begin{equation}
\label{equation:fibre_describing_y_plus_an}
\mathcal{M}((Y+(A,n)_{-/Y}) \rightarrow \mathcal{M}(Y + (A, n) \looparrowright Y) \rightarrow \mathcal{M}(Y),
\end{equation}
where $\mathcal{M}(Y)$ is the $\infty$-groupoid of spaces equivalent to $Y$ with basepoint given by $Y$ itself, and $\mathcal{M}((Y+(A,n)_{-/Y})$ is the $\infty$-groupoid of spaces $X$ of type $Y + (A, n)$ equipped with an $(n-1)$-equivalence $X \rightarrow Y$. By a combination of \cref{lem:postnikov_determination}, \cref{prop:classification_of_n_types_over_a_given_nminusone_type} and \cref{observation:self_equivalences_of_an_em_space_as_a_group}, we have equivalences
\[
\mathcal{M}((Y+(A,n)_{-/Y}) \simeq \map_{\spaces}(Y, \B\Aut(\B^{n}A)) \simeq \map_{\spaces}(Y, \B(\Aut_{\abeliangroups}(A) \ltimes \B^{n}A)).
\]
Combining this with (\ref{equation:fibre_describing_y_plus_an}) we obtain a fibre sequence
\begin{equation}
\label{equation:main_result_fibre_sequence_in_classification_of_n_types_in_spaces}
\map_{\spaces}(Y, \B(\Aut_{\abeliangroups}(A) \ltimes \B^{n}A)) \rightarrow \mathcal{M}(Y + (A, n)) \rightarrow \mathcal{M}(Y). 
\end{equation}
Notice that $\mathcal{M}(Y) \simeq \B(\Aut_{\spaces}(Y))$ and so by passing to path components with the above fibre sequence we get the following result:

\begin{prop}
\label{prop:postnikov_invariant_determining_homotopy_type_of_y_plus_an}
The homotopy type of a space of type $Y + (A, n)$ is uniquely determined by a homotopy class of maps
\[
k_{n}  \colon Y \rightarrow \B(\Aut_{\abeliangroups}(A) \ltimes \B^{n}A)
\]
well-defined up to the action of $\pi_{0}\Aut_{\spaces}(Y)$. 
\end{prop}

\begin{proof}
As explained in \cref{observation:homotopy_quotient_by_an_infty_group}, the data of a fibre sequence (\ref{equation:main_result_fibre_sequence_in_classification_of_n_types_in_spaces}) is equivalent to an equivalence 
\begin{equation}
\label{equation:moduli_space_of_spaces_of_type_yplusan_as_homotopy_orbits}
\mathcal{M}(Y + (A, n)) \simeq \map_{\spaces}(Y, \B(\Aut_{\abeliangroups}(A) \ltimes \B^{n}A))_{h \\Aut_{\spaces}(Y)},
\end{equation}
where the right hand side is given by homotopy orbits. As $\pi_{0} \colon \spaces \rightarrow \euscr{S}\mathrm{et}$ is a left adjoint, it commutes with colimits and we deduce that 
\[
\pi_{0} \mathcal{M}(Y + (A, n)) \simeq \faktor{\pi_{0} \map_{\spaces}(Y, \B(\Aut_{\abeliangroups}(A) \ltimes \B^{n}A))}{\pi_{0} \Aut_{\spaces}(Y)},
\]
where the right hand side is quotient of a set by a group action. 
\end{proof}

\begin{rem}[The moduli space as homotopy orbits]
The equivalence (\ref{equation:moduli_space_of_spaces_of_type_yplusan_as_homotopy_orbits}) recovers a classical result due to Dwyer, Kan and Smith; see \cite{dwyer1989towers}.
\end{rem}

\begin{rem}
\label{rem:postnikov_invariant_as_a_cohomology_class}
This homotopy class of maps of \cref{prop:postnikov_invariant_determining_homotopy_type_of_y_plus_an} is also a variation on Postnikov invariants. To recover the latter in their classical form (as explained, for example, in \cite[\S 4]{hatcher_algebraic_topology}) from the given homotopy class of maps, notice that when $Y$ is simply-connected, the map $k_{n}$ uniquely lifts along the universal cover of the target, which can be identified with 
\[
\B^{n+1} A \simeq \B(\B^{n} A) \rightarrow \B(\Aut_{\abeliangroups}(A) \ltimes \B^{n}A)
\]
To see this, observe that the source is simply-connected, and that the map is an isomorphism on higher homotopy groups as can be verified after applying $\Omega$. The space $\B^{n+1} A$ is an Eilenberg-MacLane space of type $K(A, n+1)$, so that the unique lift of $k_{n}$ can be identified with a cohomology class in $\mathrm{H}^{n+1}(Y, A)$.
\end{rem}

\subsection{Action of the fundamental group} 
\label{subsection:action_of_the_fundamental_group}

In this section, we refine the classification of $n$-types of \cref{prop:postnikov_invariant_determining_homotopy_type_of_y_plus_an} to yields results about spaces with fundamental group isomorphic to a fixed group $G$ and with higher homotopy groups prescribed \emph{as $G$-modules}. 

\begin{defin}
\label{definition:module_over_a_space}
If $X$ is a space, then a \emph{module} over $X$ is a discrete abelian group object in the $\infty$-category $\spaces_{/X}$ of spaces over $X$.
\end{defin}

\begin{rem}
A space $Y \rightarrow X$ is discrete over $X$ if it has discrete fibres. In this context, an abelian group structure is determined by a map 
\[
Y \times_{X} Y \rightarrow Y 
\]
which is commutative, associative and unital. We usually abusively denote a module as $Y \rightarrow X$, with the abelian group structure being implicit. 
\end{rem}

\begin{observation}[Modules and local coefficient systems]
\label{observation:modules_over_space_modules_over_one_truncation}
By the Grothendieck construction, a module $Y \rightarrow X$ is uniquely determined by a functor $X \rightarrow \abeliangroups$ into the category of abelian groups. Because $\abeliangroups$ is an ordinary category, any such functor factors uniquely through $X_{\leq 1}$. In particular, the canonical map $X \rightarrow X_{\leq 1}$ induces an equivalence between $X$-modules and $X_{\leq 1}$-modules. 

Note that since $X_{\leq 1}$ can be identified with the fundamental groupoid of $X$, modules in the sense of \cref{definition:module_over_a_space} are equivalent to the classical notion of a local coefficient system. We work with modules rather than local coefficient systems as they generalize more easily to the case of presheaves. 
\end{observation}

\begin{rem}
If $X$ is connected with a chosen basepoint $x \in X$, then we have a canonical equivalence $X_{\leq 1} \simeq \B(\pi_{1}(X, x))$. Thus, in this case the data of module over $X$ is the same as that of module over the group ring $\mathbb{Z}[\pi_{1}(X, x)]$, motivating our terminology.

Note that in the particular case when $X \simeq \B G$ for $G$ a group, a $\B G$-module in the sense of \cref{definition:module_over_a_space} is the same data as that of a $G$-module; that is, of an abelian group with a $G$-action. 
\end{rem}

The standard way to produce local coefficient systems is to begin with a map $X' \rightarrow X$ of spaces, which by the Grothendieck construction can be thought of as encoding a functor $X \rightarrow \spaces$. One can then compose the latter with a functor from spaces into abelian groups to obtain a module. 

Recall that a space $A$ is said to be simple if it is connected and there exists (equivalently, for any) a basepoint $a \in A$ such that $\pi_{1}(A, a)$ acts trivially on $\pi_{n}(A, a)$ for $n \geq 1$. These are exactly the spaces for which homotopy groups are canonically defined without choosing a basepoint\footnote{Indeed, suppose that $A$ is simple and let $a, a' \in A$ be two points. A choice of a path between these points determines an isomorphism $\pi_{n}(A, a) \simeq \pi_{n}(A, a')$, and the isomorphisms determines by two different paths differ by the action of $\pi_{1}(A, a)$ on $\pi_{n}(A, a)$. As this action is trivial by assumption, this isomorphism is canonical.}.

\begin{defin}
\label{defin:relative_homotopy_groups}
Suppose that $p \colon Y \rightarrow X$ be a map of spaces all of whose fibres are simple. For any $n \geq 2$, the \emph{relative homotopy groups} $\pi_{n}(Y \rightarrow X)$ is the module over $X$ corresponding to the local coefficient system 
\[
(x \in X) \mapsto \pi_{n}(F_{x}),
\]
where $F_{x} := Y \times_{X} \{ x \}$. 
\end{defin}
Note that $\pi_{n}(F_{x})$ appearing in \cref{defin:relative_homotopy_groups} is well-defined without a choice of a basepoint by the simplicity assumption. 

\begin{example}[Absolute homotopy groups as relative homotopy groups]
\label{exmp:absolute_homotopy_groups}
If $X$ is a space, then for any $n \geq 2$, the underlying abelian group of the $X_{\leq 1}$-module $\pi_{n}(X \rightarrow X_{\leq 1})$ can be identified with $\pi_{n} X$. These modules encode the action of the fundamental group on the higher homotopy groups.
\end{example}

In this section, we refine the results of \S\ref{subsection:postnikov_invariants_of_spaces} by considering spaces with prescribed homotopy groups together with the action of the fundamental group. In more detail, our is to be describe the $\infty$-groupoid of spaces of the following form: 

\begin{defin}
\label{definition:moduli_space_of_spaces_of_type_y_gplus_mn}
Suppose that $n \geq 2$, $G$ is a group, $M$ is a $G$-module, and $Y$ is a connected $(n-1)$-type such that there exists an equivalence $Y_{\leq 1} \simeq \B G$. We say that an $n$-type $X$ is \emph{of type} $Y +_{G} (M ,n)$ if 
\begin{enumerate}
    \item there exists an $(n-1)$-equivalence $X \rightarrow Y$ with the property that
    \item there exists an equivalence $\phi \colon \B G \simeq Y_{\leq 1}$ such that the $G$-modules $M$ and $\phi^{*} \pi_{n}(X \rightarrow Y)$ are isomorphic. 
\end{enumerate}
The moduli space
\[
\mathcal{M}(Y +_{G} (M, n))
\]
is the full subgroupoid of $\spaces^{\cong}$ spanned by spaces of type $Y +_{G} (M, n)$.
\end{defin}
Our main result of this section, see \cref{thm:classification_of_n_types_fixed_module} below, is an explicit description of the moduli space of \cref{definition:moduli_space_of_spaces_of_type_y_gplus_mn} in terms of $Y, G$ and $M$. To prove it, we need an appropriate notion of an Eilenberg-MacLane space, given by the following: 

\begin{defin}
\label{defin:eilenberg_maclane_space}
Let $G$ be a group, $M$ a $G$-module and $n \geq 1$. The \emph{Eilenberg-MacLane space} of type $(G, M, n+1)$, denoted $\B_{G}(M, n+1)$, is defined using the cartesian diagram of spaces

\begin{center}
	\begin{tikzpicture}
		\node (TL) at (0, 1.5) {$ \B_{G}(M, n+1)$};
		\node (BL) at (0, 0) {$ \B G  $};
		\node (TR) at (3, 1.5) {$ \emspaces_{n}^{\cong} $};
		\node (BR) at (3, 0) {$ \abeliangroups ^{\cong} $};
		
		\draw [->] (TL) to (TR);
		\draw [->] (TL) to (BL);
		\draw [->] (TR) to (BR);
		\draw [->] (BL) to (BR);
	\end{tikzpicture},
\end{center}
where on the right we have, respectively, the $\infty$-groupoid $\emspaces_{n}^{\cong}$ of abelian Eilenberg-MacLane spaces of degree $n$ and the groupoid of discrete abelian groups. Here, the right vertical map is given by taking $\pi_{n}(-)$\footnote{These homotopy groups are well-defined without a choice of a basepoint as abelian Eilenberg-MacLane spaces are simple.} and the bottom horizontal map is the classifying map for the $G$-action on $M$. 
\end{defin}

\begin{observation}
\label{observation:canonical_basepoint_of_twisted_em_space}
The canonical map $\B_{G}(M, n+1) \rightarrow \B G$ has a canonical section induced from the functor $\B^{n}(-) \colon \abeliangroups^{\cong} \rightarrow \emspaces_{n}^{\cong}$. The image of the canonical basepoint of $\B G$ under this section equips $\B _{G}(M, n+1)$ with a canonical basepoint of its own. 
\end{observation}

\begin{observation}[Universal property]
\label{observation:moduli_interpretation_of_eilenberg_maclane_spaces}
Notice that by definition, for any space $Y$ we have an equivalence

\begin{center}
$\map(Y, \B_{G}(M, n+1)) \simeq \map(Y, \B G) \times _{\map(Y, \abeliangroups ^{\cong})} \map(Y, \emspaces_{n}^{\cong})$.
\end{center}
Unwrapping what this means, we see that $\map(Y, B_{G}(M, n+1))$ is $\infty$-groupoid of the following triples:
\begin{enumerate}
\item a map $X \rightarrow Y$ levelwise fibred in Eilenberg-MacLane spaces of degree $n$, 
\item a map $\alpha  \colon Y \rightarrow \B G$ and 
\item an isomorphism $\epsilon  \colon \pi_{n}X \simeq \alpha^{*}M$ of modules over $Y$.
\end{enumerate}
Note that $\pi_{n}X$ is well-defined as a module as by assumptions all fibres are simple. This gives a moduli-theoretical interpretation of $\B_{G}(M, n+1)$.
\end{observation}

\begin{rem}
\label{rem:homotopy_of_an_eilenberg_maclane_space}
Using \cref{observation:self_equivalences_of_an_em_space_as_a_group}, we see that 
\begin{enumerate}
    \item $\B_{G}(M, n+1)$ is connected,
    \item $\pi_{1}\B_{G}(M, n+1) \simeq G$, 
    \item $\pi_{n+1}\B_{G}(M, n+1) \simeq M$ with the given $G$-module structure, as the classifying map for the $G$-action factors through $\B G \rightarrow \B \Aut_{\abeliangroups}(M)$, and we have
    \[
    \emspaces_{n}^{\cong} \times_{\abeliangroups^{\cong}} \B\Aut_{\abeliangroups}(M) \simeq \B \Aut_{\spaces}(\B^{n}M) \simeq \B(\Aut_{\abeliangroups}(M) \ltimes \B^{n}M)
    \]
    \item other homotopy groups vanish,
\end{enumerate}
where all of the homotopy groups are taken with respect to the basepoint of \cref{observation:canonical_basepoint_of_twisted_em_space}. This calculation motivates viewing $\B_{G}(M, n+1)$ as a variation on the notion of an Eilenberg-MacLane space and the corresponding notation, which is taken from \cite{realization_space_of_a_pi_algebra}. Beware that not all spaces satisfying the above four properties are equivalent to $\B_{G}(M, n+1)$. 
\end{rem}

\begin{example}
\label{example:y_gplus_mn_type_classified_by_a_map}
Let $Y$ be a connected $(n-1)$-type and suppose that we are given a map
\[
f  \colon Y \rightarrow \B_{G}(M, n+1),
\]
which by \cref{observation:moduli_interpretation_of_eilenberg_maclane_spaces} corresponds to a triple 
\[
(X \rightarrow Y, \alpha  \colon Y \rightarrow \B G, \pi_{n}X \simeq \alpha^{*} M).
\]
After unwrapping the definitions, we see that if $f$ is a $1$-equivalence (that is, induces an isomorphism $\pi_{1} Y \simeq G$ on fundamental groups) then $X$ is of type $Y +_{G} (M, n)$ in the sense of \cref{definition:moduli_space_of_spaces_of_type_y_gplus_mn}.
\end{example}

\begin{defin}
\label{defin:infty_group_of_autos_of_a_pair}
Let $G$ be a group and let $\B G \rightarrow \abeliangroups^{\cong}$ be a classifying map for a $G$-module $M$. The \emph{automorphism $\infty$-group of the pair}
\[
\Aut(\B G, M)
\]
is the space of self-equivalences of $\B G$ considered as an object of $\spaces_{/\abeliangroups^{\cong}}$.
\end{defin}

Notice that since $\B_{G}(M, n+1)$ is a pullback along the map $\B G \rightarrow \abeliangroups^{\cong}$, it acquires an action of $\Aut(BG, M)$, inducing one on $\map(Y, \B_{G}(M, n+1))$ for any space $Y$. The subspace $\map_{1-\mathrm{eq}}(Y, \B_{G}(M, n+1))$ of $1$-equivalences is stable under this action. Similarly, we have an action of $\Aut_{\spaces}(Y)$ by acting on the source, and these two actions commute with each other as they are given respectively by pre- and post-composition. 

The main result of this section, \cref{thm:classification_of_n_types_fixed_module} below, informally says that we have a canonical equivalence
\begin{equation}
\label{equation:homotopy_orbit_statement_classification_of_n_types}
\mathcal{M}(Y +_{G} (M, n)) \simeq \map_{1-\mathrm{eq}}(Y, \B_{G}(M, n+1))_{h( \Aut_{\spaces}(Y) \times \Aut(\B G, M))};
\end{equation}
that is, the moduli space on the left can be recovered as an appropriate space of homotopy orbits. By \cref{observation:homotopy_quotient_by_an_infty_group}, to give such an equivalence is equivalent to giving a fibre sequence of space of the form
\begin{equation}
\label{equation:fibre_seq_in_classification_of_n_types_with_fixed_module}
\map_{1-\mathrm{eq}}(Y, \B_{G}(M, n+1)) \rightarrow \mathcal{M}(Y +_{G} (M, n)) \rightarrow \mathcal{M}(Y) \times \B \Aut(\B G, M),
\end{equation}
where we consider the target as pointed using $Y \in \mathcal{M}(Y)$ and the canonical basepoint of $\B \Aut(\B G, M)$. We now explain the construction of these two maps together with the nullhomotopy of their composite. 

\begin{construction}
\label{construction:fibre_sequence_of_moduli_spaces_for_n_types}
We begin by defining the right hand side map in (\ref{equation:fibre_seq_in_classification_of_n_types_with_fixed_module}). Since its codomain is a product, we have to define a map into each of the factors. The map $\mathcal{M}(Y +_{G} (M, n)) \rightarrow \mathcal{M}(Y)$ sends $X$ to $X_{\leq (n-1)}$. To define the map onto the second factor, observe that $\B \Aut(\B G, M)$ is canonically equivalent to the subgroupoid of $(\spaces_{/\abeliangroups ^{\cong}})^{\cong}$ spanned by spaces equivalent to $\B G$ equipped with its map to $\abeliangroups^{\cong}$. The needed map $\mathcal{M}(Y +_{G} (M, n)) \rightarrow \B \Aut(\B G, M)$ takes $X$ to $X_{\leq 1}$ considered as a space over $\abeliangroups^{\cong}$ through the classifying map for the action on $\pi_{n}(X)$. 

The left map is defined by the construction of  \cref{example:y_gplus_mn_type_classified_by_a_map}. By the universal property of \cref{observation:moduli_interpretation_of_eilenberg_maclane_spaces}, spaces $X$ of type $Y +_{G} (M, n)$ obtained in this way have $(n-1)$-truncation canonically equivalent to $Y$ and come with canonical isomorphisms $\pi_{1}X \simeq \pi_{1}Y \simeq G$ and $\pi_{n}X \simeq M$, the latter an isomorphism of $G$-modules. It follows that the composite of the left and right map is canonically null-homotopic. 
\end{construction}

\begin{theorem}
\label{thm:classification_of_n_types_fixed_module}
Let $Y$ be a connected $(n-1)$-type whose fundamental group is isomorphic to $G$ and let $M$ be a $G$-module. Then, the diagram 
\[
\map_{1-\mathrm{eq}}(Y, \B_{G}(M, n+1)) \rightarrow \mathcal{M}(Y +_{G} (M, n)) \rightarrow \mathcal{M}(Y) \times \B \Aut(\B G, M)
\]
of \cref{construction:fibre_sequence_of_moduli_spaces_for_n_types} together with its canonical nullhomotopy is a fibre sequence.
\end{theorem}

\begin{proof}
The canonical nullhomotopy induces a map 
\[
\map_{1-\mathrm{eq}}(Y, \B_{G}(M, n+1)) \rightarrow F
\]
into the fibre of the second map. Tracing through definitions, we see that $F$ can be identified with the $\infty$-groupoid of spaces $X$ of type $Y +_{G} (M, N)$ together with a fixed equivalence $X_{\leq n-1} \simeq Y$, fixed equivalence $\phi \colon \B G \simeq X_{\leq 1}$, and an isomorphism $\phi^{*} \pi_{n}(X \rightarrow Y) \simeq M$ of $G$-modules. By an application of \cref{observation:moduli_interpretation_of_eilenberg_maclane_spaces}, this is equivalent to $\map_{1-\mathrm{eq}}(Y, \B_{G}(M, n+1))$, as needed. 
\end{proof}

\begin{cor}
A homotopy type of a space of type $Y +_{G} (M, n)$ is uniquely specified by a homotopy class of maps $Y \rightarrow \B_{G}(M, n+1)$ inducing an isomorphism on fundamental groups, well-defined up to the action of $\pi_{0}(\Aut_{\spaces}(Y) \times \Aut(BG, M))$.
\end{cor}

\begin{proof}
Keeping in mind \cref{thm:classification_of_n_types_fixed_module}, this is the same argument as in the proof of \cref{prop:postnikov_invariant_determining_homotopy_type_of_y_plus_an}.
\end{proof}

\begin{rem}
\cref{thm:classification_of_n_types_fixed_module} is essentially equivalent to \cite[Prop. 3.7]{realization_space_of_a_pi_algebra} of Blanc-Dwyer-Goerss, with minor differences due to the fact that we work with unpointed spaces. The same arguments as above used in the $\infty$-category $\spaces_{*}$ of pointed spaces would yield a result identical to theirs. 
\end{rem}

\section{Classification of spherical presheaves} 
\label{section:classification_of_presheaves}

In this chapter, which is the technical heart of the current work, we develop a Postnikov theory for spherical (that is, product-preserving) presheaves on a small $\infty$-category $\inftyc$ subject to some technical conditions. Our main results relate the Postnikov invariants in $\presheaves_{\Sigma}(\inftyc)$ and $\presheaves_{\Sigma}(h\inftyc)$. This is useful, since $\presheaves_{\Sigma}(h \inftyc)$ only involves presheaves on an ordinary category and hence is quite computable in practice. 

\subsection{Moduli objects in presheaves}

In \S\ref{section:classification_of_spaces}, we developed a theory of Postnikov invariants of spaces by using the Grothendieck construction. The latter gave a canonical equivalence, for any space $X \in \spaces$, of the form 
\[
\Fun(X, \spaces^{\cong}) \simeq (\spaces_{/X})^{\cong}.
\]
In these terms, the $\infty$-groupoid $\spaces^{\cong}$ can be thought of as a moduli object, as maps into it parametrize spaces over $X$. In this section, we show the existence of similar objects internal to the setting of presheaves on fixed small indexing $\infty$-category $\inftyc$. 

The $\infty$-category of $\presheaves(\inftyc) := \fun(\inftyc^{op}, \spaces)$ of presheaves is an archetypical example of an \emph{$\infty$-topos}. The latter are, informally, $\infty$-categories which behave like the $\infty$-category of spaces; for a detailed account see \cite[\S 6]{lurie_higher_topos_theory}. 

\begin{rem}
In the current work we do not consider any $\infty$-topoi which are not presheaf categories, hence a formal definition will not be necessary. We only use the language of $\infty$-topoi as the results we reference are proven at this level of generality and so it would be unnatural to avoid it. 
\end{rem}
We have the Yoneda embedding $y  \colon \inftyc \hookrightarrow \presheaves(\inftyc)$ given informally by the formula
\[
y(c)(c^{\prime}) := \map_{\inftyc}(c^{\prime}, c);
\]
Lurie shows that this is a fully faithful functor of $\infty$-categories \cite[5.1.3.1]{lurie_higher_topos_theory}. Limits and colimits in $\presheaves(\inftyc)$ are computed levelwise, because it is a functor category \cite[5.1.3.2]{lurie_higher_topos_theory}. In particular, presheaf $\infty$-categories are always complete and cocomplete, because that is true for the $\infty$-category $\spaces$. They also enjoy the following universal property characterizing them as free cocompletions.

\begin{prop}[Universal property of the presheaf $\infty$-category]
\label{prop:universal_property_presheaf_category}
Let $\inftyd$ be a cocomplete $\infty$-category and $\Fun^{L}(\presheaves(\inftyc), \inftyd)$ be the $\infty$-category of cocontinous functors. Then, the restriction
\[
\fun^{L}(\presheaves(\inftyc), \inftyd) \rightarrow \fun(\inftyc, \inftyd)
\]
along the Yoneda embedding is an equivalence of $\infty$-categories.
\end{prop}

\begin{proof}
This is \cite[5.1.5.6]{lurie_higher_topos_theory}.
\end{proof}

Note that one interpretation of \cref{thm:grothendieck_construction_for_spaces}, which in particular implies that for any space $X$ 
\[
\Fun(X, \spaces^{\cong}) \simeq (\spaces_{/X})^{\cong},
\]
is that the $\infty$-groupoid $\spaces^{\cong}$ is the object classifier for the $\infty$-topos of spaces. By a result of Lurie, such an object classifier exists in any $\infty$-topos \cite[6.1.6]{lurie_higher_topos_theory}. In the case of presheaf $\infty$-categories, we have the following explicit description:

\begin{prop}
\label{prop:existence_of_an_object_classifier}
The presheaf $\presheaves(\inftyc_{/-}) ^{\cong} \colon \inftyc^{op} \rightarrow \largespaces$ valued in the $\infty$-category of large spaces defined on objects by 
\[
\presheaves(\inftyc_{/-}) ^{\cong}(c) := \presheaves(\inftyc_{/c})^{\cong}
\]
and on morphisms by sending an arrow $f \colon c \rightarrow d$ to restriction 
\[
\presheaves(\inftyc_{/d})^{\cong} \rightarrow \presheaves(\inftyc_{/c})^{\cong}
\]
along composition with $f$ functor $\inftyc_{/c} \rightarrow \inftyc_{/d}$, is the object classifier in the $\infty$-topos $\presheaves(\inftyc)$. That is, for any presheaf $X \in \presheaves(\inftyc) := \Fun(\inftyc^{op}, \spaces)$ of small spaces we have a natural equivalence
\begin{equation}
\label{equation:object_classifier_for_presheaves}
\map_{\fun(\inftyc^{op}, \largespaces)} (X, \presheaves(\inftyc_{/-}) ^{\cong}) \simeq (\presheaves(\inftyc) _{/X}) ^{\cong}
\end{equation}
between natural transformations into $\presheaves(\inftyc_{/-}) ^{\cong}$ and the $\infty$-groupoid of presheaves over $X$. 
\end{prop}

\begin{proof}
Both sides of (\ref{equation:object_classifier_for_presheaves}) take (small) colimits in $X$ to limits, the latter by \cite[6.1.3.9, 6.1.3.10]{lurie_higher_topos_theory}. Thus, it is enough to show that both sides are naturally equivalent when $X = y(c)$ is in the image of the Yoneda embedding, in which case this restricts to the standard equivalence
\[
\presheaves(\ccat)_{/y(c)} \simeq \presheaves(\ccat_{/c}),
\]
which holds even before passing to maximal subgroupoids, see \cite[5.1.6.12]{lurie_higher_topos_theory}.
\end{proof}

\begin{warning}[Size issues]
When discussing object classifiers we invariably run into size issues. In the case of \cref{prop:existence_of_an_object_classifier}, this is visible in the fact that the object classifier is a functor
\[
\presheaves(\inftyc_{/-}) ^{\cong} \colon \inftyc^{op} \rightarrow \widehat{\spaces},
\]
thus is a presheaf valued in not necessarily small spaces and hence not an object of $\presheaves(\ccat)$ itself. This is enough for our purposes. 

Note that the size is essentially the only issue, as a theorem of Rezk asserts that for any $\infty$-topos $\euscr{X}$, the functor that associates to each $x \in \euscr{X}$ the $\infty$-groupoid $(\euscr{X} _{/x}) ^{\kappa}$ of those morphisms into $x$ which are relatively $\kappa$-compact, is representable for all sufficiently large regular cardinals $\kappa$, see \cite[6.1.6.8]{lurie_higher_topos_theory}.  
\end{warning}

The description of the object classifier given above might seem slightly opaque at first sight, but it can be understood in the following way. Given a morphism $X \rightarrow Y$ of presheaf, its associated classifying map has the form $Y \rightarrow \presheaves(\inftyc_{/-}) ^{\cong}$ and so to each object $c \in \mathcal{C}$ and each element $e \in Y(c)$ it attaches a presheaf over $\inftyc_{/c}$. By naturality, this presheaf is the local fibre over $p$ in the following sense:

\begin{defin}
\label{defin:local_fibre}
Let $p \colon X \rightarrow Y$ be a map of presheaves and let $e \in Y(c)$ be a point. The \emph{local fibre over $e$}, denoted by $F_{e}$, is the pullback of the diagram 

\begin{center}
$(p_{c})_{*} X \rightarrow (p_{c})_{*} Y \leftarrow 1$
\end{center}
in $\presheaves(\inftyc_{/c})$, where $p_{c} \colon \inftyc_{/c} \rightarrow \inftyc$ is the projection, $(p_{c})_{*} \colon \presheaves(\inftyc) \rightarrow \presheaves(\inftyc _{/c})$ restriction along the projection, and $1 \rightarrow (p_{c})_{*}Y$ is the map from the terminal presheaf specified by
\[
e \in Y(c) \simeq ((p_{c})_{*} Y)(\mathrm{id}_{c}) \simeq \varprojlim_{\inftyc_{/c}} \ (p_{c})_{*} Y \simeq \map_{\presheaves(\inftyc_{/c})}(1, Y)
\]
\end{defin}

\begin{rem}
After unwrapping \cref{defin:local_fibre}, we see that in concrete terms the local fibre $F_{e} \colon (\ccat_{/c})^{op} \rightarrow \spaces$ is given by the formula
\[
F_{e}(f  \colon c^{\prime} \rightarrow c) := X(c^{\prime}) \times _{Y(c^{\prime})} \{ f^{*}e \}.
\]
\end{rem}

\begin{rem}
The intuition behind \cref{defin:local_fibre} is as follows: Given a morphism $p \colon X \rightarrow Y$ of presheaves, to encode $X$ it is not enough to remember for each $c \in \inftyc$ the ordinary fibres of the map $p(c) \colon X(c) \rightarrow Y(c)$ of spaces (which can be recovered as the value of the local fibre when evaluated on $\mathrm{id}_{c} \in \inftyc_{/c}$). Instead, the local fibre also encodes the maps between fibers of $p(c)$ for varying objects $c$. 
\end{rem}

One might likewise be interested in classifying families of objects equipped with some additional algebraic structure. For example, in \S\ref{subsection:action_of_the_fundamental_group} we made use of the theory of modules (more commonly known as local coefficient systems). By the Grothendieck construction, to give a module over a space $X$ was the same as to give a functor $X \rightarrow \abeliangroups^{\cong}$ into the groupoid of abelian groups. Informally, this says that $\abeliangroups ^{\cong}$ plays the role of \emph{moduli of abelian groups} in the $\infty$-category of spaces. Below, we show such that an analogous object exists for any presheaf category and describe it explicitly.

Let $\ltopoi$ denote the $\infty$-category of $\infty$-topoi and functors that preserve all colimits and finite limits considered in \cite[\S 6.3]{lurie_higher_topos_theory}. The $\infty$-category $\ltopoi$ is convenient setting for representing moduli problems due to the following:

\begin{lemma}
\label{lem:existence_of_classifying_topoi}
Let $\euscr{A}$ be small a $\infty$-category with finite limits. Then, for any $\infty$-topos $\euscr{X}$ there is a natural equivalence
\[
\Fun^{L}_{lex}(\presheaves(\euscr{A}), \euscr{X}) \simeq \Fun_{lex}(\euscr{A}, \euscr{X})
\]
between the $\infty$-categories of cocontinous, left exact functors $\presheaves(\euscr{A}) \rightarrow \euscr{X}$ and that of left exact functors $\euscr{A} \rightarrow \euscr{X}$. 
\end{lemma}

\begin{proof}
By the universal property of the presheaf $\infty$-category, which we stated as \cref{prop:universal_property_presheaf_category}, to give a cocontinous functor $\presheaves(\euscr{A}) \rightarrow \euscr{X}$ is the same as to give a functor $\euscr{A} \rightarrow \euscr{X}$. This restricts to an equivalence as above by \cite[6.1.5.2]{lurie_higher_topos_theory}.
\end{proof}

\begin{rem}
Notice that the maximal subgroupoid of $\Fun^{L}_{lex}(\presheaves(\euscr{A}), \euscr{X})$ is exactly the mapping space in the $\infty$-category $\ltopoi$. Thus, another way to state \cref{lem:existence_of_classifying_topoi} is to say that the functor
\[
\fun_{lex}(\euscr{A}, -)^{\cong} \colon \ltopoi \rightarrow \largespaces
\]
is corepresented inside the $\infty$-category $\ltopoi$ by the presheaf $\infty$-topos $\presheaves(\euscr{A})$. 
\end{rem}

The usefulness of \cref{lem:existence_of_classifying_topoi} comes from the observation of Lawvere that algebraic structures can often be described as left exact functors out of a small $\infty$-category. The following example will be of particular importance: 

\begin{defin}
\label{defin:discrete_abelian_group}
Let $\infty$ be an $\infty$-category with finite limits. A \emph{discrete abelian group object} in $\inftyc$ is an abelian group object in the subcategory $\text{Disc}(\inftyc) := \tau_{\leq 0} \inftyc$ of discrete objects. We will write $\abeliangroups(\inftyc)$ for the category of discrete abelian group objects. 
\end{defin}

\begin{rem}
\label{remark:functoriality_of_discrete_abelian_groups}
Since by \cite[5.5.6.16]{lurie_higher_topos_theory} left exact functors preserve discrete objects, the construction $\inftyc \mapsto \abeliangroups(\inftyc)$ is functorial in $\infty$-categories with finite limits and left exact functors. 
\end{rem}

\begin{warning}
Note that in the context of \cref{defin:discrete_abelian_group}, we have
\[
\abeliangroups(\inftyc) \simeq \abeliangroups(\text{Disc}(\inftyc)).
\]
This is consistent with our convention of reserving the word ``group'' for the classical notion, rather than any of its variants appearing in higher algebra. 
\end{warning}

We will write $\abeliangroups_{fg}^{op}$ for the opposite of the category of finitely generated abelian groups. Since for any $A \in \abeliangroups_{fg}^{op}$
\[
y(\mathbb{Z})(A) \simeq \map_{\abeliangroups_{fg}^{op}}(A, \mathbb{Z}) \simeq \map_{\abeliangroups_{fg}}(\mathbb{Z}, A) \simeq A
\]
is canonically a discrete abelian group, the representable presheaf $y(\mathbb{Z}) \in \presheaves(\abeliangroups_{fg}^{op})$ is canonically an abelian group object inside the presheaf $\infty$-category. 

\begin{prop}
\label{prop:explicit_description_of_the_classifying_topos_for_abelian_groups}
For any $\infty$-topos $\euscr{X}$, the construction 
\[
F \mapsto F(y(\mathbb{Z})) 
\]
yields an equivalence
\[
\Fun^{L}_{lex}(\presheaves(\abeliangroups_{fg}^{op}), \euscr{X}) \simeq \abeliangroups(\euscr{X})
\]
between the $\infty$-category of cocontinuous, left exact functors and the category of discrete abelian group objects of $\euscr{X}$.
\end{prop}

\begin{proof}
We first claim that if $\inftyc$ is an $\infty$-category with finite limits, then the construction 
\[
A \mapsto A(\mathbb{Z})
\]
yields an equivalence
\[
\Fun_{lex}(\abeliangroups_{fg}^{op}, \inftyc) \simeq \abeliangroups(\inftyc)
\]
between left exact functors $\abeliangroups_{fg}^{op} \rightarrow \inftyc$ and the category of discrete abelian group objects. To see this, we can replace $\inftyc$ by its subcategory $\text{Disc}(\inftyc)$ of discrete objects, which is an ordinary category. In this context, the needed equivalence follows from a classical argument using Lawvere theories. 

The statement then follows from the above observation applied to $\inftyc = \euscr{X}$ and \cref{lem:existence_of_classifying_topoi}. 
\end{proof}

We can now prove the existence of a moduli object for families of abelian groups in the context of $\infty$-categories of presheaves: 

\begin{prop}
\label{prop:existence_of_moduli_of_discrete_abelian_groups}
Let $\inftyc$ be a small $\infty$-category. Then, the $\largespaces$-valued presheaf $\mathcal{M}_{\inftyc}(\abeliangroups)$ on $\inftyc$ defined by
\[
\mathcal{M}_{\inftyc}(\abeliangroups)(c) := \abeliangroups(\presheaves(\inftyc_{/c}))^{\cong}
\]
is the moduli object for discrete abelian groups in $\presheaves(\inftyc)$; that is, for any $X \in \presheaves(\inftyc)$ there is a natural equivalence of spaces 
\[
\map_{\fun(\inftyc^{op}, \largespaces)}(X, \mathcal{M}_{\inftyc}(\abeliangroups))  \simeq \abeliangroups(\presheaves(\inftyc)_{/X})^{\cong}
\]
between natural transformations into the moduli and the groupoid of discrete abelian groups in presheaves over $X$.
\end{prop}

\begin{proof}
We first construct a morphism
\begin{equation}
\label{equation:map_for_the_universal_property_of_moduli_of_abelian_groups}
\phi \colon \abeliangroups(\presheaves(\inftyc)_{/X})^{\cong} \rightarrow \map_{\fun(\inftyc^{op}, \largespaces)}(X, \mathcal{M}_{\inftyc}(\abeliangroups)) 
\end{equation}
of $\infty$-groupoids; it will be clear from the construction that it is contravariantly natural in $X$. Suppose that $A \in \abeliangroups(\presheaves(\inftyc)_{/X})^{\cong}$; we have to declare its image 
\[
\phi_{A} \colon X \rightarrow \mathcal{M}_{\inftyc}(\abeliangroups). 
\]
inside the mapping space. For $c \in \inftyc$ we define  
\[
\phi_{A}(c) \colon X(c) \rightarrow \mathcal{M}_{\inftyc}(\abeliangroups)(c) = \abeliangroups(\presheaves(\mathcal{C}_{/c}))^{\cong} 
\]
by declaring that if $e \in X(c)$ then
\[
\phi_{A}(c)(e) := e_{*}(A),
\]
where we identify $e$ with a map $y(c) \rightarrow X$ and 
\[
e_{*} \colon \presheaves(\inftyc)_{/X} \rightarrow \presheaves(\inftyc)_{/y(c)} \simeq \presheaves(\inftyc_{/c})
\]
is the corresponding pullback functor, which is left exact and hence induces a functor between groupoids of abelian group objects. We will show that $\phi$ is an equivalence. 

Notice that the target of (\ref{equation:map_for_the_universal_property_of_moduli_of_abelian_groups}) take small colimits in $X$ to limits of spaces. We claim that the same is true for the source. Indeed, by \cref{prop:explicit_description_of_the_classifying_topos_for_abelian_groups} the source can be identified with the $\infty$-groupoid of of cocontinuous, left exact functors, so that 
\[
\abeliangroups(\presheaves(\inftyc)_{/X})^{\cong} \simeq \map_{\ltopoi}(\presheaves(\abeliangroups_{fg}^{op}), \presheaves(\inftyc)_{/X}).
\]
The claim then follows from a combination of \cite[6.1.3.9, 6.1.3.10]{lurie_higher_topos_theory}, which implies that passing to overcategories of 
 an $\infty$-topos takes colimits to limits in $ \widehat{\euscr{C}at _{\infty}}$, and \cite[6.3.2.3]{lurie_higher_topos_theory}, which says that 
the natural inclusion $\ltopoi \hookrightarrow \widehat{\euscr{C}at _{\infty}}$ preserves limits. 

Since any $X \in \presheaves(\inftyc)$ is a small colimit of representables, we deduce that it is enough to verify that $\phi$ is an equivalence when $X \simeq y(c)$ is representable. In this case, unwrapping the definitions shows that $\phi$ reduces to the standard equivalence
\[
\abeliangroups(\presheaves(\inftyc)_{/y(c)})^{\cong} \simeq \abeliangroups(\presheaves(\inftyc_{/c}))^{\cong} =  \mathcal{M}_{\inftyc}(\abeliangroups)(c).
\]
\end{proof}

\begin{rem}
Note that in \S\ref{subsection:postnikov_invariants_of_spaces} we worked in the $\infty$-topos of spaces and used the word ``moduli'' to refer to $\infty$-groupoids of objects of 
a given type. This terminology is consistent with \cref{prop:existence_of_moduli_of_discrete_abelian_groups} in the following sense: if $\inftyc \simeq \Delta^{0}$ is a one-point category, so that $\presheaves(\inftyc) \simeq \spaces$, the presheaf $\mathcal{M}_{\inftyc}(\abeliangroups)$ appearing in \cref{prop:existence_of_moduli_of_discrete_abelian_groups} is exactly the groupoid of abelian groups. 
\end{rem}

\begin{rem}
In the context of \cref{prop:existence_of_moduli_of_discrete_abelian_groups}, note that we have a canonical equivalence
\[
\abeliangroups(\presheaves(\inftyc_{/c}))^{\cong} \simeq \fun((\inftyc_{/c})^{op}, \abeliangroups)^{\cong},
\]
as an abelian group in a presheaf category is the same as a presheaf of abelian groups.
\end{rem}
We will also need moduli objects for Eilenberg-MacLane spaces, which we describe now. Recall that in the case of the $\infty$-category $\spaces$ of spaces, we observed that if $Y$ is an $(n-1)$-type, then a map $X \rightarrow Y$ is an $(n-1)$-equivalence with $X$ an $n$-type if and only if all of the fibres are Eilenberg-MacLane spaces of degree $n$. Alternatively, this happens precisely when the classifying map $Y \rightarrow \spaces$ factors through the \inftycat $\emspaces_{n}$ of Eilenberg-MacLane spaces of degree $n$. 

A similar statement holds in the $\infty$-category $\presheaves(\inftyc)$. More precisely, if $Y \in \presheaves(\inftyc)$ is an $(n-1)$-type, then a map $X \rightarrow Y$ of presheaves is a $(n-1)$-equivalence with $X$ an $n$-type if and only if all of the local fibres in the sense of \cref{defin:local_fibre} are valued in Eilenberg-MacLane spaces of degree $n$. This observation allows one to write down the a presheaf which plays for $\presheaves(\inftyc)$ a role analogous to the role played for spaces by the $\infty$-groupoid $\emspaces_{n}^{\cong}$. 

In the following, as elsewhere in the text, by an Eilenberg-Maclane space we mean a space of type $K(A, n)$ for some \emph{abelian} group $A$ and fixed degree $n \geq 1$. 

\begin{prop}
\label{prop:existence_of_moduli_of_em_spaces_in_presheaves}
Let $\inftyc$ be the small $\infty$-category. Then, the $\largespaces$-valued presheaf
\[
\mathcal{M}_{\inftyc}(\emspaces_{n}) \colon \inftyc^{op} \rightarrow \largespaces
\]
defined on objects by 
\[
\mathcal{M}_{\inftyc}(\emspaces_{n})(c) = \fun((\inftyc_{/c})^{op}, \emspaces_{n}^{\cong})
\]
and sending $f \colon c \rightarrow d$ to precomposition with the induced functor $\inftyc_{/c} \rightarrow \inftyc_{/d}$, is the moduli of Eilenberg-MacLane spaces of degree $n$ in the following sense: for any $X \in \presheaves(\inftyc)$ there is a natural equivalence 

\begin{center}
$\map_{\fun(\inftyc^{op}, \largespaces)}(X, \mathcal{M}_{\inftyc}(\emspaces_{n})) \simeq (\presheaves(\inftyc)^{\emspaces_{n}}_{/X})^{\cong}$,
\end{center}
where on the right hand side we have the $\infty$-groupoid of those presheaves over $X$ all of whose local fibres are valued in Eilenberg-MacLane spaces of degree $n$. 
\end{prop}

\begin{proof}
This is clear, as $\mathcal{M}_{\inftyc}(\emspaces_{n})$ is defined as a subobject of the object classifier described in \cref{prop:existence_of_an_object_classifier} and a map into $\presheaves(\inftyc_{/-}) ^{\cong}$ factors through it if and only if it all of the local fibres are valued in Eilenberg-MacLane spaces of degree $n$.
\end{proof}

Because of the key role played by the presheaf \cref{prop:existence_of_moduli_of_em_spaces_in_presheaves}, understanding its levelwise homotopy type is of considerable importance. This is equivalent to describing the possible equivalence classes of $\infty$-groupoids valued in Eilenberg-MacLane spaces of a fixed degree degree as well as their automorphisms. We show that in the case of $\infty$-categories possessing a terminal object, in particular in the case of overcategories, such presheaves can be understood using only their homotopy groups:

\begin{prop}
\label{prop:pointing_eilenberg_maclane_presheaves}
Let $\inftyc$ be a small \inftycat with a terminal object $t \in \inftyc$. Suppose that $X \colon \inftyc^{op} \rightarrow \emspaces_{n}$ is a presheaf of Eilenberg-MacLane spaces of degree $n$ and let 
\[
\pi_{n}X \colon \inftyc^{op} \rightarrow \abeliangroups
\]
be the presheaf of abelian groups defined by $(\pi_{n})X(c) := \pi_{n} X(c)$\footnote{Note that this is well-defined without choosing a basepoint as we only consider abelian Eilenberg-MacLane spaces, which are simple.}. Then,
\begin{enumerate}
    \item $X$ can be lifted to a presheaf of $\mathbf{E}_{\infty}$-spaces, 
    \item for any such lift there is a canonical equivalence $X \simeq \B^{n}(\pi_{n}X)$ of presheaves of $\mathbf{E}_{\infty}$-spaces 
\end{enumerate}
\end{prop}

\begin{proof}
If $\mathrm{pt}_{\inftyc} \colon \inftyc^{op} \rightarrow \spaces$ denotes the terminal presheaf, then since mapping from a constant functor is mapping into a limit we have 
\[
\map_{\presheaves(\inftyc)}(\mathrm{pt}_{\inftyc}, X) \simeq \varprojlim_{\inftyc^{op}} X \simeq X(t),
\]
where we've used that $t \in \inftyc^{op}$ is initial so that its inclusion is a final functor. It follows that a choice of a point of $X(t)$, which is non-empty as an Eilenberg-MacLane space, gives $X$ a structure of a pointed presheaf, which we can identify with a functor $X \colon \inftyc^{op} \rightarrow (\emspaces_{n})_{*}$. 

By \cref{prop:projectivity_of_em_spaces}, $X$ can be then further uniquely lifted to a presheaf valued in $\Alg_{\mathbf{E}_{\infty}}(\emspaces_{n})$, proving $(1)$. The functor 
\[
\B^{n}(-) \colon \abeliangroups \rightarrow \Alg_{\mathbf{E}_{\infty}}(\emspaces_{n})
\]
is an equivalence of $\infty$-categories with inverse 
\[
\pi_{n} \colon \Alg_{\mathbf{E}_{\infty}}(\emspaces_{n}) \rightarrow \abeliangroups
\]
which implies (2). 
\end{proof}

\begin{cor}
\label{corollary:presheaves_of_em_spaces_described_by_homotopy_groups_up_to_equivalence}
If $\ccat$ has a terminal object, then two presheaves $X, X^\prime \colon \inftyc^{op} \rightarrow \emspaces_{n}$ valued in Eilenberg-MacLane spaces 
of degree $n$ are equivalent if and only if their presheaves $\pi_{n}X$ and  $\pi_{n}X^\prime$ of homotopy groups are isomorphic. 
\end{cor}

\begin{proof}
This is immediate from the second part of \cref{prop:pointing_eilenberg_maclane_presheaves}, as once basepoints of $X, X'$ are chosen, they are both equivalent to $\B^{n} (\pi_{n} X) \simeq \B^{n} (\pi_{n} X')$. 
\end{proof}

\begin{warning}
In the context of \cref{prop:pointing_eilenberg_maclane_presheaves}, the assumption that the indexing $\infty$-category $\ccat$ has a terminal object is crucially important. Without the latter, it is no longer true that any presheaf of Eilenberg-MacLane spaces of fixd degree can be given a basepoint. 
\end{warning}

Using \cref{prop:pointing_eilenberg_maclane_presheaves}, we see that to understand automorphisms of an arbitrary Eilenberg-MacLane-valued presheaf, it is enough to do so in the case of a presheaf of the form $\B^{n} A$, where $A: \ccat^{op} \rightarrow \abeliangroups$ is a presheaf of abelian groups. This can be done by a map analogous to the absolute case studied in \cref{lemma:maps_between_eilenberg_maclane_spaces}. 

Left multiplication an an element $b \in \B^{n}A(t)$ defines an automorphism of $A(t)$, and since $t$ is assumed to be terminal, left multiplication by $f^{*}b$ for the unique $f \colon c \rightarrow t$ defines a compatible family of automorphisms of $\B^{n}A(c)$ for $c \in \ccat$. This gives a map of $\infty$-groups 
\[
\B^{n} A(t) \rightarrow \Aut_{\presheaves(\ccat)}(\B^{n} A)
\]
which together with functoriality of $\B^{n}$ assembles into a morphism
\begin{equation}
\label{equation:map_into_automorphisms_of_a_presheaf}
\phi \colon \B^{n} A(t) \rtimes \Aut_{\Fun(\ccat^{op}, \abeliangroups)}(\pi_{n}X) \rightarrow \Aut_{\presheaves(\ccat)}(\B^{n} A).
\end{equation}

\begin{prop}
\label{proposition:automorphism_of_em_presheaves}
For any $A: \ccat^{op} \rightarrow \abeliangroups$, the map $\phi$ of (\ref{equation:map_into_automorphisms_of_a_presheaf}) is an equivalence of $\infty$-groups. 
\end{prop}

\begin{proof}
We argue as in the proof of \cref{lemma:maps_between_eilenberg_maclane_spaces}, which covered the case where $\ccat = \Delta^{0}$. By construction, we have a commutative diagram 
\[
\begin{tikzcd}
	{\B^{n} A(t) \rtimes \Aut_{\Fun(\ccat^{op}, \abeliangroups)}(A) } && {\Aut_{\presheaves(\ccat)}(\B^{n} A)} \\
	& {\B^{n}A(t)}
	\arrow["{p_{1}}"', from=1-1, to=2-2]
	\arrow["", from=1-3, to=2-2]
	\arrow["\phi", from=1-1, to=1-3]
\end{tikzcd}
\]
of spaces. Here, the right vertical arrow is the composite 
\[
\Aut_{\presheaves(\ccat)}(\B^{n} A) \rightarrow \Aut_{\spaces}(\B^{n} A(t)) \rightarrow \B^{n}A(t),
\]
with the first arrow is restriction to the value over the terminal object and the second is given by evaluating at a basepoint of $\B^{n}A(t)$. The left vertical arrow is the projection onto the first coordinate. 

It is enough to verify that we have an induced equivalence of fibres of the vertical arrows. The fibre of the right map is given by the space of automorphisms of $\B^{n}A$ which preserve the basepoint of $\B^{n}A(t)$, which since $t$ is terminal we can identify with automorphisms of $\B^{n}(A)$ as a pointed presheaf. As 
\[
\abeliangroups \simeq \Alg_{\mathbf{E}_{\infty}}(\emspaces_{n}) \simeq (\emspaces_{n})_{*},
\]
by \cref{prop:projectivity_of_em_spaces}, this is exactly the automorphism $\infty$-group
\[
\Aut_{\presheaves(\ccat)_{*}}(\B^{n} A) \simeq  \Aut_{\Fun(\ccat^{op}, \abeliangroups)}(A).
\]
That's also the fibre of the left vertical map, ending the argument. 
\end{proof}

\subsection{Spherical presheaves and their homotopy groups}

In this section, we study the $\infty$-category of spherical presheaves on a small $\infty$-category satisfying certain technical assumptions. We show that such presheaves behave somewhat like pointed spaces in the sense that we have a good theory of homotopy groups. The latter will be modules in an appropriate sense, mimicking the action of the fundamental group of a pointed space on the higher homotopy groups. 

\begin{defin}
If $\inftyc$ is a small $\infty$-category admitting finite coproducts, we say that a presheaf $X \colon \inftyc^{op} \rightarrow \spaces$ is \emph{spherical} or \emph{product-preserving} if it takes finite coproducts in $\inftyc$ to products of spaces. We denote the full subcategory of $\presheaves(\inftyc)$ spanned by spherical presheaves by $\presheaves_{\Sigma}(\inftyc)$. 
\end{defin} 
In the terminology of \cite[\S5.5.8]{lurie_higher_topos_theory}, $\presheaves_{\Sigma}(\inftyc)$ is the \emph{nonabelian derived category}. It has the following universal property:

\begin{prop}
\label{prop:universal_property_of_the_nonabelian_derived_category}
Let $\inftyd$ be an $\infty$-category admitting filtered colimits and geometric realizations and let
\[
\Fun^{\mathrm{sft}}(\presheaves_{\Sigma}(\inftyc), \inftyd) \subseteq \Fun(\presheaves_{\Sigma}(\inftyc), \inftyd)
\]
denote the full subcategory spanned by functors which preserve sifted colimits; that is, which preserve filtered colimits and geometric realizations. Then, the restriction 
\[
\Fun^{\mathrm{sft}}(\presheaves_{\Sigma}(\inftyc), \inftyd) \rightarrow \fun(\inftyc, \inftyd)
\]
along the Yoneda embedding $y \colon \inftyc \rightarrow \presheaves_{\Sigma}(\inftyc)$ is an equivalence of $\infty$-categories, with an explicit inverse given by left Kan extension along $y$. 
\end{prop}

\begin{proof}
This is \cite[5.5.8.15]{lurie_higher_topos_theory}.
\end{proof}

\begin{observation}
\label{observation:cocontinuity_of_the_derived_functor} 
In the context of \cref{prop:universal_property_of_the_nonabelian_derived_category}, a functor $f \in \Fun^{\mathrm{sft}}(\presheaves_{\Sigma}(\inftyc), \inftyd)$ is cocontinuous if and only if its restriction $f \circ y$ along the Yoneda embedding preserves finite coproducts \cite[5.5.8.15, (3)]{lurie_higher_topos_theory}.
\end{observation}

\begin{observation}
\label{observation:limits_colimits_in_spherical_presheaves}
The inclusion $\presheaves_{\Sigma}(\inftyc) \hookrightarrow \presheaves(\inftyc)$ preserves limits, filtered colimits and geometric realizations, hence those are computed levelwise in $\presheaves_{\Sigma}(\inftyc)$. Moreover, $\presheaves_{\Sigma}(\inftyc)$ is the smallest subcategory of $\presheaves(\inftyc)$ containing the image of the Yoneda embedding and closed under filtered colimits and geometric realizations, see \cite[5.5.8.10, 13, 14]{lurie_higher_topos_theory}. 
\end{observation}

\begin{rem}
Since truncation preserves products, if $X \in \presheaves(\inftyc)$ is spherical, so are all of its truncations $X_{\leq n}$, which in the presheaf $\infty$-category are computed levelwise. It follows that a spherical preshaef is an $n$-type if and only if it is valued in $n$-types in spaces. In other words, Postnikov towers in $\presheaves_{\Sigma}(\inftyc)$ are calculated levelwise. 
\end{rem}

To endow $\presheaves_{\Sigma}(\inftyc)$ with some additional properties, we make assumptions on the indexing $\infty$-category. These assumptions are relatively weak, and usually satisfied in practice; they are nevertheless crucial to our approach. Intuitively, they make $\inftyc$ behave like an $\infty$-category parametrizing some kind of operations, so that objects of $\presheaves_{\Sigma}(\inftyc)$ behave like algebras. 

\begin{assumption}
\label{assumption:assumption_on_c_to_have_good_presheaves}
We assume that 
\begin{enumerate}
    \item $\inftyc$ is a pointed $\infty$-category, 
    \item $\inftyc$ admits all finite coproducts and 
    \item every object $c \in \inftyc$ admits a cogroup structure in the homotopy category $h \inftyc$.
\end{enumerate}
\end{assumption}

\begin{notation}
To emphasize the pointed (but not necessarily additive) nature of $\inftyc$, we denote the coproduct of $c_{1}, c_{2} \in \inftyc$ by $c_{1} \vee c_{2}$. 
\end{notation}
The last assumption, while strange, is crucial to applications. The case relevant to $\Pi$-algebras is that of the $\infty$-category $\spheres$ of finite-wedges of positive-dimensional spheres, where all objects can be in fact given an $\infty$-cogroup structure by virtue of being suspensions, not only the structure of a cogroup in the homotopy category.

\begin{rem}
We assume that for each $c \in \inftyc$ some homotopy cogroup structure was chosen. However, our results do not depend on this choice. These homotopy cogroup structures are not assumed to be compatible with each other and maps in $\inftyc$ are not required to be homotopy cogroup maps.
\end{rem}
We first derive some of the basic consequences which \cref{assumption:assumption_on_c_to_have_good_presheaves} has on the $\infty$-category of spherical presheaves.

\begin{prop}
\label{prop:spherical_presheaves_are_pointed}
A presheaf $X \in \presheaves_{\Sigma}(\inftyc)$ can be uniquely lifted to a presheaf of pointed spaces. 
\end{prop}

\begin{proof}
Let $0 \in \inftyc$ be the zero object. Since $X$ is product-preserving, we have a chain of equivalences
\[
X(0) \simeq X(0 \vee 0) \simeq X(0) \times X(0)
\]
and so $X(0)$ is contractible. This, combined with essentially unique maps $c \rightarrow 0$ for any $c \in \inftyc$, equips $X$ with the unique structure of a pointed presheaf. 
\end{proof}

\begin{cor}
The $\infty$-category $\presheaves_{\Sigma}(\inftyc)$ is pointed. In particular, any morphism $X \rightarrow Y$ of spherical presheaves has a well-defined fibre $F := X \times _{Y} \{ * \}$, which is again a spherical presheaf (since limits in $\presheaves_{\Sigma}(\inftyc)$ are computed levelwise). 
\end{cor}

\begin{prop}
\label{prop:spherical_presheaves_take_values_in_h_spaces}
Let $X$ be a spherical presheaf and let $c \in \inftyc$. Then $X(c)$ admits a structure of a group object in the homotopy category $h \spaces$ of spaces, functorial in maps of spherical presheaves. In particular, the basepoint component of $X(c)$ is a simple space.
\end{prop}

\begin{proof}
Since $X$ is spherical, a choice of a cogroup structure on $c \in h \inftyc$ induces a structure of a group object on $X(c) \in h \spaces$. In particular, $X(c)$ is an $H$-space and hence so is its basepoint component. It follows that the latter is simple by the Eckmann-Hilton argument. 
\end{proof}
Our main goal is to show that similarly to pointed, connected spaces, spherical presheaves are controlled by discrete objects, namely their homotopy groups. These will be algebras in the following sense: 

\begin{defin}
\label{defin:c_algebra}
A \emph{$\inftyc$-algebra} $X$ is a discrete object of $\presheaves_{\Sigma}(\inftyc)$; that is, a spherical presheaf valued in discrete spaces. 
\end{defin}

\begin{rem}
If $\inftyc = \spheres \subseteq \spaces_{*}$ is the $\infty$-category of wedges of positive-dimensional spheres, \cref{defin:c_algebra} recovers the notion of a $\Pi$-algebra. In general, we get a mild generalization of algebras in the sense of Lawvere \cite{lawvere1963functorial}, the generalization coming from the fact that we do not require that all objects of $\inftyc$ are coproducts of a single generating object. 
\end{rem}

Note that the as a full subcategory of $\presheaves_{\Sigma}(X)$, $\inftyc$-algebras form an ordinary category. A basic example is given by homotopy groups of a spherical presheaf:

\begin{defin}
\label{defin:homotopy_groups_of_a_spherical_presheaf}
If $X \in \presheaves_{\Sigma}(\inftyc)$ and $n \geq 0$, then its \emph{$n$-th homotopy $\inftyc$-algebra}, denoted by $\pi_{n}(X)$, is defined by
\[
(\pi_{n}(X))(c) := \pi_{n}(X(c)).
\]
In the particular case of $n = 0$, we will refer to 
\[
\pi_{0}(X) \simeq X_{\leq 0} 
\]
as the \emph{underlying $\inftyc$-algebra}. 
\end{defin}
Notice that this is well-defined since by \cref{prop:spherical_presheaves_are_pointed}, a spherical presheaf canonically takes values in pointed spaces. 

\begin{rem}[Group structure on homotopy algebras]
Since for every $c \in \inftyc$, the basepoint component of $X(c)$ is a simple space by \cref{prop:spherical_presheaves_take_values_in_h_spaces}, $\pi_{n}X$ is canonically an abelian group object in spherical presheaves for each $n \geq 1$. 
\end{rem}

The standard properties of homotopy groups of pointed spaces yield analogous properties of homotopy algebras of \cref{defin:homotopy_groups_of_a_spherical_presheaf}. For example, a fibre sequence of spherical presheaves leads to a long levelwise exact sequence of homotopy algebras. Moreover, homotopy algebras detect equivalences, as the following shows:

\begin{prop}
A map $X \rightarrow Y$ of spherical presheaves is an equivalence if and only if it induces isomorphisms $\pi_{n} X \simeq \pi_{n} Y$ on homotopy $\inftyc$-algebras for each $n \geq 0$. 
\end{prop}

\begin{proof}
The forward implication is clear, so let us prove its converse. A map $X \rightarrow Y$ is an equivalence if and only if it is so levelwise, so it's enough to show that $X(c) \rightarrow Y(c)$ is an equivalence for any $c \in \inftyc$. By \cref{prop:spherical_presheaves_take_values_in_h_spaces}, this admits a structure of a map of group objects in the homotopy category $h \spaces$ of spaces and since we assume that this map is a $\pi_{0}$-isomorphism, it is an equivalence if and only if it restricts to one on basepoint components. The restriction is an equivalence by the Whitehead's theorem, as it is a map of connected spaces which is an isomorphism on positive-dimensional homotopy groups. 
\end{proof}
Recall that in the case of spaces, the homotopy groups carried an additional structure of a module over the fundamental group, as we discussed in \S\ref{subsection:action_of_the_fundamental_group}. We now describe a generalization of this construction to spherical presheaves. 

We observe that the construction of relative homotopy groups of \cref{defin:relative_homotopy_groups} can be readily performed levelwise, leading to discrete abelian group objects in presheaves.  This relative homotopy group is a module (that is, a local coefficient system) over $X(c)$. Using functoriality, these relative homotopy groups assemble into a module over $X$ in the following way: 

\begin{defin}
\label{defin:relative_homotopy_groups_of_presheaves}
Let $Y \rightarrow X$ be a morphism in $\presheaves(\inftyc)$ such that for each $c \in \inftyc$, $Y(c) \rightarrow X(c)$ has simple fibres. For each $n \geq 1$, we call the presheaf
\[
\pi_{n}(Y \rightarrow X) \in \presheaves(\ccat_{/X})
\]
defined by
\[
\pi_{n}(Y \rightarrow X)(\eta \colon y(c) \rightarrow X) := \pi_{n}(F_{\eta}),
\]
where $F_{\eta} := Y(c) \times_{X(c)} \{ \eta \}$, the $n$-th \emph{relative homotopy group}. It is a (discrete) abelian group object in $\presheaves(\ccat_{/X})$
\end{defin}
We will be only interested in the relative homotopy groups the situation where $Y \rightarrow X$ is a morphism of spherical presheaves inducing an isomorphism $\pi_{0}X \simeq \pi_{0}Y$, so that the fibre is levelwise connected and hence levelwise simple by \cref{prop:spherical_presheaves_take_values_in_h_spaces}. In this case, the relative homotopy groups of \cref{defin:relative_homotopy_groups_of_presheaves} are modules over $X$ in the following sense: 

\begin{defin}
\label{defin:module_over_a_spherical_presheaf}
Let $X \in \presheaves_{\Sigma}(\inftyc)$ be a spherical presheaf. A \emph{module} $E \rightarrow X$ is a discrete abelian group object in $\presheaves_{\Sigma}(\inftyc) _{/X}$, the $\infty$-category of spherical presheaves over $X$. 

If $E \rightarrow X$ is a module, then its \emph{underlying abelian group object} is the fibre $A := E \times_{X} \mathrm{pt}$. It is an abelian group object in $\inftyc$-algebras. 
\end{defin}

\begin{rem}
Informally, one can think of the map $E \rightarrow A$ as encoding an action of $\Omega X$ on the underlying abelian group object $A$.
\end{rem}

\begin{rem}
\label{rem:modules_over_inftyc_algebras}
If $E \rightarrow X$ is a module and $X$ is a $\inftyc$-algebra (that is, if it is valued in discrete spaces), then so is $E$. In this case, \cref{defin:module_over_a_spherical_presheaf} reduces to the standard definition of a module over an algebra, cf. the case of $\Pi$-algebras discussed in \cite[\S4.9]{realization_space_of_a_pi_algebra}. 

We show below in \cref{theorem:equivelence_of_modules_with_zero_truncation} that a module over a spherical presheaf $X$ is equivalent data as one over  its underlying $\inftyc$-algebra $X_{\leq 0}$, so that our definition is in fact at the same level of generality as the classical one. 
\end{rem}

The following shows that sphericity of a presheaf can be detected at the level of homotopy groups, by asking whether they are modules in the sense of \cref{defin:module_over_a_spherical_presheaf}:

\begin{lemma}
\label{lem:sphericity_detected_by_modules}
Let $X \in \presheaves_{\Sigma}(\inftyc)$ be a spherical presheaf $Y \in \presheaves(\inftyc)_{/X}$ be a presheaf such that for every $c \in \inftyc$, $Y(c) \rightarrow X(c)$ has simple fibres. Then the following are equivalent: 
\begin{enumerate}
    \item for any $n \geq 1$ $\pi_{n}(Y \rightarrow X) \in \presheaves(\ccat_{/X})$ is spherical (and hence a module in the sense of \cref{defin:module_over_a_spherical_presheaf}) or
    \item $Y$ is spherical. 
\end{enumerate}
\end{lemma}

\begin{proof}
Let $c_{1}, c_{2} \in \inftyc$ and let $c_{1} \vee c_{2}$ be their coproduct together with structure maps $i_{1} \colon c_{1} \hookrightarrow c_{1} \vee c_{2}$ and $i_{2} \colon c_{2} \rightarrow c_{1} \vee c_{2}$. The presheaf $Y$ is spherical if and only if the map 
\[
Y(c_{1} \vee c_{2}) \rightarrow Y(c_{1}) \times Y(c_{2})
\]
induced by $i_{1}^{*}$ and $i_{2}^{*}$ is an equivalence. Consider the commutative diagram 

\begin{center}
	\begin{tikzpicture}
		\node (TL) at (0, 1.5) {$ Y(c_{1} \vee c_{2}) $};
		\node (BL) at (0, 0) {$ X(c_{1} \vee c_{2})   $};
		\node (TR) at (3.5, 1.5) {$ Y(c_{1}) \times Y(c_{2}) $};
		\node (BR) at (3.5, 0) {$ X(c_{1}) \times X(c_{2}) $};
		
		\draw [->] (TL) to (TR);
		\draw [->] (TL) to (BL);
		\draw [->] (TR) to (BR);
		\draw [->] (BL) to (BR);
	\end{tikzpicture},
\end{center}
where the bottom map is an equivalence by the assumption that $X$ is spherical. Thus, $Y$ is spherical if and only if the upper horizontal arrow induces an equivalence on fibres of the vertical ones. This is equivalent to
\begin{equation}
\label{equation:map_between_fibres_in_sphericity_criterion}
F_{x} \rightarrow F_{i_{1}^{*}x} \times F_{i_{2}^{*} x},
\end{equation}
being an equivalence for all $x \in X(c_{1} \vee c_{2})$, where $F$ denotes fibres over the corresponding points. 

As we assume that the fibres are simple, in particular connected, (\ref{equation:map_between_fibres_in_sphericity_criterion}) is an equivalence if and only if for all $n \geq 1$, the map
\begin{equation}
\label{equation:map_between_homotopy_groups_of_fibres_in_sphericity_condition}
\pi_{n}(F_{x}) \rightarrow \pi_{n}(F_{i_{1}^{*}x}) \times \pi_{n}(F_{i_{2}^{*} x})
\end{equation}
is an isomorphism. These are, by definition, the fibres in an analogous diagram for the map of presheaves $\pi_{n}(Y \rightarrow X) \rightarrow X$, and it follows that (\ref{equation:map_between_homotopy_groups_of_fibres_in_sphericity_condition}) is an equivalence for all $x \in X(c_{1} \vee c_{2})$ if and only if $\pi_{n}(Y \rightarrow X)$ is spherical. This ends the argument. 
\end{proof}

We move on towards the equivalence between the category of modules over a spherical presheaf and its underlying $\inftyc$-algebra promised in \cref{rem:modules_over_inftyc_algebras}. We will first need a few lemmas. 

\begin{lemma}
\label{lemma:truncation_comparison_map_for_modules_an_equivalence}
Let $X$ be a spherical presheaf and let $(E \rightarrow X) \in \abeliangroups(\presheaves_{\Sigma}(C_{/X}))$ be a module. Then the canonical comparison map 
\[
E \rightarrow X \times_{X_{\leq 0}} E_{\leq 0}
\]
is an equivalence. 
\end{lemma}

\begin{proof}
This is a levelwise statement, so it is enough to show that for any $c \in \inftyc$, the canonical comparison map 
\[
E(c) \rightarrow X(c) \times_{X_{\leq 1}(c)} E_{\leq 1}(c)
\] 
is an equivalence of spaces. As this is a map of spaces over $X(c)$, it is enough to verify that it induces an equivalence between the fibres. As this map can be lifted to a morphism of group objects in the homotopy category of spaces by \cref{prop:spherical_presheaves_take_values_in_h_spaces}, it is enough to check that the map between the fibres over the basepoint is an equivalence. 

Since $E(c) \rightarrow X(c)$ admits a zero section $X(c) \rightarrow E(c)$ as part of the module structure, the long exact sequences of homotopy groups of the fibres split into short exact sequences and yield isomorphisms
\[
\pi_{0} (\mathrm{fib}(E_{\leq 0}(c) \rightarrow X_{\leq 0}(c))) \simeq \mathrm{ker}(\pi_{0}E_{\leq 0}(c) \rightarrow \pi_{0} X_{\leq 0}(c)) \simeq \pi_{0} (\mathrm{fib}(E(c) \rightarrow X(c)))
\]
so that 
\[
\mathrm{fib}(E(c) \rightarrow X(c)) \rightarrow \mathrm{fib}(E_{\leq 0}(c) \rightarrow X_{\leq 0}(c))
\]
is an isomorphism on $\pi_{0}$. As both sides are discrete, the right hand side as it is a fibre of a map of discrete spaces, and the left hand side as it is a fibre of a discrete abelian group object, we deduce that this map is an equivalence, ending the argument. 
\end{proof}

\begin{lemma}
\label{lemma:0_truncation_pullback_comparison_is_an_equivalence}
Let $X$ be a spherical presheaf and let $(E \rightarrow X) \in \abeliangroups(\presheaves_{\Sigma}(C_{/X}))$ be a module. Then the canonical map 
\begin{equation}
\label{equation:comparison_map_with_pullback_of_truncations}
(E \times_{X} E)_{\leq 0} \rightarrow E_{\leq 0} \times_{X_{\leq 0}} E_{\leq 0}
\end{equation}
is an equivalence. Thus, $E_{\leq 0} \rightarrow X_{\leq 0}$ has a canonical structure of a module. 
\end{lemma}

\begin{proof}
As $X \rightarrow X_{\leq 0}$ is levelwise surjective on path components and hence an effective epimorphism of presheaves, it is enough to verify that the base-change 
\begin{equation}
\label{equation:base_change_of_comparison_map}
X \times_{X_{\leq 0}} (E \times_{X} E)_{\leq 0} \rightarrow X \times_{X_{\leq 0}} (E_{\leq 0} \times_{X_{\leq 0}} E_{\leq 0})
\end{equation}
is an equivalence. For the left hand side of (\ref{equation:base_change_of_comparison_map}), as $E \times_{X} E$ is a module over $X$, \cref{lemma:truncation_comparison_map_for_modules_an_equivalence} implies that 
\[
X \times_{X_{\leq 0}} (E \times_{X} E)_{\leq 0} \simeq E \times_{X} E. 
\]
For the right hand side of (\ref{equation:base_change_of_comparison_map}), two applications of \cref{lemma:truncation_comparison_map_for_modules_an_equivalence} yield 
\[
X \times_{X_{\leq 0}} (E_{\leq 0} \times_{X_{\leq 0}} E_{\leq 0}) \simeq (X \times_{X_{\leq 0}} E_{\leq 0}) \times_{X} (X \times_{X_{\leq 0}} E_{\leq 0}) \simeq E \times_{X} E.
\]
Chasing through these canonical maps, we see that these identifications are compatible with each other, so that (\ref{equation:base_change_of_comparison_map}) is an equivalence as needed, and hence so is (\ref{equation:comparison_map_with_pullback_of_truncations}). 

The second claim is that $E_{\leq 0} \rightarrow X_{\leq 0}$ has a canonical structure of a module. By what we've shown above, the limit comparison map of (\ref{equation:comparison_map_with_pullback_of_truncations}) is an equivalence and hence the construction $E \mapsto E_{\leq 0}$ defines a strongly symmetric monoidal functor 
\[
\abeliangroups(\presheaves_{\Sigma}(\inftyc _{/X})) \rightarrow \presheaves_{\Sigma}(\inftyc _{/X_{\leq 0}})
\]
where we consider both sides with the cartesian symmetric monoidal structure. It follows that the functor has a canonical lift to 
\[
\abeliangroups(\presheaves_{\Sigma}(\inftyc _{/X})) \simeq \abeliangroups(\abeliangroups(\presheaves_{\Sigma}(\inftyc _{/X}))) \rightarrow \abeliangroups(\presheaves_{\Sigma}(\inftyc _{/X_{\leq 0}})),
\]
which is what we wanted to show. 
\end{proof}

\begin{warning}
For a general map of spaces $Y \rightarrow X$, even one with discrete fibres, it is not in general true that
\[
(Y \times_{X} Y)_{\leq 0} \rightarrow Y_{\leq 0} \times_{X_{\leq 0}} Y_{\leq 0}
\]
is an equivalence, giving \cref{lemma:0_truncation_pullback_comparison_is_an_equivalence} some bite. For example, for a non-zero group $G$ we have
\[
G \simeq G_{\leq 0} \simeq (\mathrm{pt} \times_{\B G} \mathrm{pt}) _{\leq 0} \not\simeq (\mathrm{pt}_{\leq 0} \times_{(\B G)_{\leq 0}} \mathrm{pt}_{\leq 0}) \simeq \mathrm{pt} \times_{\mathrm{pt}} \mathrm{pt} \simeq \mathrm{pt}.
\]
Here, the issue is that $Y \rightarrow X$ does not admit a section. 

For a more complicated example, consider $C_{2} \simeq \Aut_{\abeliangroups}(\mathbb{Z})$ and the map $(\mathbb{Z})_{h C_{2}} \rightarrow \B C_{2}$, which is an abelian group object in $\spaces_{/ \B C_{2}}$ classified by the functor $\B C_{2} \rightarrow \abeliangroups$ picking up $\mathbb{Z}$ with the tautological action of its automorphism group. Then 
\[
(\mathbb{Z} \times \mathbb{Z})_{C_{2}} \simeq (\mathbb{Z}_{h C_{2}} \times_{\B C_{2}} \mathbb{Z}_{h C_{2}})_{\leq 0} \not\simeq (\mathbb{Z}_{h C_{2}})_{\leq 0} \times_{(\B C_{2})_{\leq 0}} (\mathbb{Z}_{h C_{2}})_{\leq 0} \simeq \mathbb{Z}_{C_{2}} \times \mathbb{Z}_{C_{2}}, 
\]
where $C_{2}$ denotes the classical orbits of the action. A counterexample of this form is avoided in the proof of \cref{lemma:0_truncation_pullback_comparison_is_an_equivalence} through the use of grouplike $H$-space structures. Here, $\mathbb{Z}_{h C_{2}}$ cannot be made into a grouplike $H$-space, as it has non-equivalent path components. 
\end{warning}

\begin{theorem}
\label{theorem:equivelence_of_modules_with_zero_truncation}
If $X \in \presheaves_{\Sigma}(\inftyc)$ is a spherical presheaf, then functors
\[
L(E \rightarrow X) := (E_{\leq 0} \rightarrow X_{\leq 0})
\]
and 
\[
R(F \rightarrow X_{\leq 0}) := (X \times_{X_{\leq 0}} F \rightarrow X)
\]
define an adjoint equivalence
\[
L \colon \Mod(X) \rightleftarrows \Mod(X_{\leq 0}) \colon R
\]
between the categories of modules over $X$ and its underlying $\inftyc$-algebra $X_{\leq 0}$. 
\end{theorem}

\begin{proof}
First, note that $L$ gives a functor between module categories as a consequence of \cref{lemma:0_truncation_pullback_comparison_is_an_equivalence}. The counit map of $E \in \Mod(X)$ is given by canonical the comparison map 
\[
E \rightarrow X \times_{X_{\leq 0}} E_{\leq 0}
\]
which is an equivalence by \cref{lemma:0_truncation_pullback_comparison_is_an_equivalence}. Thus, to prove that $L \dashv R$ is an adjoint equivalence, it is enough to check that $R$ is conservative. This is clear, since it is given by pulling back along an effective epimorphism $X \rightarrow X_{\leq 0}$. 
\end{proof}

Note that as a consequence of \cref{theorem:equivelence_of_modules_with_zero_truncation}, we see that if $Y \rightarrow X$ is a morphism of spherical presheaves, then the relative homotopy groups of \cref{defin:relative_homotopy_groups_of_presheaves} groups can be identified with modules over the underlying $\inftyc$-algebra, and hence form a purely algebraic invariant. A particularly important example is that of absolute homotopy groups:

\begin{example}
\label{example:action_on_absolute_homotopy_groups_for_spherical_presheaves}
If $X$ is a spherical presheaf, then for any $n \geq 1$ the abelian $\inftyc$-algebra underlying $X_{\leq 0}$-module $\pi_{n}(X \rightarrow X_{\leq 0})$ can be identified with $\pi_{n} X$. This module structure can be thought of as encoding the action of $\pi_{0} X$ on the higher homotopy groups. 
\end{example}

\begin{warning}
In the context of \cref{example:action_on_absolute_homotopy_groups_for_spherical_presheaves}, beware the degree shift from the case of spaces reviewed in \S\ref{subsection:action_of_the_fundamental_group}, where it was the $\pi_{1}$ which was acting on the higher homotopy groups, rather than $\pi_{0}$!. This degree shift stems from the semi-algebraic behaviour of the $\infty$-category of spherical presheaves and crucially depends on \cref{assumption:assumption_on_c_to_have_good_presheaves}. 
\end{warning}

\subsection{Postnikov invariants of spherical presheaves} 

In this section analogous to \S\ref{subsection:postnikov_invariants_of_spaces}
and \S\ref{subsection:action_of_the_fundamental_group}, we develop the theory of Postnikov invariants of spherical presheaves. As in the case of spaces, the main technical tool will be an appropriate Eilenberg-MacLane object. The following is the presheaf analogue of \cref{defin:eilenberg_maclane_space}: 

\begin{defin}
\label{defin:eilenberg_maclane_presheaf}
Let $\Lambda$ be a $\inftyc$-algebra, $M$ be a $\Lambda$-module and $n \geq 0$. The \emph{Eilenberg-MacLane presheaf} of type $(\Lambda, M, n+1)$, denoted by $\B^{\inftyc}_{\Lambda}(M, n+1)$, is defined by the cartesian diagram
\begin{center}
	\begin{tikzpicture}
		\node (TL) at (0, 1.5) {$ \B^{\inftyc}_{\Lambda}(M, n+1)$};
		\node (BL) at (0, 0) {$ \Lambda  $};
		\node (TR) at (3, 1.5) {$ \mathcal{M}_{\inftyc}(\emspaces_{n}) $};
		\node (BR) at (3, 0) {$ \mathcal{M}_{\inftyc}(\abeliangroups) $};
		
		\draw [->] (TL) to (TR);
		\draw [->] (TL) to (BL);
		\draw [->] (TR) to (BR);
		\draw [->] (BL) to (BR);
	\end{tikzpicture},
\end{center}
where 
\begin{enumerate}
    \item $\mathcal{M}_{\inftyc}(\abeliangroups)$ is the moduli presheaf for abelian groups of \cref{prop:existence_of_moduli_of_discrete_abelian_groups}, 
    \item $\mathcal{M}_{\inftyc}(\emspaces_{n})$ is the moduli presheaf for abelian Eilenberg-MacLane spaces of degree $n$ of \cref{prop:existence_of_moduli_of_em_spaces_in_presheaves},
    \item the bottom horizontal map is the classifying map for the module $M$ and 
    \item the right vertical map is induced by taking an Eilenberg-MacLane space to its $n$-th homotopy group, which is well-defined by simplicity. 
\end{enumerate}
\end{defin}

\begin{rem}
Note that, a priori, $\B^{\inftyc}_{\Lambda}(M, n+1)$ is an object of $\Fun(\inftyc^{op}, \largespaces)$, the $\infty$-category of presheaves of large spaces. As a consequence of \cref{cor:eilenberg_maclane_presheaf_is_spherical} below, it is in fact an object of $\presheaves_{\Sigma}(\inftyc)$.
\end{rem}

\begin{observation}[Universal property]
\label{observation:universal_property_of_em_presheaf}
As in the case of spaces discussed in \cref{observation:moduli_interpretation_of_eilenberg_maclane_spaces}, by construction the presheaf $\B^{\inftyc}_{\Lambda}(M, n+1)$ has a universal property. Namely, for any $Y \in \presheaves_{\Sigma}(\inftyc)$, the space $\map(Y, \B^{\inftyc}(M, n+1))$ is the $\infty$-groupoid of the following triples: 

\begin{enumerate}
\item a map of presheaves $X \rightarrow Y$ levelwise fibred in Eilenberg-MacLane spaces of degree $n$, 
\item a morphism $\delta \colon Y \rightarrow \Lambda$ and 
\item an isomorphism $\epsilon \colon  \pi_{n}(X \rightarrow Y) \simeq \delta^{*}M$ of modules over $Y$. 
\end{enumerate}
Notice that \cref{lem:sphericity_detected_by_modules} implies that in this situation $X$ is also spherical.
\end{observation}

Using the universal property of \cref{observation:universal_property_of_em_presheaf}, we develop below a classification result for $n$-types in terms of maps into the Eilenberg-MacLane presheaf. However, such a result wouldn't be very useful if the presheaf $B_{\Lambda}^{\inftyc}(M, n+1)$ was very complicated. This is not the case; in fact, its homotopy groups admit a simple description analogous to \cref{rem:homotopy_of_an_eilenberg_maclane_space}:

\begin{prop}
\label{prop:description_of_the_eilenberg_maclane_presheaf}
The map $\B^{\inftyc}_{\Lambda}(M, n+1) \rightarrow \Lambda$ has the following three poperties: 
\begin{enumerate}
\item it is levelwise fibred in Eilenberg-MacLane spaces of degree $(n+1)$,
\item there exists an isomorphism $\pi_{n+1}(B^{\inftyc}_{\Lambda}(M, n+1) \rightarrow \Lambda) \simeq M$ of $\Lambda$-modules and 
\item there exists a section $\Lambda \rightarrow \B^{\inftyc}_{\Lambda}(M, n+1)$.
\end{enumerate}
Moreover, any map with these three properties is equivalent to $\B^{\inftyc}_{\Lambda}(M, n+1)$ as an object of $\presheaves_{\Sigma}(\inftyc)_{/\Lambda}$. 
\end{prop}

\begin{proof}
As $\B^{\inftyc}_{\Lambda}(M, n+1)$ is defined as a pullback along the map 
\begin{equation}
\label{equation:map_of_big_moduli_spaces}
\mathcal{M}_{\inftyc}(\emspaces_{n}) \rightarrow \mathcal{M}_{\inftyc}(\abeliangroups),
\end{equation}
it is enough to show that the latter has the three analogous properties. After evaluating at $c \in \inftyc$, (\ref{equation:map_of_big_moduli_spaces}) is equivalent to the map 
\begin{center}
$\fun((\inftyc_{/c})^{op}, \emspaces_{n})^{\cong} \rightarrow \fun((\inftyc_{/c})^{op}, \abeliangroups)^{\cong}$
\end{center}
induced by taking $\pi_{n}$. This clearly admits a section given by $\B^{n}(-)$, the iterated classifying space functor. This shows (3). 

Now suppose that we are given a point
\[
M \in \fun((\inftyc_{/c})^{op}, \abeliangroups)^{\cong} \simeq \mathcal{M}_{\inftyc}(\abeliangroups)(c).
\]
We have to show that the fibre over $M$ is an Eilenberg-MacLane space of degree $(n+1)$. This fibre can be described as the $\infty$-groupoid of pairs 
\[
(X \in \fun((\inftyc_{/c})^{op}, \emspaces_{n})^{\cong}, \phi \colon \pi_{n} X \simeq M)
\]
of a presheaf together with a trivialization of its $\pi_{n}$. By \cref{corollary:presheaves_of_em_spaces_described_by_homotopy_groups_up_to_equivalence}, any two such $X$ are equivalent, as $\inftyc_{/c}$ has a terminal object, and by \cref{proposition:automorphism_of_em_presheaves} the $\infty$-group of automorphisms acting trivially on $\pi_{n}$ of any such presheaf is equivalent to $\B^{n}M(\mathrm{id}_{c})$. Thus, the fibre is a space of type $\B^{n+1}M(\mathrm{id}_{c})$, which shows $(1)$ and $(2)$. 

To see that $\mathcal{M}_{\inftyc}(\emspaces_{n}) \rightarrow \mathcal{M}_{\inftyc}(\abeliangroups)$ is, up to equivalence, the only map satisfying these two properties, observe that a choice of section equips the levelwise fibres with a choice of a basepoint. As a consequence of  \cref{prop:projectivity_of_em_spaces}, the functor $\pi_{n+1} \colon (\emspaces_{n+1})_{*} \rightarrow \abeliangroups$ is an equivalence. It follows that maps with fibres given by Eilenberg-MacLane spaces of degree $(n+1)$ and equipped with a choice of a section are uniquely determined by their relative $\pi_{n+1}$. This ends the argument. 
\end{proof}

\begin{cor}
\label{cor:eilenberg_maclane_presheaf_is_spherical}
The presheaf $\B^{\inftyc}_{\Lambda}(M, n+1)$ is spherical and its only non-zero homotopy groups are $\pi_{0} \simeq \Lambda$ and $\pi_{n+1} \simeq M$. 
\end{cor}

\begin{proof}
The description of homotopy groups is immediate from \cref{prop:description_of_the_eilenberg_maclane_presheaf}. Sphericity of $\B^{\inftyc}_{\Lambda}(M, n+1)$ follows from \cref{lem:sphericity_detected_by_modules}. 
\end{proof}

As we noted in \cref{observation:universal_property_of_em_presheaf}, a map $Y \rightarrow \B^{\inftyc}_{\Lambda}(M, n+1)$ determines, among other things, a map of presheaves $X \rightarrow Y$ which is levelwise fibred in Eilenberg-MacLane spaces of degree $n$. By the Yoneda lemma, this correspondence is determined by pulling back along a universal map $\mathrm{E}^{\inftyc}_{\Lambda}(M, n+1)) \rightarrow Y \rightarrow \B^{\inftyc}_{\Lambda}(M, n+1)$. The following gives an explicit description of this family: 

\begin{prop}
\label{prop:universal_family_over_moduli_of_n_types}
Let $\mathrm{E}^{\inftyc}_{\Lambda}(M, n+1) \rightarrow  \B^{\inftyc}_{\Lambda}(M, n+1)$ denote the canonical morphism levelwise fibred in Eilenberg-MacLane spaces of degree $n$. Then the composite 
\[
\mathrm{E}^{\inftyc}_{\Lambda}(M, n+1) \rightarrow \B^{\inftyc}_{\Lambda}(M, n+1) \rightarrow \Lambda 
\]
is an equivalence. 
\end{prop}

\begin{proof}
Using \cref{observation:universal_property_of_em_presheaf}, one can describe the space $\map(Y, \mathrm{E}^{\inftyc}_{\Lambda}(M, n+1))$ as the $\infty$-groupoid of following triples: 

\begin{enumerate}
\item a map of presheaves $X \rightarrow Y$ levelwise fibred in Eilenberg-MacLane spaces of degree $n$, 
\item a morphism $\delta \colon Y \rightarrow \Lambda$ and 
\item an isomorphism $\epsilon \colon \pi_{n}(X \rightarrow Y) \simeq \delta^{*}M$ of modules over $Y$. 
\item a section $Y \rightarrow X$. 
\end{enumerate}

A choice of a section of $X \rightarrow Y$ lifts it to an object of  
\[
(\presheaves(\inftyc)_{/Y})_{*} \simeq \presheaves(\inftyc_{/Y})_{*} \simeq \Fun(\inftyc^{op}, \spaces_{*}), 
\]
and since the fibres are Eilenberg-MacLane spaces, it is in fact an object of
\[
\Fun(\inftyc^{op}, (\emspaces_{n})_{*}).
\]

By \cref{prop:projectivity_of_em_spaces}, the functor $\B^{n} \colon \abeliangroups \rightarrow (\emspaces_{n})_{*}$ is an equivalence, so that the automorphism group of $X \in (\presheaves(\inftyc)_{/Y})_{*}$ coincides with the automorphism group of $\pi_{n}(X \rightarrow Y)$ as a $Y$-module. It follows that if a morphism $\delta \colon Y \rightarrow Y \rightarrow \Lambda$ is fixed, then the space of choices of $(1), (3)$ and $(4)$ is contractible. As in terms of the functors they represent
\[
\mathrm{E}^{\inftyc}_{\Lambda}(M, n+1) \rightarrow \Lambda
\]
corresponds to forgettting $(1), (3)$ and $(4)$, we deduce that this composite is an equivalence. 
\end{proof}

We proceed with the promised classification of $n$-types. We begin by fixing a $\inftyc$-algebra $\Lambda$ which plays the role previously played by the fundamental group.

\begin{defin}
\label{defin:spherical_presheaf_of_type_as_lambda_module}
Let $Y \in \presheaves_{\Sigma}(\inftyc)$ be an $(n-1)$-type, where $n \geq 1$, and let $M$ be a $\Lambda$-module. We say that $X \in \presheaves_{\Sigma}(\inftyc)$ is \emph{of type} $Y +_{\Lambda} (M, n)$ if there exists an equivalence $X_{\leq (n-1)} \simeq Y$ and an isomorphism $Y_{\leq 0} \simeq \Lambda$ under which $\pi_{n}(X \rightarrow Y) \simeq M$ as $\Lambda$-modules.
\end{defin}

\begin{notation}
 We denote by $\mathcal{M}(Y +_{\Lambda} (M, n))$ the full subgroupoid of $\presheaves_{\Sigma}(\ccat)$ spanned by spherical presheaves of type $Y +_{\Lambda} (M, n)$.
\end{notation}
Our goal is to give a description of $\mathcal{M}(Y +_{\Lambda} (M, n))$ in terms of $Y$, $M$ and $\Lambda$. 

\begin{defin}
\label{defin:automorphism_of_pair_of_c_algebra_and_module}
An \emph{automorphism} of the pair $(\Lambda, M)$ is a pair $(\alpha, \beta)$, where $\alpha \colon \Lambda \rightarrow \Lambda$ is an isomorphism of $\inftyc$-algebras and $\beta \colon M \rightarrow \alpha^{*}M$ is an isomorphism of modules. We denote the group of such automorphisms by $\Aut(\Lambda, M)$. 
\end{defin}

\begin{rem}[Action on the Eilenberg-MacLane presheaf]
\label{rem:action_of_pair_autos_on_em_presheaf}
The group of \cref{defin:automorphism_of_pair_of_c_algebra_and_module} can also be described as the group of automorphisms of $\Lambda$, considered as a presheaf over $\mathcal{M}_{\inftyc}(\abeliangroups)$. Since $\B^{\inftyc}_{\Lambda}(M, n+1)$ is defined as a pullback along the map $\Lambda \rightarrow \mathcal{M}_{\inftyc}(\abeliangroups)$, the Eilenberg-MacLane presheaf carries a canonical action of the group $\Aut(\Lambda, M)$, inducing an action on $\map(Y, \B^{\inftyc}_{\Lambda}(M, n+1))$ for any $Y \in \presheaves_{\Sigma}(\inftyc)$. 

In terms of the moduli interpretation of \cref{observation:universal_property_of_em_presheaf}, the action is given as follows: an element $(\alpha, \beta) \in \Aut(\Lambda, M)$ acts on a triple 
\[
(X \rightarrow Y, \delta, \epsilon) \in \map(Y, \B^{\inftyc}_{\Lambda}(M, n+1))
\]
by sending it to $(X \rightarrow Y, \alpha \circ \delta, \beta \circ \epsilon)$.
\end{rem}
The main result of this section will be a presheaf analogue of \cref{thm:classification_of_n_types_fixed_module}, asserting the existence of a canonical fibre sequence of moduli spaces. Analogously to \cref{construction:fibre_sequence_of_moduli_spaces_for_n_types}, we first define the two relevant maps and a nullhomotopy of their composite: 

\begin{construction}
\label{construction:potential_fibre_sequence_of_moduli_spaces_of_presheaves}
Let $Y$ be an $(n-1)$-type in $\presheaves_{\Sigma}(\inftyc)$, $\Lambda$ be a $\inftyc$-algebra and let $M$ be a $\Lambda$-module. We will construct a diagram of spaces
\[
\map_{0-\mathrm{eq}}(Y, \B^{\inftyc}_{\Lambda}(M, n+1)) \rightarrow \mathcal{M}(Y +_{\Lambda} (M, n)) \rightarrow \mathcal{M}(Y) \times \B\Aut(\Lambda, M).
\]
together with the nullhomotopy of the composite map, where $\map_{0-\mathrm{eq}}$ denotes the space of $0$-equivalences and $\mathcal{M}(Y)$ is the $\infty$-groupoid of presheaves equivalent to $Y$. 

We first define the right map, proceeding as in the proof of \cref{thm:classification_of_n_types_fixed_module}. As we're mapping into a product, it is enough to specify a map into each factor. For the first factor, the needed map $\mathcal{M}(Y +_{\Lambda} (M, n)) \rightarrow \mathcal{M}(Y)$ takes $X$ to $X_{\leq (n-1)}$. For the second factor, notice that $\B\Aut(\Lambda, M)$ can be identified with the groupoid of $\inftyc$-algebras over $\mathcal{M}_{\inftyc}(\abeliangroups)$ equivalent to the map $\Lambda \rightarrow \mathcal{M}_{\inftyc}(\abeliangroups)$ classifying $M$. In these terms, the needed map $\mathcal{M}(Y +_{\Lambda} (M, n)) \rightarrow \B\Aut(\Lambda, M)$ takes $X$ to $X_{\leq 0}$ together with the classifying map for the module $\pi_{n}X$.  

To define the left map we use the universal property of the Eilenberg-MacLane presheaf of \cref{observation:universal_property_of_em_presheaf}, sending an arrow $Y \rightarrow  \B^{\inftyc}_{\Lambda}(M, n+1))$ to the corresponding arrow $X \rightarrow Y$ which is levelwise fibred in Eilenberg-MacLane spaces. 

We move on to the nullhomotopy; again, as we're mapping into a product, we can define nullhomotopies on each factor separately. For the nullhomotopy of 
\[
\map_{0-\mathrm{eq}}(Y, \B^{\inftyc}_{\Lambda}(M, n+1)) \rightarrow \mathcal{M}(Y),
\]
notice that the map $X \rightarrow Y$ induces a canonical equivalence $X_{\leq n-1} \simeq Y$. The nullhomotopy of 
\[
\map_{0-\mathrm{eq}}(Y, \B^{\inftyc}_{\Lambda}(M, n+1)) \rightarrow \B\Aut(\Lambda, M)
\]
is determined by the structure isomorphisms $\pi_{n}(X \rightarrow Y) \simeq M$ and $\pi_{0} X \simeq \pi_{0} Y \simeq \Lambda$ coming from the universal property of $\B^{\inftyc}_{\Lambda}(M, n+1))$. 
\end{construction}

\begin{theorem}
\label{thm:classification_of_n_type_presheaves_fixed_module}
The sequence 
\[
\map_{0-\mathrm{eq}}(Y, \B^{\inftyc}_{\Lambda}(M, n+1)) \rightarrow \mathcal{M}(Y +_{\Lambda} (M, n)) \rightarrow \mathcal{M}(Y) \times \B\Aut(\Lambda, M)
\]
of \cref{construction:potential_fibre_sequence_of_moduli_spaces_of_presheaves} is a fibre sequence. 
\end{theorem}

\begin{proof}
After consulting the definition of the right map, we see that its fibre can be identified with the $\infty$-groupoid of triples of 
\begin{enumerate}
    \item a presheaf $X$ of type $Y +_{\Lambda} (M, n)$
    \item a fixed equivalence $X_{\leq n-1} \simeq Y$ and 
    \item isomorphisms $\pi_{0}X \simeq \Lambda$ and $\pi_{n}X \simeq M$, the latter one being an isomorphism of $\Lambda$-modules. 
\end{enumerate}
As $X$ is an $n$-type, the fixed equivalence $X_{\leq n-1} \simeq Y$ determines and is determined by map $X \rightarrow Y$ which is levelwise fibred in Eilenberg-MacLane spaces of degree $n$. \cref{observation:universal_property_of_em_presheaf} identifies this $\infty$-groupoid of triple with the required space of $0$-equivalences, as required. 
\end{proof}

\begin{cor}
\label{cor:postnikov_invariant_of_a spherical_presheaf_in_terms_of_homotopy_classes}
An equivalence class of a presheaf of type $Y +_{\Lambda} (M, n)$ is determined by a homotopy class of $0$-equivalences $Y \rightarrow \B^{\inftyc}_{\Lambda}(M, n+1)$ well-defined up to the action of $\Aut_{\presheaves_{\Sigma}(\inftyc)}(Y) \times \Aut(\Lambda, M)$. 
\end{cor}

\begin{proof}
Since $\mathcal{M}(Y) \simeq \B \Aut_{\presheaves_{\Sigma}(\inftyc)}(Y)$, using \cref{observation:homotopy_quotient_by_an_infty_group} we can rephrase \cref{thm:classification_of_n_type_presheaves_fixed_module} as giving an equivalence 
\[
\mathcal{M}(Y +_{\Lambda} (M, n)) \simeq \map_{0-\mathrm{eq}}(Y, \B^{\inftyc}_{\Lambda}(M, n+1))_{h (\Aut_{\presheaves_{\Sigma}(\inftyc)}(Y) \times \Aut(\Lambda, M))}.
\]
Passing to path components yields the desired statement. 
\end{proof}

\begin{rem}
As we observed in the proof of \cref{cor:postnikov_invariant_of_a spherical_presheaf_in_terms_of_homotopy_classes}, the fibre sequence of \cref{thm:classification_of_n_type_presheaves_fixed_module} encodes an action of the $\infty$-group $\Aut_{\presheaves_{\Sigma}(\inftyc)}(Y) \times \Aut(\Lambda, M)$ on $\map_{0-\mathrm{eq}}(Y, \B^{\inftyc}_{\Lambda}(M, n+1))$. By retracing the definitions of the relevant maps, one can verify that $\Aut_{\presheaves_{\Sigma}(\inftyc)}(Y)$ acts on the mapping space by acting on the source and that $\Aut(\Lambda, M)$ acts on the target through the action of \cref{rem:action_of_pair_autos_on_em_presheaf}.
\end{rem}

\subsection{Homotopy presheaves}

In this section we make results of the previous section more computationally accessible by showing that under certain assumptions, one can reduce from $\inftyc$ to its homotopy category $h\inftyc$. 

We continue working under \cref{assumption:assumption_on_c_to_have_good_presheaves}, so that the base $\infty$-category $\inftyc$ is pointed, has finite coproducts, and every object $c \in \inftyc$ admits a cogroup structure in the homotopy category $h \inftyc$. Note that it is clear that if $\inftyc$ satisfies these properties, so does its homotopy category. Thus, all of the results proven in previous sections apply to $\presheaves_{\Sigma}(h\inftyc)$ directly. To distinguish the latter from $\presheaves_{\Sigma}(\inftyc)$ we adopt the following convention:

\begin{defin}
\label{defin:spherical_homotopy_presheaves}
A \emph{spherical homotopy presheaf} is a product-preserving functor  
\[
X \colon h \inftyc^{op} \rightarrow \spaces.
\]
We denote the $\infty$-category of spherical homotopy presheaves by $\presheaves_{\Sigma}(h \inftyc)$. 
\end{defin}

\begin{warning}
Notice that the word ``homotopy'' in \cref{defin:spherical_homotopy_presheaves} refers to the fact that the homotopy presheaves are indexed by the homotopy category $h\inftyc$. The presheaves themselves are still valued in spaces! In particular, $\presheaves_{\Sigma}(h\inftyc)$ is very different from $h (\presheaves_{\Sigma}(\inftyc))$, the homotopy category of $\presheaves_{\Sigma}(\inftyc)$, which we will not discuss. 
\end{warning}

The following two remarks explain the algebraic nature of homotopy presheaves. 
\begin{rem}[Homotopy presheaves are animated $\inftyc$-algebras]
\label{rem:homotopy_presheaves_and_animated_inftyc_algebras}
In the terminology of Cesnavicius and Scholze \cite[\S 5.1.4]{cesnavicius2019purity}, the animation of a presentable category $\dcat$ generated by its subcategory $\dcat^{cp} \subseteq \dcat$ of compact projective objects is the universal $\infty$-category $\dcat^{an}$ freely generated by $\dcat^{cp}$ under colimits. As a consequence of the universal property of the $\infty$-category of spherical presheaves of \cref{prop:universal_property_of_the_nonabelian_derived_category}, there is a canonical equivalence $\dcat^{an} \simeq \presheaves_{\Sigma}(\dcat^{cp})$.  

In the case at hand, the category $\inftyc\mhyphen\Alg \simeq \presheaves_{\Sigma}(\inftyc; \sets)$ of $\inftyc$-algebras is generated by its subcategory of compact projectives, which up to idempotent completion coincides with $h \inftyc$. It follows that there is a canonical equivalence 
\[
\presheaves_{\Sigma}(h \inftyc) \simeq (\inftyc \mhyphen \Alg)^{an}
\]
between spherical homotopy presheaves and the $\infty$-category of animated $\inftyc$-algebras. With the exception of the introduction, we decided not to phrase our results in terms of animation, to emphasize the analogy between $\presheaves_{\Sigma}(\inftyc)$ and $\presheaves_{\Sigma}(h \inftyc)$. 
\end{rem}

\begin{rem}[Homotopy presheaves and simplicial $\inftyc$-algebras]
\label{rem:homotopy_presheaves_and_simplicial_c_algebras}
Consider the category
\[
s \inftyc \mhyphen \Alg \simeq \presheaves_{\Sigma}(h \inftyc; s \sets)
\]
of simplicial $\inftyc$-algebras. If $X_{\bullet} \in s \inftyc \mhyphen \Alg$, then it defines a spherical homotopy presheaf $|X_{\bullet}|$ by the formula
\[
|X_{\bullet}|(c) := | X_{\bullet}(c) |,
\]
where the geometric realization is computed in spaces. By a result of Bergner, see \cite{bergner_rigidification_of_algebras} or \cite[5.5.9.2]{lurie_higher_topos_theory}, the resulting functor is a localization of $\infty$-categories and hence induces an equivalence
\[
(s \inftyc \mhyphen \Alg)[W^{-1}] \simeq \presheaves_{\Sigma}(h \inftyc),
\]
where $W$ is the class of maps $X_{\bullet} \rightarrow Y_{\bullet}$ such that $|X_{\bullet}(c)| \rightarrow |Y_{\bullet}(c)|$ is an equivalence of spaces for any $c \in h\inftyc$. 
\end{rem}

\begin{notation}
To relate spherical presheaves to spherical homotopy presheaves, we make us the of projection onto the homotopy category, which we denote by
\[
p  \colon \inftyc \rightarrow h\inftyc.
\]
\end{notation}

\begin{prop}
\label{prop:adjunction_between_spherical_preshaves_and_homotopy_spherical_presheaves}
There exists an adjoint pair
\[
p^{*} \colon \presheaves_{\Sigma}(\inftyc) \rightleftarrows \presheaves_{\Sigma}(h\inftyc) \colon p_{*},
\]
where $p_{*}$ is given by precomposition and $p^{*}$ is the unique colimit preserving functor extending $p$ on representable presheaves. 
\end{prop}

\begin{proof}
The extension $p^{*}$ exists by the universal property of the $\infty$-category of spherical presheaves, which we have stated as \cref{prop:universal_property_of_the_nonabelian_derived_category}. Since $p  \colon \inftyc \rightarrow h\inftyc$ preserves coproducts, \cref{observation:cocontinuity_of_the_derived_functor} implies that $p^{*}$ is cocontinous, hence a left adjoint. One verifies immediately that the right adjoint is given by $p_{*}$. 
\end{proof}

\begin{observation}
\label{observation:right_adjoint_preserves_sifted_colimits}
By \cref{observation:limits_colimits_in_spherical_presheaves}, filtered colimits and geometric realizations are computed in spherical preshaves levelwise. Since $p_{*}$ is given by precomposition, it follows that, in addition to preserving limits, it also preserves filtered colimits and geometric realizations.
\end{observation}

\begin{rem}
The left adjoint $p^{*} \colon \presheaves_{\Sigma}(\inftyc) \rightarrow \presheaves_{\Sigma}(\inftyc)$ can be described as the restriction of the left Kan extension $\presheaves(\inftyc) \rightarrow \presheaves(h\inftyc)$, which preserves spherical presheaves by \cref{observation:limits_colimits_in_spherical_presheaves}.
\end{rem}

Recall from \cref{defin:c_algebra} that a $\inftyc$-algebra is a discrete object of $\presheaves_{\Sigma}(\inftyc)$. One expects that when working with discrete objects, there shouldn't be much difference between $\inftyc$ and $h\inftyc$. That is indeed the case, which we now make explicit:

\begin{prop}
\label{prop:inftyc_algebras_and_hinftyc_algebras_are_the_same}
The functor $p_{*} \colon \presheaves_{\Sigma}(h \inftyc) \rightarrow \presheaves_{\Sigma}(\inftyc)$ restricts to an equivalence between the categories of $h\inftyc$-algebras and $\inftyc$-algebras. 
\end{prop}

\begin{proof}
A $\inftyc$-algebra can be identified with a product-preserving functor $\Lambda \colon \inftyc^{op} \rightarrow \sets$. Since $\sets$ is an ordinary category, $\Lambda$ uniquely factors through the projection $p \colon \inftyc \rightarrow h \inftyc$; in other words, it's uniquely in the image of $p_{*}$. 
\end{proof}

\begin{notation}
Keeping in mind \cref{prop:inftyc_algebras_and_hinftyc_algebras_are_the_same}, we will in general not distinguish between $\inftyc$- and $h \inftyc$-algebras. 
\end{notation}

As $p_{*}$ is given by precomposition, $X \rightarrow Y$ is a morphism of spherical homotopy presheaves with the property of being levelwise fibred in Eilenberg-MacLane spaces, then $p_{*}X \rightarrow p_{*}Y$ is a morphism of spherical sheaves with the same property. Similarly, $p_{*}$ takes discrete abelian group objects to discrete abelian group objects. Applying these observations to the universal case; that is, to the case of moduli objects of \cref{prop:existence_of_moduli_of_discrete_abelian_groups} and \cref{prop:existence_of_moduli_of_em_spaces_in_presheaves}, we obtain canonical comparison maps
\[
p_{*} \mathcal{M}_{h\inftyc}(\emspaces_{n}) \rightarrow \mathcal{M}_{\inftyc}(\emspaces_{n})
\]
and
\[
p_{*} \mathcal{M}_{h\inftyc}(\abeliangroups) \rightarrow \mathcal{M}_{\inftyc}(\abeliangroups).
\]
These two are closely related in the following way: 
\begin{prop}
\label{prop:pullback_diagram_of_em_spaces}
For any $n \geq 1$, the diagram of moduli of abelian Eilenberg-MacLane spaces of degree $n$ and of moduli of discrete abelian groups

\begin{center}
	\begin{tikzpicture}
		\node (TL) at (0, 1.5) {$ p_{*} \mathcal{M}_{h\inftyc}(\emspaces_{n}) $};
		\node (BL) at (0, 0) {$ p_{*} \mathcal{M}_{h\inftyc}(\abeliangroups) $};
		\node (TR) at (3, 1.5) {$ \mathcal{M}_{\inftyc}(\emspaces_{n}) $};
		\node (BR) at (3, 0) {$ \mathcal{M}_{\inftyc}(\abeliangroups) $};
		
		\draw [->] (TL) to (TR);
		\draw [->] (TL) to (BL);
		\draw [->] (TR) to (BR);
		\draw [->] (BL) to (BR);
	\end{tikzpicture},
\end{center}
where the horizontal maps classify $p^{*}$ of the universal families and and the vertical maps are induced by taking $\pi_{n}$, is a cartesian diagram in $\Fun(\inftyc^{op}, \largespaces)$. 
\end{prop}

\begin{proof}
To check whether this diagram of presheaves is cartesian we only have to check whether it is levelwise cartesian. Thus, we can fix $c \in \inftyc$ and look a the corresponding diagram of $\infty$-groupoids. In this case, we have described the fibres of the vertical maps explicitly in the proof of \cref{prop:description_of_the_eilenberg_maclane_presheaf} and one verifies that they are the same. 
\end{proof}

\begin{cor}
\label{cor:equivalence_between_em_and_homotopy_em}
For any $n \geq 1$, a $\inftyc$-algebra $\Lambda$ and a $\Lambda$-module $M$, the canonical comparison map
\[
p_* \B^{h\inftyc}_{\Lambda}(M, n+1) \rightarrow \B^{\inftyc}_{\Lambda}(M, n+1)
\]
between Eilenberg-MacLane presheaves of \cref{defin:eilenberg_maclane_presheaf} is an equivalence. 
\end{cor}

\begin{proof}
This is immediate from the pullback pasting lemma and the fact that $p_{*}$ preserves limits, as the relevant Eilenberg-MacLane presheaves are defined by pullbacks along the vertical maps in the diagram of \cref{prop:pullback_diagram_of_em_spaces}.
\end{proof}

We note that in the case of spherical homotopy presheaves, the Eilenberg-MacLane presheaf is closely related to a very classical notion of cohomology:

\begin{rem}[Andr\'{e}-Quillen cohomology of $\inftyc$-algebras]
\label{rem:andre_quillen_cohomology_of_c_algebras_as_maps_into_em_presheaf}
The homotopy presheaf $\B_{\Lambda}^{h \inftyc}(M, n)$ can be used to define cohomology of $\inftyc$-algebras in the following way: if $K \rightarrow \Lambda$ is a map of $\inftyc$-algebras, we define the Andr\'{e}-Quillen cohomology of $K$ with coefficients in $M$ as 
\[
\mathrm{H}^{n}_{\Lambda}(K, M) := \pi_{0} \map_{\presheaves_{\Sigma}(h \inftyc)_{/\Lambda}}(K, \B_{\Lambda}^{h \inftyc}(M, n)), 
\]
the set of homotopy classes of maps into the Eilenberg-MacLane presheaf in the overcategory. This is the $\inftyc$-algebra analogue  of Andr\'{e}-Quillen cohomology as in \cite[Chapter II, \S 5]{quillen2006homotopical}, \cite{quillen1970co}; see also \cite[\S 6.7]{realization_space_of_a_pi_algebra} for the case of $\Pi$-algebras.
\end{rem}

\subsection{The spiral fibre sequence} 
\label{subsection:spiral}

In this section, we extend the results of the previous one by describing a stronger relation between the $\infty$-categories $\presheaves_{\Sigma}(\inftyc)$ of spherical presheaves and spherical homotopy presheaves $\presheaves_{\Sigma}(h \inftyc)$. 

A classical shadow of this relationship, going back to the work of Dwyer-Kan-Stover \cite{bigraded_homotopy_groups}, is known as the spiral long exact sequence, which relates two kinds of homotopy groups of a simplicial space. Our main goal is to show that this long exact sequence arises naturally from of a certain fibre sequence of presheaves. To do so we need to make a minor assumption on $\inftyc$ that makes the construction possible. 

Recall that we assumed that $\inftyc$ is a pointed $\infty$-category; in particular, all of its mapping spaces are pointed. In this context, we have the following classical notion: 

\begin{defin}
\label{defin:tensor_with_pointed_space}
If $U$ is a pointed space and $c \in \inftyc$, the \emph{tensor} $U \otimes c$ is an object of $\inftyc$ together with an equivalence
\[
\map_{\inftyc}(U \otimes c, c^\prime) \simeq \map_{\spaces_{*}}(U, \map_{\inftyc}(c, c^{\prime}))
\]
of spaces natural in $c^{\prime} \in \inftyc$. 
\end{defin}

\begin{example}
\label{example:examples_of_tensors_by_s1_and_s0}
For any $c \in \inftyc$, we have a canonical equivalence $S^{0} \otimes c \simeq c$. Similarly, assuming they exist, we have $S^{1} \otimes c \simeq \Sigma (c)$, where the right hand side is the suspension.  
\end{example}

\begin{warning}
Beware that the tensor of \cref{defin:tensor_with_pointed_space} is the reduced one; that is, it uses the canonical enrichment of $\inftyc$ in in pointed spaces $\spaces_{*}$, rather than the one in spaces. 
\end{warning}
Note that by the Yoneda lemma, if a tensor $U \otimes c$ exists then it is unique up to equivalence and functorial in both $c$ and $U$. In addition to the previous \cref{assumption:assumption_on_c_to_have_good_presheaves}, we will need the following additional assumption on our indexing $\infty$-category: 

\begin{assumption}
\label{assumption:inftyc_admits_tensors_by_the_circle}
The $\infty$-category $\inftyc$ admits tensors by $S^{1}$; that is, admits suspensions. 
\end{assumption}
As it turns out, combined with \cref{assumption:assumption_on_c_to_have_good_presheaves} the existence of suspensions implies the existence of other kinds of tensors. 

\begin{lemma}
\label{lemma:inftyc_satisfying_assumptions_admits_tensors_by_s1plus_which_also_split}
The $\infty$-category $\inftyc$ admits tensors by $S^{1}_{+}$. Moreover, for any $c \in \inftyc$ there exists a (non-canonical) map $S^{1} \otimes c \rightarrow S^{1}_{+} \otimes c$ which together with the canonical map
\[
c \simeq S^{0} \otimes c \simeq (\mathrm{pt})_{+} \otimes c \rightarrow S^{1}_{+} \otimes c
\]
induces an equivalence
\[
S^{1}_{+} \otimes c \simeq (S^{1} \otimes c) \vee c.
\]
\end{lemma}

\begin{proof}
Let $c \in \inftyc$ and consider its image under the Yoneda embedding $y \colon \inftyc \rightarrow \presheaves_{\Sigma}(\inftyc)$. Since the target $\infty$-category is presentable, it admits tensor products by any pointed space and so in particular there exists a tensor $S^{1}_{+} \otimes y(c) \in \presheaves_{\Sigma}(\inftyc)$. We claim that the latter is in the image of the Yoneda embedding; that is, we have $S^{1}_{+} \otimes y(c) \simeq y(c')$ for some $c' \in \inftyc$. It then follows that $c' \simeq S^{1}_{+} \otimes c$, ending the argument. 

Using the defining property of the tensor, for any $X \in \presheaves_{\Sigma}(\inftyc)$ we have 
\[
\map_{\presheaves_{\Sigma}(\inftyc)}(S^{1}_{+} \otimes c, X) \simeq \map_{\spaces_{*}}(S^{1}_{+}, \map_{\presheaves_{\Sigma}(\inftyc)}(y(c), X)) \simeq \map_{\spaces}(S^{1}, X(c)).
\]
We will use the cogroup structure of $c$ in the homotopy category $h \inftyc$ to construct a canonical equivalence
\[
\map_{\spaces}(S^{1}, X(c)) \simeq \Omega X(c) \times X(c) \simeq \map_{\presheaves_{\Sigma}(\inftyc)}(S^{1} \otimes y(c) \vee y(c), X),
\]
from which it follows that we can take 
\[
c' \simeq (S^{1} \otimes c) \vee c, 
\]
ending the argument. Our construction will be dual dual to the classical argument on the splitting of free loop spaces in the presence of a group structure, following \cite{haine2021splitting}. 

Since we assume that $c$ admits a structure of a cogroup in the homotopy category $h \inftyc$, there exists a counit $u \colon c \rightarrow \mathrm{1}_{\inftyc}$, where $\mathrm{1}_{\inftyc} \in \inftyc$ is the zero object, and a counital comultiplication $m \colon c \rightarrow c \vee c$ such that the associated shear map
\[
sh := (i_{1}, m) \colon c \vee c \rightarrow c \vee c
\]
is an equivalence, where $i_{1} \colon c \rightarrow c \vee c$ is the inclusion of the first summand. Applying $X$ to the the commutative diagram (dual to that of \cite[Construction 3.1]{haine2021splitting}) 
\[
\begin{tikzcd}
	c && {c \vee c} & {} & c \\
	c & {} & {c \vee c} && c
	\arrow["sh", from=2-3, to=1-3]
	\arrow["{\mathrm{id} \vee u}", from=1-3, to=1-5]
	\arrow["{id \vee u}"', from=1-3, to=1-1]
	\arrow["{\mathrm{id} \vee \mathrm{id}}"', from=2-3, to=2-1]
	\arrow["{\mathrm{id} \vee \mathrm{id}}", from=2-3, to=2-5]
	\arrow["{\mathrm{id}}"', from=2-5, to=1-5]
	\arrow["{\mathrm{id}}", from=2-1, to=1-1]
\end{tikzcd}
\]
induces a map 
\[
\map(S^{1}, X(c)) \rightarrow \Omega X(c) \times X(c)
\]
which is an equivalence by \cite[Proposition 3.3]{haine2021splitting}. This ends the proof that $S^{1}_{+} \otimes c$ exists and is equivalent to $(S^{1} \otimes c) \vee c$. By retracing the construction, we see that the composite 
\[
c \hookrightarrow (S^{1} \otimes c) \vee c \simeq S^{1}_{+} \otimes c
\]
coincides with the canonical map $c \simeq (\mathrm{pt})_{+} \otimes c \rightarrow S^{1}_{+} \otimes c$, showing the the second part. 
\end{proof}

\begin{rem}
Note that as a consequence of the defining property, formation of the tensor $U \otimes c$ preserves colimits in either of the variables. In particular, if $U$ and $U'$ are pointed spaces and $c \in \inftyc$, then 
\[
(U \vee U') \otimes c \simeq (U \otimes c) \vee (U' \otimes c),
\]
where we remind the reader that $\vee$ denotes the coproduct. It follows that if $\inftyc$ satisfies both \cref{assumption:assumption_on_c_to_have_good_presheaves} and \cref{assumption:inftyc_admits_tensors_by_the_circle}, then it admits tensors by any finite wedge of $S^{1}$ and $S^{1}_{+}$. 
\end{rem}

\begin{observation}[Cogroup structure on the tensor]
\label{observation:cogroup_structure_on_tensor}
As $S^{1}$ is an $\infty$-cogroup object in pointed spaces and since formation of tensors preserves colimits, $S^{1} \otimes c$ is canonically a $\infty$-cogroup for any $c \in \inftyc$. Through \cref{example:examples_of_tensors_by_s1_and_s0}, this corresponds to the usual $\infty$-cogroup structure on the suspension. 

Similarly, $S^{1}_{+}$ is an $\infty$-cogroup in the $\infty$-category $(\spaces_{/S^{0}})_{*}$ of pointed objects in spaces over $S^{0}$, which makes the tensor $S^{1}_{+} \otimes c$ into an $\infty$-cogroup object in $(\inftyc_{/c})_{*}$. 
\end{observation}

We will need the following lemma. 

\begin{lemma}
\label{lemma:spherical_presheaf_preserves_s1_s1plus_cofibre_sequence}
Let $Y$ be a spherical presheaf and $c \in \inftyc$. Then the sequence of spaces
\[
Y(S^{1} \otimes c) \rightarrow Y(S^{1}_{+} \otimes c) \rightarrow Y(c)
\]
is fibre. 
\end{lemma}

\begin{proof}
This follows from the second part of \cref{lemma:inftyc_satisfying_assumptions_admits_tensors_by_s1plus_which_also_split}, as the above sequence can be (non-canonically) identified with 
\[
Y(S^{1} \otimes c) \rightarrow Y(S^{1} \otimes c) \times Y(c) \rightarrow Y(c). 
\]
\end{proof}

\begin{defin}
Let $U$ be either $S^{1}$ or $S^{1}_{+}$. If $Y \in \presheaves_{\Sigma}(\inftyc)$ is a spherical presheaf, then its $U$-\emph{shift} is the presheaf $Y_{U}$ defined by the formula
\[
Y_{U}(c) := Y(U \otimes c).
\]
This is again spherical, as $U \otimes - \colon \inftyc \rightarrow \inftyc$ preserves coproducts. 
\end{defin} 

\begin{observation}[Group structure on the shift]
\label{observation:group_structure_of_a_shift_and_recovering_ys1_from_ys1plus}
As we observed in \cref{observation:cogroup_structure_on_tensor}, the tensor $S^{1} \otimes c$ has a canonical structure of an $\infty$-cogroup. Since a spherical presheaf 
\[
Y \colon \inftyc^{op} \rightarrow \spaces
\]
takes finite coproducts to products, we deduce that $Y(S^{1} \otimes c)$ is canonically an $\infty$-group in spaces. By varying $c \in \inftyc$, this endows the shift $Y_{S^{1}}$ with the structure of an $\infty$-group object in spherical presheaves, natural in $Y$. Similarly, $Y_{S^{1} _{+}}$ is a $\infty$-group object in $\presheaves_{\Sigma}(\inftyc)_{/Y}$.

Note that one can recover $Y_{S^{1}}$ (together with its $\infty$-group structure) from $Y_{S^{1}_{+}}$: as a consequence of \cref{lemma:spherical_presheaf_preserves_s1_s1plus_cofibre_sequence}, $Y_{S^{1}}$ can be canonically identified with the fibre of $Y_{S^{1}_{+}} \rightarrow Y$.
\end{observation}

The shift construction can be used to produce modules in the sense of \cref{defin:module_over_a_spherical_presheaf}:

\begin{example}[Module structures on shifts]
\label{example:module_structure_on_shifts}
If $\Lambda$ is a $\inftyc$-algebra, then so is $\Lambda_{S^{1}_{+}}$ and the map $\Lambda_{S^{1}_{+}} \rightarrow \Lambda$ has a canonical structure of a $\Lambda$-module with underlying abelian group object $\Lambda_{S^{1}}$. To see this, note that by \cref{observation:group_structure_of_a_shift_and_recovering_ys1_from_ys1plus}, $\Lambda_{S^{1}_{+}} \in \presheaves_{\Sigma}(\inftyc)_{/\Lambda}$ has a canonical structure of a group object. Morever, this group structure is abelian (which is a property rather than a structure, as $\Lambda_{S^{1}_{+}}$ is discrete) as a consequence of the Eckmann-Hilton argument and \cref{prop:spherical_presheaves_take_values_in_h_spaces}. 

Since the shift construction commutes with limits, if $E \rightarrow \Lambda$ is a $\Lambda$-module with underlying abelian group object $A$, then $E_{S^{1}_{+}} \rightarrow \Lambda_{S^{1}_{+}}$ is a $\Lambda_{S^{1}_{+}}$-module with underlying abelian group object $A_{S^{1}_{+}}$.  The base-change is $E _{S^{1}_{+}} \times _{\Lambda_{S^{1}_{+}}} \Lambda \rightarrow \Lambda$ is then again a $\Lambda$-module. The diagram
\[
\begin{tikzcd}
	{E_{S^{1}_{+}}} & E \\
	{\Lambda_{S^{1}_{+}}} & \Lambda
	\arrow[from=1-1, to=1-2]
	\arrow[from=1-2, to=2-2]
	\arrow[from=1-1, to=2-1]
	\arrow[from=2-1, to=2-2]
\end{tikzcd}
\]
induces a map of $\Lambda$-modules whose kernel 
\[
K := \mathrm{ker}_{\abeliangroups(\presheaves_{\Sigma}(\inftyc)_{/\Lambda}}(E _{S^{1}_{+}} \times _{\Lambda_{S^{1}_{+}}} \Lambda \rightarrow E)
\]
is a $\Lambda$-module with underlying abelian group object $A_{S^{1}}$. 

In the case of $\Pi$-algebras, these constructions correspond to the module structures on shifts studied in the work of Blanc-Dwyer-Goerss, see \cite[4.13]{realization_space_of_a_pi_algebra}.
\end{example}
Dually to the shift construction, we have the levelwise mapping space, and the two are related by the following comparison map: 

\begin{construction}
\label{construction:loop_and_free_loop_comparison_maps}
Let $U$ be either $S^{1}$ or $S^{1}_{+}$. If $Y$ is a spherical presheaf, then we write $\map_{*}(U, Y)$ for the spherical presheaf defined by 
\[
\map_{*}(U, Y)(c) := \map_{\spaces_{*}}(U, Y(c)).
\]
Since $\map_{*}(S^{0}, U) \simeq U$ , there is a canonical family of maps $Y_{U} \rightarrow Y$ parametrized by $U$. These assemble into a natural transformation of spherical presheaves 
\[
Y_{U} \rightarrow \map_{*}(U, Y)
\]
which we call the \emph{loop comparison map}. 
\end{construction}

Since $S^{1}$ is an $\infty$-cogroup in pointed spaces, $\map_{*}(S^{1}, Y) \simeq \Omega Y$ is canonically an $\infty$-group in spherical presheaves. Similarly, $\map_{*}(S^{1}_{+}, Y)$ has a structure of an $\infty$-group object in $\presheaves_{\Sigma}(\inftyc)_{/Y}$. The comparison maps
\[
Y_{S^{1}} \rightarrow \map_{*}(S^{1}, Y) \simeq \Omega Y
\]
and
\[
Y_{S^{1}_{+}} \rightarrow \map_{*}(S^{1}_{+}, Y) \simeq \Omega_{Y} (Y \times Y) := Y \times_{Y \times Y} Y 
\]
lift to maps of $\infty$-group objects, as the $\infty$-group structure on both the source and target arise from the additional $S^{1}$ (resp. $S^{1}_{+}$) coordinate. Using the bar-loops adjunction between pointed objects and $\infty$-group objects, this two maps have adjoints 
\begin{equation}
\label{equation:adjoint_of_loop_comparison_map}
\B Y_{S^{1}} \rightarrow Y
\end{equation}
and 
\begin{equation}
\label{equation:adjoint_of_free_loop_comparison_map}
\B_{Y} Y_{S^{1}_{+}} \rightarrow Y \times Y,
\end{equation}
where $\B Y_{S^{1}} := \varinjlim_{[n] \in \Delta^{op}} (Y_{S^{1}})^{\times n} $ and $\B _{Y} Y_{S^{1}_{+}} := \varinjlim_{[n] \in \Delta^{op}} (Y_{S^{1}_{+}})^{\times_{Y} n}$ are the bar constructions internal to, respectively, the $\infty$-categories $\presheaves_{\Sigma}(\inftyc)$ and $\presheaves_{\Sigma}(\inftyc)_{/Y}$. 

The main result of this section will relate the adjoints (\ref{equation:adjoint_of_loop_comparison_map}) and (\ref{equation:adjoint_of_free_loop_comparison_map}) to the projection $p \colon \inftyc \rightarrow h \inftyc$ onto the homotopy category. Recall from \cref{prop:adjunction_between_spherical_preshaves_and_homotopy_spherical_presheaves} that this projection induces an adjunction
\[
p^{*} \colon \presheaves_{\Sigma}(\inftyc) \rightleftarrows \presheaves_{\Sigma}(h\inftyc) :p_{*}
\]
between spherical presheaves and spherical homotopy presheaves, where $p_{*}$ is given by precomposition. This adjunction determines for any $Y \in \presheaves_{\Sigma}(\inftyc)$ a unit map $Y \rightarrow p_{*} p^{*} Y$ which we will relate to the adjoints of the comparison maps. 

\begin{lemma}
\label{lemma:spiral_cartesian_squares_for_a_representable}
Suppose that $Y \simeq y(c)$ is representable. Then, there exists a unique homotopy making the diagram
\[
\begin{tikzcd}
	\B Y_{S^{1}} & Y \\
	\mathrm{pt} & p_{*} p^{*} Y
	\arrow[from=2-1, to=2-2]
	\arrow[from=1-2, to=2-2]
	\arrow[from=1-1, to=2-1]
	\arrow[from=1-1, to=1-2]
\end{tikzcd}
\]
commute, where the top horizontal arrow is the adjoint (\ref{equation:adjoint_of_loop_comparison_map}) and the right vertical arrow is the unit. Similarly, there exists a unique homotopy making the diagram 
\[
\begin{tikzcd}
	\B Y_{S^{1}_{+}} & Y \\
	Y & p_{*} p^{*} Y
	\arrow[from=2-1, to=2-2]
	\arrow[from=1-2, to=2-2]
	\arrow[from=1-1, to=2-1]
	\arrow[from=1-1, to=1-2]
\end{tikzcd}
\]
of spherical presheaves commute, where the top horizontal and left vertical arrows are determined by the adjoint (\ref{equation:adjoint_of_free_loop_comparison_map}) and the other two are unit maps. Moreover, both squares are cartesian and their right vertical map is a $0$-equivalence. 
\end{lemma}

\begin{proof}
We begin with the first diagram. We have $Y(-) \simeq \map_{\inftyc}(-, c)$, so that 
\[
\B Y_{S^{1}}(-) \simeq \B(\map_{\inftyc}(S^{1} \otimes -, c)) \simeq \B \Omega (\map_{\inftyc}(-, c)) \simeq \map_{\inftyc}(-, c)_{0}, 
\]
where the subscript $(-)_0$ denotes the basepoint component. Similarly, we have 
\[
p_{*} p^{*} Y(-) \simeq \map_{h \inftyc}(-, c) \simeq \pi_{0}\map_{\inftyc}(-, c). 
\]
It follows that the first diagram can be identified with 
\[
\begin{tikzcd}
	\map_{\inftyc}(-, c)_{0} &  \map_{\inftyc}(-, c) \\
	{\mathrm{pt}} & \pi_{0}\map_{\inftyc}(-, c)
	\arrow[from=1-1, to=1-2]
	\arrow[from=1-2, to=2-2]
	\arrow[from=2-1, to=2-2]
	\arrow[from=1-1, to=2-1]
\end{tikzcd},
\]
which commutes up to homotopy. Moreover, as the bottom right vertex is discrete, such a homotopy is essentially unique. Clearly, this is a cartesian diagram with right vertical map a $0$-equivalence, as  claimed. 

Making analogous identifications for the second diagram, we have 
\[
\B_{Y} Y_{S^{1}_{+}} \simeq \B_{Y} \Omega_{Y} (\map_{\inftyc}(-, c) \times \map_{\inftyc}(-, c)) \simeq (\map_{\inftyc}(-, c) \times \map_{\inftyc}(-, c))_{\mathrm{diag}}, 
\]
where $(-)_{\mathrm{diag}}$ denotes the path components on the diagonal (this is the analogue of the basepoint component for pointed objects in presheaves over  $Y$, since the diagonal is exactly the image of the basepoint map $Y \rightarrow Y \times Y$). Thus, the second diagram is equivalent to 
\[
\begin{tikzcd}
	(\map_{\inftyc}(-, c) \times \map_{\inftyc}(-, c))_{\mathrm{diag}} &  \map_{\inftyc}(-, c) \\
	\map_{\inftyc}(-, c) & \pi_{0}\map_{\inftyc}(-, c)
	\arrow[from=1-1, to=1-2]
	\arrow[from=1-2, to=2-2]
	\arrow[from=2-1, to=2-2]
	\arrow[from=1-1, to=2-1]
\end{tikzcd},
\]
which can also be made to commute using an essentially unique homotopy, and is clearly cartesian with the right vertical map a $0$-equivalence.
\end{proof}
The commutative squares of \cref{lemma:spiral_cartesian_squares_for_a_representable} can be extended from representables to the general case in the following way. 

\begin{construction}
\label{construction:spiral_fibre_sequence_and_total_spiral_square}
Observe that each of the endofunctors of $\presheaves_{\Sigma}(\inftyc)$ 
\begin{enumerate}
    \item $Y \mapsto \B Y_{S^{1}}$,
    \item $Y \mapsto \B_{Y} Y_{S^{1}{+}}$ and 
    \item $Y \mapsto p_{*} p^{*} Y$ 
\end{enumerate}
preserves sifted colimits. The first one is a composite of $Y \mapsto Y_{S^{1}}$, which preserves sifted colimits of spherical presheaves by \cref{observation:limits_colimits_in_spherical_presheaves} as it is calculated levelwise, and the classifying object functor $\B(-)$ which preserves sifted colimits as it can be computed by the bar construction; that is, $\B Y \simeq \varinjlim_{[n] \in \Delta^{op}} Y^{\times n}$. The same argument applies to the second one. The third one is a composite of $p^{*}$ which is cocontinuous as a left adjoint, and $p_{*}$ which preserves sifted colimits by \cref{observation:right_adjoint_preserves_sifted_colimits}.

Since  by \cref{prop:universal_property_of_the_nonabelian_derived_category} the Yoneda embedding induces an equivalence
\[
\Fun(\inftyc, \presheaves_{\Sigma}(\inftyc) \simeq \Fun^{\mathrm{sft}}(\presheaves_{\Sigma}(\inftyc), \presheaves_{\Sigma}(\inftyc)), 
\]
where the right hand side is the $\infty$-category of sifted colimit-preserving functors, the commutative squares of \cref{lemma:spiral_cartesian_squares_for_a_representable} induce commutative squares in the functor $\infty$-category. In other words, for any $Y \in \presheaves_{\Sigma}(\inftyc)$, we have canonical commutative squares
\begin{equation}
\label{equation:spiral_fibre_square}
\begin{tikzcd}
	\B Y_{S^{1}} & Y \\
	\mathrm{pt} & p_{*} p^{*} Y
	\arrow[from=2-1, to=2-2]
	\arrow[from=1-2, to=2-2]
	\arrow[from=1-1, to=2-1]
	\arrow[from=1-1, to=1-2]
\end{tikzcd}
\end{equation}
and 
\begin{equation}
\label{equation:spiral_cartesian_square}
\begin{tikzcd}
	\B Y_{S^{1}_{+}} & Y \\
	Y & p_{*} p^{*} Y
	\arrow[from=2-1, to=2-2]
	\arrow[from=1-2, to=2-2]
	\arrow[from=1-1, to=2-1]
	\arrow[from=1-1, to=1-2]
\end{tikzcd}
\end{equation}
and these squares are covariantly functorial in $Y$. 
\end{construction}

\begin{defin}
We call the commutative squares (\ref{equation:spiral_fibre_square}) and (\ref{equation:spiral_cartesian_square}) of \cref{construction:spiral_fibre_sequence_and_total_spiral_square} the \emph{spiral sequence} and the \emph{spiral total square}. 
\end{defin}
Our main result of this section will show that both of these squares are in fact cartesian. We've shown this is the case for representables in \cref{lemma:spiral_cartesian_squares_for_a_representable}, but the general case does not follow immediately from the construction, as cartesian squares are not closed under sifted colimits. Instead, w will argue using the notion of realization fibration due to Rezk \cite{rezk_realization_fibration}. 

\begin{theorem}
\label{thm:spiral_fibre_sequence}
For any spherical presheaf $Y \in \presheaves_{\Sigma}(\inftyc)$
\begin{enumerate}
    \item the unit map $Y \rightarrow p_{*} p^{*} X$ is a $0$-equivalence and
    \item the spiral sequence (\ref{equation:spiral_fibre_square}) and the spiral total square (\ref{equation:spiral_cartesian_square}) are both cartesian.
\end{enumerate}
\end{theorem}

\begin{proof}
As the subcategory of those $Y \in \presheaves_{\Sigma}(\inftyc)$ which satisfy these two conditions contains all representables by \cref{lemma:spiral_cartesian_squares_for_a_representable}, it is enough to verify that it is closed under sifted colimits. Note that the formation of both squares commutes with sifted colimits by construction. 

We can work with filtered colimits and geometric realizations separately. The former is clear, as both $0$-equivalences and cartesian squares are closed under filtered colimits. Thus, suppose that $Y \simeq | Y_{\bullet} |$ is a geometric realization of spherical presheaves satisfying the above two conditions. We have to show that the spiral sequence of $Y$, which can be identified with 
\[
\begin{tikzcd}
	{| \B (Y_{\bullet})_{S^{1}} |} & { | Y_{\bullet} | } \\
	{\mathrm{pt}} & { | p_{*} p^{*} Y_{\bullet} | }
	\arrow[from=1-1, to=1-2]
	\arrow[from=1-2, to=2-2]
	\arrow[from=2-1, to=2-2]
	\arrow[from=1-1, to=2-1]
\end{tikzcd}, 
\]
also has these two properties. 

Both of the properties are detected levelwise, so it suffices to verify that for any $c \in \inftyc$, the square 
\[
\begin{tikzcd}
	{| (\B (Y_{\bullet})_{S^{1}})(c) |} & { | Y_{\bullet}(c) | } \\
	{\mathrm{pt}} & { | (p_{*} p^{*} Y_{\bullet})(c) | }
	\arrow[from=1-1, to=1-2]
	\arrow[from=1-2, to=2-2]
	\arrow[from=2-1, to=2-2]
	\arrow[from=1-1, to=2-1]
\end{tikzcd}
\]
is cartesian with the right vertical arrow a $0$-equivalence of spaces. Using \cref{prop:spherical_presheaves_take_values_in_h_spaces}, the map 
\[
Y_{\bullet}(c) \rightarrow (p_{*} p^{*} Y_{\bullet})(c)
\]
can be lifted to a morphism of simplicial group objects in the homotopy category of spaces, thus by \cite[Cor. 5.8]{rezk_realization_fibration} is a realization fibration in the sense of Rezk \cite[Def. 1.1]{rezk_realization_fibration}. By definition, this means that taking pullbacks along this map preserves geometric realization, proving that the needed square is cartesian. To see that the right vertical arrow is a $0$-equivalence, observe that $\pi_{0} \colon \spaces \rightarrow \sets$ is a left adjoint and hence commutes with all colimits, in particular geometric realizations.
\end{proof}


By passing to the homotopy groups in the spiral fibre sequence of \cref{thm:spiral_fibre_sequence} we obtain a long exact sequence of $\inftyc$-algebras, which is of the form 
\begin{equation}
\label{equation:spiral_long_exact_sequence}
\ldots \rightarrow (\pi_{1} Y)_{S^{1}} \rightarrow \pi_{2} Y \rightarrow \pi_{2} (p_{*} p^{*} Y)  \rightarrow (\pi_{0} Y)_{S^{1}} \rightarrow \pi_{1} Y \rightarrow \pi_{1} (p_{*} p^{*} Y) \rightarrow 0,
\end{equation}
where we're using that $\pi_{n}(\B Y_{S^{1}}) \simeq \pi_{n-1}(Y)_{S^{1}}$ for $n \geq 1$ and $\pi_{0} (\B Y_{S^{1}}) = 0$. Since the map $Y \rightarrow p_{*} p^{*}$ is a $0$-equivalence, the long exact sequence ends with the isomorphism $\pi_{0} Y \simeq \pi_{0} (p^{*} p_{*} Y)$, which we omit. 

\begin{defin}
\label{defin:spiral_long_exact_sequence}
We call the long exact sequence of (\ref{equation:spiral_long_exact_sequence}) the \emph{spiral long exact sequence}. 
\end{defin}

\begin{rem}
We will not need this, but through the equivalence between spherical presheaves over the $\infty$-category of spheres and the localization of simplicial topological spaces at $E_{2}$-equivalences sketched in \S\ref{subsection:overview_of_the_approach}, one can show that the spiral long exact sequence of \cref{defin:spiral_long_exact_sequence} corresponds to the analogous notion introduced by Dwyer-Kan-Stover in \cite[8.1]{bigraded_homotopy_groups}.
\end{rem}
For calculational purposes, it will be important to know that the spiral long exact sequence is compatible with the $\pi_{0}Y$-module structure, which we verify now. 

\begin{lemma}
\label{lemma:spiral_long_exact_sequence_is_a_sequence_of_modules}
The spiral long exact sequence of  (\ref{equation:spiral_long_exact_sequence}) can be canonically promoted to a long exact sequence of $\pi_{0}Y$-modules, where we equip $(\pi_{n} Y)_{S^{1}}$ with the module structure of \cref{example:action_on_absolute_homotopy_groups_for_spherical_presheaves} and \cref{example:module_structure_on_shifts}. 
\end{lemma}

\begin{proof}
Since the spiral long exact sequence is a long exact sequence of homotopy associated to a fibre sequence $\B Y_{S^{1}} \rightarrow Y \rightarrow p_{*} p^{*} Y$, arguments analogous to the case of spaces show that it can be promoted to a long exact sequence of $\pi_{0} Y$-modules. To see that the module structure on
\[
\pi_{n}(\B Y_{S^{1}}) \simeq \pi_{n}(Y \rightarrow p_{*} p^{*}Y)
\]
(where we abusively identify the module on the right with its underlying $\inftyc$-algebra) corresponds to that of the shift $(\pi_{n-1}Y)_{S^{1}}$ described in \cref{example:module_structure_on_shifts}, note that since the spiral total square is cartesian, we have 
\[
\pi_{n}(Y \rightarrow p_{*} p^{*}Y) \simeq \pi_{n}(\B_{Y} Y_{S^{1}_{+}} \rightarrow Y) \simeq \pi_{n-1}(Y_{S^{1}_{+}} \rightarrow Y)
\]
and it is the left hand side which is used to define the module structure on shifts. 
\end{proof}

\section{Moduli of spaces with prescribed homotopy groups}
\label{section:moduli_of_pi_algebras}

In this chapter we use the theory of classification of $n$-types in $\infty$-categories of spherical presheaves developed in \S\ref{section:classification_of_presheaves} to give decomposition results for moduli of realizations of a given $\Pi$-algebra. 

Previously, we worked relative to \emph{some} indexing $\inftyc$; in this chapter, we will focus exclusively on the following indexing $\infty$-category: 

\begin{defin}
\label{defin:infty_cat_of_wedges_of_spheres}
We let $\spheres$ denote the full subcategory $\spheres \subseteq \spaces_{*}$ of pointed spaces spanned by spaces of the form $\bigvee_{i \in I} S^{k_{i}}$ where $I$ is finite and $k_{i} \geq 1$ for each $i \in I$. We will call objects of $\spheres$ \emph{wedges of spheres}. 
\end{defin}
Note that $\spheres$ satisfies both 
and \cref{assumption:assumption_on_c_to_have_good_presheaves} and \cref{assumption:inftyc_admits_tensors_by_the_circle}. Throughout this chapter, \emph{spherical presheaf} will always refer to an object of $\presheaves_{\Sigma}(\spheres)$, and \emph{spherical homotopy presheaf} will be an object of $\presheaves_{\Sigma}(h \spheres)$. Note that $\spheres$-algebras in the sense of \cref{defin:c_algebra} can be identified with $\Pi$-algebras of Stover, as given in \cref{defin:pi_algebra_definition_in_introduction}. 

\begin{rem}[Spherical homotopy presheaves and $\Pi$-algebras]
The $\infty$-category $\presheaves_{\Sigma}(\spheres)$ can be described purely in terms of $\Pi$-algebras, as it is canonically equivalent to both: 
\begin{enumerate}
\item the $\infty$-category $\Pi\mhyphen\Alg^{an}$ of animated $\Pi$-algebras by \cref{rem:homotopy_presheaves_and_animated_inftyc_algebras} and 
\item the $\infty$-categorical localization $s \Pi\mhyphen\Alg[W^{-1}]$ of the category of simplicial $\Pi$-algebras at levelwise weak equivalences of simplicial sets by \cref{rem:homotopy_presheaves_and_simplicial_c_algebras}.
\end{enumerate}
The first one is the terminology we used in the introduction in the statement of \cref{thm:main_theorem}; we will prove it as \cref{thm:cartesian_square_of_moduli} below using the language of spherical homotopy presheaves. Using the second equivalence, one can translate between our account and the one of Blanc-Dwyer-Goerss \cite{realization_space_of_a_pi_algebra}, but we leave the translation to the interested reader. 
\end{rem}

\subsection{Spaces and spherical presheaves}

In this section we clarify the relation between pointed spaces and spherical presheaves indexed by $\spheres$. The key result is that pointed, connected spaces can be identified with a certain class of spherical presheaves. 

We have a restricted Yoneda embedding $y \colon \spaces _{*} \hookrightarrow \presheaves(\spheres)$ given by 
\[
y(A)(U) := \map_{\spaces_{*}}(U, A)
\]
As this formula sends finite coproducts in $U$ to products of spaces, $y(A)$ is spherical for any $A \in \spaces_{*}$. 

Notice that since all spaces in $\spheres \subseteq \spaces_{*}$ are connected, the restricted Yoneda $y(A)$ depends only on the basepoint component $A_{0} \subseteq A$. Thus, it will be convenient to view the restricted Yoneda embedding as a functor 
\[
y \colon \spaces_{*}^{\geq 1} \rightarrow \presheaves_{\Sigma}(\spheres)
\]
defined on the subcategory of connected spaces. 

\begin{prop}
\label{prop:characterization_of_spacelike_spherical_presheaves}
The restricted Yoneda embedding $y \colon \spaces_{*}^{\geq 1} \rightarrow \presheaves_{\Sigma}(\spheres)$ is a full and faithful functor of $\infty$-categories. Moreover, its essential image consists of those $Y \in \presheaves_{\Sigma}(\spheres)$ such that the loop comparison map $Y_{S^{1}} \rightarrow \Omega Y$ of \cref{construction:loop_and_free_loop_comparison_maps} is an equivalence. 
\end{prop}

\begin{proof}
Since $S^{1} \in \spheres$ has a canonical $\infty$-cogroup structure, $Y(S^{1})$ is an $\infty$-group object in spaces for any $Y \in \presheaves_{\Sigma}(\spheres)$. We define a functor
\[
l \colon \presheaves_{\Sigma}(\spheres) \rightarrow \spaces_{*}^{\geq 1}
\]
by the formula 
\[
l(Y) := \B Y(S^{1}) \simeq \varinjlim_{[n] \in \Delta^{op}} Y(S^{1})^{\times n}
\]
the classifying space of the $\infty$-group $Y(S^{1})$. 

We claim that $l$ is left adjoint to the restricted Yoneda embedding; that is, that there exists an equivalence
\[
\map_{\spaces_{*}}(\B Y(S^{1}), A) \simeq \map_{\spaces_{*}}(Y, y(A))
\]
natural in $A \in \spaces_{*}$ and $Y \in \presheaves_{\Sigma}(\spheres)$. As the classifying space construction preserves sifted colimits of spaces, both sides take sifted colimits in $Y$ to limits. Thus, it is enough to construct such an equivalence when $Y = y(U)$ for $U \in \spheres$. In this case we have 
\[
\map_{\spaces_{*}}(\B y(U)(S^{1}), A) \simeq \map_{\spaces_{*}}(\B \Omega U, A) \simeq \map_{\spaces_{*}} (U, A) \simeq \map_{\spaces_{*}}(y(U), y(A)),
\]
as needed. 

We will now show that the counit $l(y(A)) \rightarrow A$ is an equivalence for any $A \in \spaces_{*}$; this implies that $y$ is fully faithful. After unwrapping the definitions, we see that this counit map can be identified with 
\[
l(y(A)) \simeq \B(y(A)(S^{1})) \simeq \B(\Omega A) \rightarrow A,
\]
the counit of the adjunction $\B \dashv \Omega$ between pointed spaces and $\infty$-groups. This is an equivalence since $A$ is assumed to be connected. 

The essential image of $y$ is given by those $Y \in \presheaves_{\Sigma}(\spheres)$ for which the unit map 
\[
Y \rightarrow y(l(Y))
\]
is an equivalence. After evaluating at $U \in \spheres$, the unit map takes the form 
\begin{equation}
\label{equation:unit_map_in_proof_of_recognition_of_loop_transfer_presheaves}
Y(U) \rightarrow \map_{\spaces_{*}}(U, \B Y(S^{1})).
\end{equation}
We claim that the condition of (\ref{equation:unit_map_in_proof_of_recognition_of_loop_transfer_presheaves}) being an equivalence for any $U \in \spheres$ is equivalent  to the loop comparison map being an equivalence, which will finish the argument. One direction is clear, as anything in the essential image is of the form $y(A)$, in which case the loop comparison map is an equivalence. 

Now assume that $Y$ is arbitrary and the loop comparison map $Y_{S^{1}} \rightarrow \Omega Y$ is an equivalence; we have to show that the unit $Y \rightarrow y(l(Y))$ is an equivalence of presheaves. Since $\Omega \B Y(S^{1}) \simeq Y(S^{1})$, (\ref{equation:unit_map_in_proof_of_recognition_of_loop_transfer_presheaves}) is an equivalence when $U = S^{1}$. Consider the diagram 
\begin{center}
	\begin{tikzpicture}
		\node (TL) at (0, 1.5) {$ Y_{S^{1}} $};
		\node (BL) at (0, 0) {$ \Omega Y$};
		\node (TR) at (1.5, 1.5) {$ (ylY) _{S^{1}} $};
		\node (BR) at (1.5, 0) {$ \Omega (ylY) $};
		
		\draw [->] (TL) to (TR);
		\draw [->] (TL) to (BL);
		\draw [->] (TR) to (BR);
		\draw [->] (BL) to (BR);
	\end{tikzpicture}
\end{center}
where the vertical maps are given by the loop comparison and so are equivalences by assumption. After unwrapping the definition of the comparison map, we see that this means that 
\[
Y(S^{n+1}) \simeq \Omega Y(S^{n}) 
\]
and 
\[
(ylY)(S^{n+1}) \simeq \Omega (ylY)(S^{n}).
\]
Thus, (\ref{equation:unit_map_in_proof_of_recognition_of_loop_transfer_presheaves}) being an equivalence for $U = S^{1}$ implies inductively the same for $U = S^{n}$ for all $n \geq 1$. As any $U \in \spheres$ is a finite coproduct of spheres and both presheaves are spherical, we see that the unit of $Y$ is an equivalence, as needed. 
\end{proof}

\begin{rem}
One natural way to rephrase \cref{prop:characterization_of_spacelike_spherical_presheaves} is as follows: 
 a presheaf $X \colon \spheres^{op} \rightarrow \spaces$ is representable by a pointed space if and only if it takes finite coproducts to products and takes suspensions to loops. 
\end{rem}


\begin{rem}[Why spherical presheaves?]
Note that \cref{prop:characterization_of_spacelike_spherical_presheaves} above gives some intuition as to why $\presheaves_{\Sigma}(\spheres)$ is the right place to study the moduli theory of $\Pi$-algebras: it contains both the category of $\Pi$-algebras and the $\infty$-category of pointed, connected spaces as full, naturally defined subcategories. 
\end{rem}

\subsection{Moduli of potential $n$-types} 

In this section we construct a tower of $\infty$-groupoids whose limit is the moduli space $\mathcal{M}(A)$ defined below. 

We begin by introducing the relevant notions.

\begin{defin}[{\cite{realization_space_of_a_pi_algebra}}]
\label{defin:realization_and_moduli_of_realizations}
If $A$ is a $\Pi$-algebra, then a pointed, connected space $X$ is a \emph{realization} of $A$ if there exists an isomorphism $\pi_{*}X \simeq A$. 
\end{defin}

\begin{defin}[{\cite{realization_space_of_a_pi_algebra}}]
The \emph{moduli of realizations} denoted by $\mathcal{M}(A)$, is the full subgroupoid of $\spaces_{*}^{\cong}$ spanned by realizations of $A$ and their equivalences.
\end{defin}

\begin{warning}
Note that to be realization of $A$ is a property, rather than a structure: the isomorphism of $\Pi$-algebras $\pi_{*}X \simeq A$ is assumed to exist, but it is not fixed. 
\end{warning}

Notice that by \cref{prop:characterization_of_spacelike_spherical_presheaves} the $\infty$-category of pointed, connected spaces can be identified through the Yoneda embedding with the full subcategory of those spherical presheaves for which the loop comparison map of \cref{construction:loop_and_free_loop_comparison_maps} is an equivalence. 

If $X$ such a spherical presheaf, then the underlying $\Pi$-algebra of the corresponding pointed, connected space is easy to describe:  it is given by the underlying $\spheres$-algebra $X_{\leq 0}$, considered as a discrete object of $\presheaves_{\Sigma}(\spheres)$. Because of that, we can identify $\mathcal{M}(A)$ of \cref{defin:realization_and_moduli_of_realizations} with a full subgroupoid of $\presheaves_{\Sigma}(\spheres)^{\cong}$ spanned by those $X$ such that 
\begin{enumerate}
\item the loop comparison map $X_{S^{1}} \rightarrow \Omega X$ of \cref{construction:loop_and_free_loop_comparison_maps} is an equivalence and 
\item there exists an isomorphism $X_{\leq 0} \simeq A$ of $\Pi$-algebras. 
\end{enumerate}
The second condition suggests a way of using $\presheaves_{\Sigma}(\spheres)$ to interpolate between a $\Pi$-algebra and its realizations by studying $n$-truncations of realizations for $0 \leq n \leq \infty$. We don't know how these truncations might look like before knowing the realizations themselves, so instead we axiomatize their common features in the following way: 

\begin{defin}[{cf. \cite[Definition 9.1]{realization_space_of_a_pi_algebra}}]
\label{defin:potential_n_stages_and_their_moduli}
Let $0 \leq n \leq \infty$. We say that a spherical presheaf $X \in \presheaves_{\Sigma}(\spheres)$ is a \emph{potential $n$-stage for $A$} if it satisfies the following conditions: 
\begin{enumerate}
\item $X_{\leq 0} \simeq A$; that is, the path components of $X$ form a $\Pi$-algebra isomorphic to $A$,
\item $X$ is an $n$-type; that is, $X(U)$ is an $n$-type for each $U \in \spheres$,
\item if $n \geq 1$, the loop comparison map $X_{S^{1}} \rightarrow \Omega X$ is an $(n-1)$-equivalence.
\end{enumerate}
We let $\mathcal{M}_{n}(A)$ denote the full subgroupoid of $\presheaves_{\Sigma}(\spheres)^{\cong}$ spanned by potential $n$-stages for $A$. 
\end{defin}

\begin{example}[Potential 0-stages]
\label{example:moduli_of_0_stages}
Since any map of pointed spaces is a $(-1)$-equivalence, condition $(3)$ is vacuous when $n = 0$. It follows that $\mathcal{M}_{0}(A)$ is the subcategory of discrete objects of $\presheaves_{\Sigma}(\spheres)$ (in other words, of $\Pi$-algebras) that happen to be isomorphic to $A$. Thus, the $\infty$-groupoid $\mathcal{M}_{0}(A)$ is canonically equivalent to $\B \Aut_{\Pi \mhyphen \Alg}(A)$.
\end{example}

\begin{example}[Potential $\infty$-stages]
By  \cref{prop:characterization_of_spacelike_spherical_presheaves}, 
 the Yoneda embedding induces a canonical equivalence
 \[
 \mathcal{M}_{\infty}(A) \simeq \mathcal{M}(A)
 \]
 with the moduli of realization of \cref{defin:realization_and_moduli_of_realizations}. To see this, note that when $n = \infty$, then condition $(3)$ in \cref{defin:potential_n_stages_and_their_moduli} forces the loop comparison map to be an equivalence. 
\end{example}

\begin{rem}
In the case of $n < \infty$, notice that the second condition implies that $\Omega X$ is an $(n-1)$-type. Thus, the third condition simply states that the loop comparison map gives an equivalence $(X_{S^{1}}) _{\leq n-1} \simeq \Omega X$ or, intuitively, that it is as connected as possible. 
\end{rem}

\begin{observation}[Homotopy groups of potential $n$-stages]
\label{observation:homotopy_groups_of_potential_n_stages}
Using the spiral long exact sequence of \cref{defin:spiral_long_exact_sequence}, which by \cref{lemma:spiral_long_exact_sequence_is_a_sequence_of_modules} can be promoted to a long exact sequence of modules, one reads off that the homotopy groups of a potential $n$-stage $X$ for a $\Pi$-algebra $A$ are given by 
\[
\pi_{n} X \simeq A_{S^{n}}.
\]
Here, right hand side has the classical $A$-module structure on the shift, which we recalled in \cref{example:module_structure_on_shifts}. Note that these isomorphism are canonically determined once a choice of an isomorphism $\pi_{0} X \simeq A$ is made. The homotopy groups of a potential $n$-stage vanish above degree $n$. 
\end{observation}

It is immediate from \cref{defin:potential_n_stages_and_their_moduli} that for each $\infty \geq m \geq n \geq 0$, $n$-truncation induces a functor
\[
(-)_{\leq n} \colon \mathcal{M}_{m}(A) \rightarrow \mathcal{M}_{n}(A).
\]
Assembling these together, we have a tower of $\infty$-groupoids
\[
\mathcal{M}(A) \simeq \mathcal{M}_{\infty}(A) \rightarrow \ldots \rightarrow \mathcal{M}_{2}(A) \rightarrow \mathcal{M}_{1}(A) \rightarrow \mathcal{M}_{0}(A) \simeq \B \Aut(A)
\]
which interpolates between the groupoid of $\Pi$-algebras isomorphic to $A$ and the moduli of its realizations. We now verify that this tower is convergent: 

\begin{prop}
\label{prop:moduli_of_infinity_stages_as_a_limit}
The natural map $\mathcal{M}_{\infty}(A) \rightarrow \varprojlim \mathcal{M}_{n}(A)$ is an equivalence.
\end{prop}

\begin{proof}
The limit on the right hand side can be identified with the $\infty$-groupoid of those functors $X \colon N(\mathbb{Z}_{\geq 0})^{op} \rightarrow \presheaves_{\Sigma}(\spheres)$ with the property that for each $m \geq n$, the induced map $X_{m} \rightarrow X_{n}$ expresses the target as the $n$-truncation of the domain and such that $X_{n}$ is a potential $n$-stage for all $n$. Thus, it is an $\infty$-groupoid of certain Postnikov pretowers in $\presheaves_{\Sigma}(\spheres)$ in the sense of \cite[5.5.6.23]{lurie_higher_topos_theory}. 

Now, Postnikov towers in $\presheaves(\spheres)$ are convergent by \cite[7.2.1.3, 7.2.1.10]{lurie_higher_topos_theory} and so the same is true for $\presheaves_{\Sigma}(\spheres)$ as a presheaf is spherical if and only if all of its truncations are. We thus have a span of $\infty$-categories 

\begin{center}
$\presheaves_{\Sigma}(\spheres) \leftarrow \mathrm{Post}^{+}(\presheaves_{\Sigma}(\spheres)) \rightarrow \mathrm{Post}(\presheaves_{\Sigma}(\spheres))$
\end{center}
between spherical presheaves and respectively, Postnikov towers and Postnikov pretowers therein. Both arrows are equivalences of $\infty$-categories: the right one by functoriality of limits and the left one by \cite[7.2.1.11]{lurie_higher_topos_theory}. Passing to the relevant subgroupoids of towers of potential $n$-stages we get the needed statement. 
\end{proof}

\subsection{Certain cartesian squares}

In this section, we relate the moduli of potential $n$-stages $\mathcal{M}_{n}(A)$ of \cref{defin:potential_n_stages_and_their_moduli} for varying $n$ by fitting them in certain cartesian squares. As a consequence, we will give a complete algebraic description of the fibres of $\mathcal{M}_{n+1}(A) \rightarrow \mathcal{M}_{n}(A)$; this will finish the proof of \cref{thm:main_theorem}.

The map $\mathcal{M}_{n}(A) \rightarrow \mathcal{M}_{n-1}(A)$ is given by $(n-1)$-truncation. Thus, if $Y \in \mathcal{M}_{n-1}(A)$ is a potential $(n-1)$-stage, to understand the fibre over $Y$ is to understand in how many ways one can extend $Y$ to a potential $n$-stage. 

\begin{prop}
\label{prop:fibre_of_map_from_potential_n_stage_into_potential_n_minus_1_stage}
Let $X \in \mathcal{M}_{n}(A)$ be a potential $n$-stage, $Y \in \mathcal{M}_{n-1}(A)$ a potential $(n-1)$-stage, and suppose there exists an equivalence $ X_{\leq (n-1)} \simeq Y$. Then, $X$ is of type $Y +_{A} (A_{S^{n}}, n)$ in the sense of \cref{defin:spherical_presheaf_of_type_as_lambda_module}. 
\end{prop}

\begin{proof}
The homotopy groups of $X$ and $Y$ are given explicitly by \cref{observation:homotopy_groups_of_potential_n_stages} and we see that the relative homotopy groups vanish outside of $\pi_{n}(X \rightarrow Y) \simeq A_{S^{n}}$, giving the first part. 
\end{proof}

\begin{cor}
In the situation of \cref{prop:fibre_of_map_from_potential_n_stage_into_potential_n_minus_1_stage}, $X$ is classified up to equivalence by a homotopy class of $0$-equivalences
\[
Y \rightarrow \B_{A}(A_{S^{n}}, n+1), 
\]
into the Eilenberg-MacLane presheaf $\B_{A}(A_{S^{n}}, n+1) := \B^{\spheres}_{A}(A_{S^{n}}, n+1)$  of \cref{defin:eilenberg_maclane_presheaf}, well-defined up the action of $\Aut_{\presheaves_{\Sigma}(\spheres)}(Y) \times \Aut(A, A_{S^{n}})$. 
\end{cor}

\begin{proof}
This is a restatement of \cref{cor:postnikov_invariant_of_a spherical_presheaf_in_terms_of_homotopy_classes}. 
\end{proof}

\begin{warning}
If $Y$ is a potential $(n-1)$-stage, then not all presheaves of type $Y +_{A} (A_{S^{n}}, n)$ are potential $n$-stages. We give a criterion for when this happens in \cref{theorem:conditions_on_a_classifying_map_of_an_n_stage}.
\end{warning}

Recall that by \cref{cor:equivalence_between_em_and_homotopy_em} we have an equivalence
\begin{equation}
\label{equation:equivalence_of_em_presheaf_in_presheaves_and_homotopy_presheaves}
\B_{A}(A_{S^{n}}, n+1) \simeq p_{*} \B^{h}_{A}(A_{S^{n}}, n+1),
\end{equation}
where $\B^{h}_{\Lambda}(M, n+1) := \B^{h \spheres}_{A}(A_{S^{n}}, n+1)$ is the Eilenberg-MacLane presheaf internal to the $\infty$-category of spherical homotopy presheaves and
\[
p^{*} \colon \presheaves_{\Sigma}(\spheres) \rightleftarrows \presheaves_{\Sigma}(h\spheres) :p_{*}
\]
is the restriction, left Kan extension adjunction of \cref{prop:adjunction_between_spherical_preshaves_and_homotopy_spherical_presheaves}. It follows from (\ref{equation:equivalence_of_em_presheaf_in_presheaves_and_homotopy_presheaves}) that any map of spherical presheaves
\[
\psi \colon Y \rightarrow \B_{A}(A_{S^{n}}, n+1) \simeq p_{*} \B^{h}_{A}(A_{S^{n}}, n+1)
\]
has an adjoint 
\[
\psi^{adj} \colon p^{*}Y \rightarrow \B^{h}_{A}(A_{S^{n}}, n+1)
\]
which is a morphism of spherical homotopy presheaves. 

\begin{theorem}
\label{theorem:conditions_on_a_classifying_map_of_an_n_stage}
Let $Y$ be a potential $(n-1)$-stage for a $\Pi$-algebra $A$ and assume that $X$ is a spherical presheaf of type $Y +_{A} (A_{S^{n}}, n)$.  Then $X$ is a potential $n$-stage if and only if the adjoint $\psi^{adj} \colon p^{*} Y \rightarrow \B^{h}_{A}(A_{S^{n}}, n+1)$ of its classifying map $\psi$ is an equivalence. 
\end{theorem}

\begin{proof}
Since $\pi_{0}Y \simeq A$, any $X$ of type $Y +_{A} (A_{S^{n}}, n)$ is an $n$-type whose underlying $\Pi$-algebra is isomorphic to $A$. Thus, such an $X$ is a potential $n$-stage if and only if the loop comparison map $X_{S^{1}} \rightarrow \Omega X$ is an $(n-1)$-equivalence. 

Recall from \cref{prop:universal_family_over_moduli_of_n_types} that if $\mathrm{E}_{A}(A_{S^{n}}, n+1) \rightarrow  \B_{A}(A_{S^{n}}, n+1)$ denotes the universal morphism of spherical presheaves levelwise fibred in Eilenberg-MacLane spaces of degree $n$, the composite 
\[
\mathrm{E}_{A}(A_{S^{n}}, n+1) \rightarrow  \B_{A}(A_{S^{n}}, n+1) \rightarrow A 
\]
is an equivalence. The same applies to spherical homotopy presheaves, and in what follows below, we will identify the universal family with $A$. 

Let $\psi \colon Y \rightarrow \B_{A}(A_{S^{n}}, n+1)$ be the classifying map, which using the equivalence (\ref{equation:equivalence_of_em_presheaf_in_presheaves_and_homotopy_presheaves}) can be factored as 
\[
Y \rightarrow p_{*} p^{*} Y \rightarrow  p_{*} \B^{h}_{A}(A_{S^{n}}, n+1),
\]
where the first map is the unit of the adjunction and the second one is $p_{*}(-)$ applied to the adjoint $\psi^{adj}$. We want to show that the latter is an equivalence if and only if $X \in \mathcal{M}_{n}(A)$. 

Let us denote by $Z$ the presheaf over $p^{*} Y$ classified by the adjoint, so that we have a pullback diagram 
\[
\begin{tikzcd}
	Z & A \\
	p^{*} Y & \B^{h}_{A}(A_{S^{n}}, n+1)
	\arrow[from=1-1, to=1-2]
	\arrow[from=1-2, to=2-2]
	\arrow["\psi^{adj}", from=2-1, to=2-2]
	\arrow[from=1-1, to=2-1]
\end{tikzcd}
\]
of spherical homotopy presheaves, where as mentioned above we identify $A$ with the universal family. Applying $p_{*}$ to all vertices, we obtain a diagram 
\begin{equation}
\label{equation:crucial_two_squares_in_conditions_for_pullback_to_be_an_n_stage}
\begin{tikzpicture}[baseline=(current  bounding  box.center)]
		\node (TL) at (0, 1.5) {$ X $};
		\node (BL) at (0, 0) {$ Y $};
		\node (TM) at (2.5, 1.5) {$ p_{*} Z $}; 
		\node (BM) at (2.5, 0) {$ p_{*} p^{*} Y $};
		\node (TR) at (5, 1.5) {$ A $};
		\node (BR) at (5, 0) {$ \B_{A}(A_{S^{n}}, n+1) $};
		
		\draw [->] (TL) to (TM);
		\draw [->] (TM) to (TR);
		\draw [->] (BL) to (BM);
		\draw [->] (BM) to (BR);
		
		\draw [->] (TL) to (BL);
		\draw [->] (TM) to (BM);
		\draw [->] (TR) to (BR);
\end{tikzpicture}
\end{equation}
of spherical presheaves, where we implicitly identify $p_{*}(A) \simeq A$, as $p_{*}$ induces an equivalence between the subcategories of discrete objects. The right square is cartesian, because $p_{*}$ preserves all limits, and so is the outer square, since $X$ is by definition the pullback of the universal family. We conclude from the pasting lemma that the left square is also cartesian. 

By construction, all of the maps in (\ref{equation:crucial_two_squares_in_conditions_for_pullback_to_be_an_n_stage}) are $0$-equivalences. Since pullbacks along $0$-equivalences reflect equivalences by \cite[6.2.3.16]{lurie_higher_topos_theory}, and so does $p_{*}(-)$, we see that $\psi^{adj}$ is an equivalence if and only if $p_{*} Z \rightarrow A$ is; in other words, if and only if $p_{*} Z$ is discrete. 

We now consider the diagram 

\begin{center}
	\begin{tikzpicture}
		\node (TL) at (0, 1.5) {$ \B X_{S^{1}} $};
		\node (BL) at (0, 0) {$ \B Y_{S^{1}} $};
		\node (TM) at (1.5, 1.5) {$ X $}; 
		\node (BM) at (1.5, 0) {$ Y $};
		\node (TR) at (3, 1.5) {$ p_{*} Z $}; 
		\node (BR) at (3, 0) {$ p_{*} p^{*} Y $};
		
		\draw [->] (TL) to (TM);
		\draw [->] (TM) to (TR);
		\draw [->] (BL) to (BM);
		\draw [->] (BM) to (BR);
		
		\draw [->] (TL) to (BL);
		\draw [->] (TM) to (BM);
		\draw [->] (TR) to (BR);
	\end{tikzpicture},
\end{center}
where the square on the right is the left square of (\ref{equation:crucial_two_squares_in_conditions_for_pullback_to_be_an_n_stage}) and the square on the left is given by adjoints to the loop comparison maps. The bottom two maps can be promoted to a fibre sequence by \cref{thm:spiral_fibre_sequence} and since the right square is cartesian, we conclude that the fibre of the map $X \rightarrow p_{*} Z$ can be identified with $\B Y_{S^{1}}$. Since $X \rightarrow Y$ is an $(n-1)$-truncation, $\B X_{S^{1}} \rightarrow \B Y_{S^{1}}$ is an $n$-truncation and thus 
\begin{equation}
\label{equation:n_truncation_is_fibre_in_n_potentiality_criterion}
\mathrm{fib}(X \rightarrow p_{*}Z) \simeq (\B X_{S^{1}})_{\leq n}.
\end{equation}

Using that $\pi_{k} ((\B X_{S^{1}})_{\leq n}) \simeq (\pi_{k-1} X)_{S^{1}}$ for $1 \leq k \leq n$, the identification (\ref{equation:n_truncation_is_fibre_in_n_potentiality_criterion}) yields a long exact sequence of homotopy 
\[
 0 \rightarrow \pi_{n+1} (p_{*} Z) \rightarrow (\pi_{n-1}X)_{S^{1}} \rightarrow \pi_{n} X \rightarrow \pi_{n} (p_{*} Z) \rightarrow (\pi_{n-2}X)_{S^{1}} \rightarrow \ldots 
\]
ending in 
\[
\ldots \rightarrow \pi_{2} (p_{*} Z) \rightarrow (\pi_{0} X)_{S^{1}} \rightarrow \pi_{1} X \rightarrow \pi_{1} (p_{*} Z) \rightarrow 0,
\]
where we omit the isomorphism $\pi_{0} X \simeq \pi_{0} (p_{*} Z)$. We deduce that the loop comparison maps $(\pi_{k-1} X)_{S^{1}} \rightarrow \pi_{k} X$ are isomorphisms for $1 \leq k \leq n$; that is, $X$ is a potential $n$-stage, if and only if $\pi_{k} (p_{*}Z) = 0$ for $k > 0$. By \cref{prop:spherical_presheaves_take_values_in_h_spaces}, the latter condition is equivalent to $p_{*}Z$ being discrete, ending the argument.   
\end{proof}

\begin{cor}
\label{cor:condition_for_an_n_stage_to_be_extendable}
Let $Y$ be a potential $(n-1)$-stage. Then $Y$ can be extended to a potential $n$-stage if and only if $p^{*} Y$ is equivalent, as a spherical homotopy presheaf, to $\B^{h}_{A}(A_{S^{n}}, n+1)$. 
\end{cor}

\begin{proof}
This is immediate from \cref{theorem:conditions_on_a_classifying_map_of_an_n_stage}.
\end{proof}

According to the above, whether a given potential $(n-1)$-stage $Y$ can be extended to a potential $n$-stage is controlled by the spherical homotopy presheaf $p^{*} Y$. In general, the latter need not be equivalent to $\B_{A}^{h}(A_{S^{n}}, n+1)$; however, the following shows that it has the same homotopy groups:

\begin{lemma}
\label{lem:homotopy_type_of_pushforward_of_an_n_stage}
Let $Y$ be a potential $(n-1)$-stage for a $\Pi$-algebra $A$. Then $p^{*}Y$ is a spherical homotopy presheaf of type $A +_{A} (A_{S^{n}}, n+1)$. 
\end{lemma}

\begin{proof}
We consider the spiral fibre sequence
\[
\B Y_{S^{1}} \rightarrow Y \rightarrow p_{*} p^{*} Y
\]
of \cref{thm:spiral_fibre_sequence}, where the first map induces an isomorphism
\[
(\pi_{i-1}Y)_{S^{1}} \simeq \pi_{i}(\B Y_{S^{1}}) \rightarrow \pi_{i} Y
\]
for $1 \leq i \leq n-1$ by assumption that $Y$ is a potential $(n-1)$-stage. Examining the associated long exact sequence in homotopy (that is, the spiral long exact sequence of \cref{defin:spiral_long_exact_sequence}) we see that the homotopy groups of $p_{*} p^{*}Y$ vanish outside of 
\[
\pi_{0}(p_{*} p^{*}Y) \simeq \pi_{0}(Y) \simeq A
\]
and
\[
\pi_{n+1}(p_{*} p^{*}Y) \simeq \pi_{n} (\B Y_{S^{1}}) \simeq (\pi_{n-1}Y) _{S^{1}} \simeq (A_{S^{n-1}})_{S^{1}} \simeq A_{S^{n}}. 
\]
As $p_{*}$ is given by precomposition, it preserves homotopy groups, and we deduce the same is true for homotopy groups of $p^{*} Y$. 
\end{proof}

\begin{rem}
If $X$ is a spherical homotopy presheaf of type $A +_{A} (A_{S^{n}}, n+1)$, then projection onto path components gives a fibre sequence of equivalent to one of the form 
\[
\B^{n+1} A_{S^{n}} \rightarrow X \rightarrow A,
\]
where we use that $A_{S^{n}}$ has an abelian group structure (in fact, can be promoted to an $A$-module as in \cref{example:module_structure_on_shifts}) to make sense of the iterated bar construction. Thus, $X$ can informally be thought of as an extension of $A$ by the presheaf $\B^{n+1}A_{S^{n}}$, with prescribed action of $A$ on $\pi_{n+1}X$. 

Note that by \cref{prop:description_of_the_eilenberg_maclane_presheaf}, up to equivalence there is a unique presheaf of type $A + _{A} (A_{S^{n}}, n+1)$ which admits a section, namely $\B^{h}_{A}(A_{S^{n}}, n+1)$. Thus, \cref{cor:condition_for_an_n_stage_to_be_extendable} can be phrased as saying that a potential $(n-1)$-stage $Y$ can be extended to a potential $n$-stage if and only if the extension $\B^{n+1}A_{S^{n}} \rightarrow p^{*}Y \rightarrow A$ is ``split'' in the sense that it admits a section $A \rightarrow p^{*}Y$. 
\end{rem}
Our main result, \cref{thm:cartesian_square_of_moduli} below, relates the moudli of potential $n$-stages for varying $n$ by fitting them into a certain cartesian square. The definition of the relevant square is somewhat involved, and so we first describe it separately: 

\begin{construction}
\label{construction:main_theorem_square}
For each $n \geq 1$, we will define a square of $\infty$-groupoids of the form
\[
\begin{tikzpicture}
		\node (TL) at (0, 1.5) {$ \mathcal{M}_{n}(A) $};
		\node (BL) at (0, 0) {$ \mathcal{M}_{n-1}(A) $};
		\node (TR) at (3.3, 1.5) {$ \B\Aut(A, A_{S^{n}}) $};
		\node (BR) at (3.3, 0) {$  \mathcal{M}^{h}(A +_{A} (A_{S^{n}}, n+1)) $};
		
		\draw [->] (TL) to (TR);
		\draw [->] (TL) to (BL);
		\draw [->] (TR) to (BR);
		\draw [->] (BL) to (BR);
	\end{tikzpicture}. 
\]
We will show that the square canonically commutes in \cref{lemma:main_square_involving_moduli_of_potential_n_stages_commutes} below. The vertices are as follows: 
\begin{enumerate}
\item $\mathcal{M}_{n}(A)$ and $\mathcal{M}_{n-1}(A)$ are the moduli of realizations of \cref{defin:realization_and_moduli_of_realizations},
\item $\mathcal{M}^{h}(A +_{A} (A_{S^{n}}, n+1))$ is the full subgroupoid of $\presheaves_{\Sigma}(h \spheres)^{\cong}$ spanned by spherical homotopy presheaves of type $A +_{A} (A_{S^{n}}, n+1)$,
\item $\B \Aut(A, A_{S^{n}})$ is the classifying space for the group of automorphism of \cref{defin:automorphism_of_pair_of_c_algebra_and_module}, which we identify with the full subgroupoid of $\presheaves_{\Sigma}(h \spheres)_{/\mathcal{M}_{h \spheres}(\abeliangroups)}$ spanned by maps equivalent to $A \rightarrow \mathcal{M}_{h \spheres}(\abeliangroups)$, the map classifying the $A$-module $A_{S^{n}}$.
\end{enumerate}
The maps in the square are defined as follows: 
\begin{enumerate} 
\item the left vertical map is the truncation $(-)_{\leq n-1} \colon \mathcal{M}_{n}(A) \rightarrow \mathcal{M}_{n-1}(A)$, 
\item the upper horizontal map sends $X \in \mathcal{M}_{n}$ to the map $\pi_{0}X \rightarrow \mathcal{M}_{\spheres}(\abeliangroups)$ classifying the $\pi_{0}X$-module $\pi_{n}X$,
\item the bottom horizontal map sends $Y \in \mathcal{M}_{n-1}(Y)$ to  $p^{*}Y$, which is of type $A +_{A} (A_{S^{n}}, n+1)$ by \cref{lem:homotopy_type_of_pushforward_of_an_n_stage}, 
\item the right vertical map sends $\Lambda \rightarrow \mathcal{M}_{h \spheres}(\abeliangroups)$ to the pullback $\Lambda \times_{\mathcal{M}_{h \spheres}(\abeliangroups)} \mathcal{M}_{h \spheres}(\emspaces_{n})$. 
\end{enumerate}
In the case of the last map, note that if $\Lambda \rightarrow \mathcal{M}_{h \spheres}(\abeliangroups)$ classifies a $\Lambda$-module $M$, then
\[
\Lambda \times_{\mathcal{M}_{h \spheres}(\abeliangroups)} \mathcal{M}_{h \spheres}(\emspaces_{n}) \simeq B^{h}_{\Lambda}(M, n+1) 
\]
by definition. The latter Eilenberg-MacLane spherical homotopy presheaf is of type $A +_{A} (A_{S^{n}}, n+1)$ since $(\Lambda, M) \simeq (A, A_{S^{n}})$. 
\end{construction}

\begin{lemma}
\label{lemma:main_square_involving_moduli_of_potential_n_stages_commutes}
The square of $\infty$-groupoids of \cref{construction:main_theorem_square} canonically commutes. 
\end{lemma}

\begin{proof}
Unwrapping \cref{construction:main_theorem_square}, we see that we have to show the following: if $X \in \mathcal{M}_{n}(A)$ is a potential $n$-stage, then there is canonical equivalence 
\[
p_{*} (X_{\leq n-1}) \simeq B^{h}_{\pi_{0}X}(\pi_{n}X, n+1)
\]
in $\presheaves_{\Sigma}(h \spheres)$. There is a canonical map, namely the adjoint of the map 
\[
X_{\leq n-1} \rightarrow B_{\pi_{0}(X)}(\pi_{n}X, n+1) \simeq p_{*} B^{h}_{\pi_{0}(X)}(\pi_{n}X, n+1)
\]
classifying $X \rightarrow X_{\leq n-1}$, and by \cref{theorem:conditions_on_a_classifying_map_of_an_n_stage} this adjoint is an equivalence. 
\end{proof}

\begin{theorem}
\label{thm:cartesian_square_of_moduli}
The square 
\[
	\begin{tikzpicture}
		\node (TL) at (0, 1.5) {$ \mathcal{M}_{n}(A) $};
		\node (BL) at (0, 0) {$ \mathcal{M}_{n-1}(A) $};
		\node (TR) at (3.3, 1.5) {$ \B\Aut(A, A_{S^{n}}) $};
		\node (BR) at (3.3, 0) {$  \mathcal{M}^{h}(A +_{A} (A_{S^{n}}, n+1)) $};
		
		\draw [->] (TL) to (TR);
		\draw [->] (TL) to (BL);
		\draw [->] (TR) to (BR);
		\draw [->] (BL) to (BR);
	\end{tikzpicture}
 \]
of \cref{construction:main_theorem_square} is a cartesian square of $\infty$-groupoids. 
\end{theorem}

\begin{proof}
We have to show that for any $Y \in \mathcal{M}_{n-1}(A)$, the induced map 
\begin{equation}
\label{equation:map_between_fibres_in_main_result}
\{ Y \} \times_{\mathcal{M}_{n-1}(Y)} \mathcal{M}_{n}(A) \rightarrow \{ p^{*}Y \} \times_{\mathcal{M}^{h}(A +_{A} (A_{S^{n}}, n+1))}  \B\Aut(A, A_{S^{n}})
\end{equation}
between fibres is an equivalence. There are two possible cases: 
\begin{enumerate}
\item the homotopy presheaves $p^{*} Y$ and $\B^{h}_{A}(A_{S^{n}}, n+1)$ are \emph{equivalent}, 
\item the homotopy presheaves $p^{*} Y$ and $\B^{h}_{A}(A_{S^{n}}, n+1)$ are \emph{not equivalent}. 
\end{enumerate}
In the second case, both fibres are empty by \cref{theorem:conditions_on_a_classifying_map_of_an_n_stage} and there is nothing to show. Thus, we move to the first case. 

The target of the comparison map (\ref{equation:map_between_fibres_in_main_result}) can be identified with the $\infty$-category of triples 
\[
(\Lambda, M, \phi \colon p^{*} Y \simeq B^{h}_{\Lambda}(M, n+1))
\]
of a $\Pi$-algebra $\Lambda$, a module $M$, and an equivalence $\phi$ of homotopy presheaves. By \cref{theorem:conditions_on_a_classifying_map_of_an_n_stage}, a choice of $\phi$ is equivalent to giving an adjoint 
\[
\psi \colon Y \rightarrow \B_{\Lambda}(M, n+1) \simeq p_{*} \B^{h}_{\Lambda}(M, n+1)
\]
classifying a potential $n$-stage. Note that $\psi$ is necessarily a $0$-equivalence, as it can be written as the composite 
\[
Y \rightarrow p_{*} p^{*} Y \rightarrow p_{*} \B^{h}_{\Lambda}(M, n+1),
\]
where the first arrow is the unit which is an $0$-equivalence by \cref{thm:spiral_fibre_sequence} and the second one is $p_{*}(-)$ applied to the equivalence $\phi$. Using the universal property of the Eilenberg-MacLane presheaf of \cref{observation:universal_property_of_em_presheaf}, we see that a choice of $\psi$ is equivalent to that of a triple 
\[
(X \rightarrow Y, \pi_{0}X \simeq \Lambda, \pi_{n}X \simeq M),
\]
where the first term is a potential $n$-stage together with a map $X \rightarrow Y$ inducing an equivalence $X_{\leq n-1} \simeq Y$, the second is an isomorphism of $\Pi$-algebras, and the third one an isomorphism of modules. 

Combining these two triple descriptions together we see that the target of 
(\ref{equation:map_between_fibres_in_main_result}) can be identified with the $\infty$-category of quintuples 
\[
(X \rightarrow Y, \Lambda, M, \pi_{0}X \simeq \Lambda, \pi_{n}X \simeq M)
\]
Fixing the first term, the space of the last four is clearly contractible as it consists of a choice of $(\Lambda, M)$ together with an isomorphism to $(\pi_{0}X, \pi_{n}X)$. Thus, the fibre is equivalent to the space of $X \rightarrow Y$ as above; ie. with $\{ Y \} \times_{\mathcal{M}_{n-1}(A)} \mathcal{M}_{n}(A)$. This ends the argument. 
\end{proof}

\begin{rem}[Obstructions to a realization]
\label{rem:obstructions_to_realization}
We describe an interpretation of \cref{thm:cartesian_square_of_moduli} as an obstruction theory for constructing realizations, where the obstructions take values in Andr\'{e}-Quillen cohomology of $\Pi$-algebras. This is explained in detail in the original paper of Blanc-Dwyer-Goerss \cite{realization_space_of_a_pi_algebra}, but in the interest of being self-contained, we include it here. 

The equivalence $\mathcal{M}_{\infty}(A) \simeq \varprojlim \mathcal{M}_{n}(A)$ gives an inductive way of producing realizations of a given $\Pi$-algebra $A$. To do so, one begins with a point in $\mathcal{M}_{0}(A) \simeq \B\Aut(A)$, which is connected and so there is a unique choice up to (a non-canonical) equivalence, and tries to lift it up the tower of $\infty$-groupoids
\[
\mathcal{M}(A) \simeq \mathcal{M}_{\infty}(A) \rightarrow \ldots \rightarrow \mathcal{M}_{1}(A) \rightarrow \mathcal{M}_{0}(A)
\]

The cartesian square of \cref{thm:cartesian_square_of_moduli} implies that a point in $\mathcal{M}_{n-1}(A)$ can be lifted to $\mathcal{M}_{n}(A)$ if and only if it lands in the correct component of $\mathcal{M}^{h}(A +_{A} (A_{S^{n}}, n+1))$; namely, the component containing the image of the connected space $\B \Aut(A, A_{S^{n}})$. Thus, there is an \emph{obstruction} to performing the lift, which is an element 
\begin{equation}
\label{equation:path_components_where_obstruction_lives}
\theta_{n} \in \pi_{0} \mathcal{M}^{h}(A +_{A} (A_{S^{n}}, n+1))
\end{equation}
of the set of path components. 

We can use \cref{thm:classification_of_n_type_presheaves_fixed_module} applied to $\inftyc = \spheres$ to give an explicit description of (\ref{equation:path_components_where_obstruction_lives}). To see this, note that we have a fibre sequence 
\[
\map_{0-\mathrm{eq}}(A, \B^{h}_{A}(A_{S^{n}}, n+2)) \rightarrow \mathcal{M}^{h}(A +_{A} (A_{S^{n}}, n+1)) \rightarrow \B\Aut_{\Pi\mhyphen\Alg}(A) \times \B\Aut(A, A_{S^{n}}),
\]
where $\map_{0-\mathrm{eq}}$ denotes the subspace of maps which are $0$-equivalences; that is, those that induce an isomorphism $A \simeq \pi_{0} \B^{h}_{A}(A_{S^{n}}, n+2)$. As $\Aut_{\Pi\mhyphen\Alg}(A)$ acts freely and transitively on the set of such isomorphisms, we can instead look at the space of maps such that the composite 
\[
A \rightarrow \B^{h}_{A}(A_{S^{n}}, n+2) \rightarrow A,
\]
where the second map is the canonical projection, is the identity. This yields a fibre sequence 
\begin{equation}
\label{equation:fibre_in_identification_of_obstructions}
\map_{\presheaves_{\Sigma}(h\spheres)_{/A}}(A, \B^{h}_{A}(A_{S^{n}}, n+2)) \rightarrow \mathcal{M}^{h}(A +_{A} (A_{S^{n}}, n+1)) \rightarrow \B\Aut(A, A_{S^{n}}),
\end{equation}
which using \cref{observation:homotopy_quotient_by_an_infty_group} gives an identification 
\[
\pi_{0} \mathcal{M}^{h}(A +_{A} (A_{S^{n}}, n+1)) \simeq \faktor{\pi_{0} \map_{\presheaves_{\Sigma}(h\spheres)_{/A}}(A, \B^{h}_{A}(A_{S^{n}}, n+2))}{\Aut(A, A_{S^{n}})}. 
\]
with the quotient of the set of path components. The homotopy groups of this mapping space can be described as the Andr\'{e}-Quillen cohomology groups of $\Pi$-algebras of \cref{rem:andre_quillen_cohomology_of_c_algebras_as_maps_into_em_presheaf}; in detail, we have
\[
\pi_{i}(\map_{\presheaves_{\Sigma}(h\spheres)_{/A}}(A, \B^{h}_{A}(A_{S^{n}}, n+2))) \simeq \mathrm{H}^{n+2-i}_{A}(A, A_{S^{n}}).
\]
We conclude that the obstruction $\theta_{n}$ of (\ref{equation:path_components_where_obstruction_lives}) lies in the quotient set
\[
\pi_{0} \mathcal{M}^{h}(A +_{A} (A_{S^{n}}, n+1)) \simeq \faktor{\mathrm{H}^{n+2}_{A}(A, A_{S^{n}})}{\Aut(A, A_{S^{n}})}. 
\]
In other words, $\theta_{n}$ can be identified with a cohomology class, well-defined up to the action of the automorphism group of the pair $(A, A_{S^{n}})$. 
\end{rem}

\begin{rem}[Generalizing to other kinds of algebras]
\label{rem:other_indexing_infty_categories}
A diligent reader might notice that the only result in this chapter that explicitly depended on the indexing $\infty$-category being $\spheres$ was \cref{prop:characterization_of_spacelike_spherical_presheaves}, which gave us the the identification $\spaces_{*}^{\geq 1} \hookrightarrow \presheaves_{\Sigma}(\spheres)$ of the $\infty$-category of pointed, connected spaces with the $\infty$-category of spherical presheaves whose loop comparison map is an equivalence. 

In particular, for any indexing $\infty$-category $\inftyc$ satisfying \cref{assumption:assumption_on_c_to_have_good_presheaves} and \cref{assumption:inftyc_admits_tensors_by_the_circle} and a $\inftyc$-algebra $A$ one can define spaces $\mathcal{M}_{n}(A)$ as in \cref{defin:potential_n_stages_and_their_moduli}. These spaces then fit into cartesian squares of \cref{thm:cartesian_square_of_moduli} and thus one has obstructions for lifting points up the tower living in Andr\'{e}-Quillen cohomology of $\inftyc$-algebras as in \cref{rem:obstructions_to_realization}. However, to make this into a useful theory one needs some analogue of \cref{prop:characterization_of_spacelike_spherical_presheaves} to identify $\mathcal{M}_{\infty}(A)$ with objects of interest. 

We refer an interested reader to the work of Balderrama \cite{balderrama2021deformations}, which greatly extends these ideas to a variety of interesting contexts. 
\end{rem}

\bibliographystyle{amsalpha}
\bibliography{moduli_of_pi_algebras_bibliography}

\end{document}